\documentclass[11pt]{article}
\usepackage{amsfonts}
\usepackage{amssymb}
\usepackage{amsthm}
\usepackage{amsmath}
\usepackage{graphicx}
\usepackage{empheq}
\usepackage{indentfirst}
\usepackage{cite} % 使参考文献引用时自动排序
\usepackage{mathrsfs}
\usepackage{cases}
\usepackage{graphics}
\usepackage{xcolor}  % 更改文字颜色
\usepackage{bm}
\usepackage{makeidx}
\usepackage[T1]{fontenc}
\usepackage{times}  %正文使用罗马字体
\newtheorem{theorem}{\textbf{Theorem}}[section]
\newtheorem{lemma}{\textbf{Lemma}}[section]
\newtheorem{proposition}{\textbf{Proposition}}[section]
\newtheorem{corollary}{\textbf{Corollary}}[section]
\newtheorem{remark}{\textbf{Remark}}[section]
\newtheorem{definition}{\textbf{Definition}}[section]
\allowdisplaybreaks[4] % 允许公式跨页显示
\def\be{\begin{equation}}
\def\ee{\end{equation}}
\def\bea{\begin{eqnarray}}
\def\eea{\end{eqnarray}}
\def\bt{\begin{theorem}}
\def\et{\end{theorem}}
\def\bl{\begin{lemma}}
\def\el{\end{lemma}}
\def\br{\begin{remark}}
\def\er{\end{remark}}
\def\bp{\begin{proposition}}
\def\ep{\end{proposition}}
\def\bc{\begin{corollary}}
\def\ec{\end{corollary}}
\def\bd{\begin{definition}}
\def\ed{\end{definition}}
\def\vp{\varphi}

\begin{document}

\title{On a Navier--Stokes--Cahn--Hilliard System for Viscous Incompressible Two-phase Flows with Chemotaxis, Active Transport and Reaction}

\author{
	Jingning He
	\footnote{School of Mathematics, Hangzhou Normal University, Yuhang District, Hangzhou 311121, Zhejiang Province, P.R. China. Email: \texttt{jingninghe2020@gmail.com}
	},
\ \ Hao Wu\footnote{Corresponding author. School of Mathematical Sciences, Fudan University, Handan Road 220, Shanghai 200433, P.R. China.
		Email: \texttt{haowufd@fudan.edu.cn}
	}
}
\date{\today}
\maketitle

%%%%%%%%%%%%%%%%%%%%%%%%%%%%%%%%%%%%%%%%%%%%%%%%%%%%%%%%%%%%%%%%%%%%%%%%%%%%%%%%%%%%%%%%%%%%%%%%%%%%%%%%%%%%%%%%%%%%%%%%%%%

\begin{abstract}
\noindent We analyze a Navier--Stokes--Cahn--Hilliard model for viscous incompressible two-phase flows where the mechanisms of chemotaxis, active transport and reaction are taken into account. The evolution system couples the Navier--Stokes equations for the volume-averaged fluid velocity, a convective Cahn--Hilliard equation for the phase-field variable, and an advection-diffusion equation for the density of certain chemical substance. This system is thermodynamically consistent and generalizes the well-known ``Model H'' for viscous incompressible binary fluids. For the initial-boundary value problem with a physically relevant singular potential in a general bounded smooth domain $\Omega\subset \mathbb{R}^3$, we first prove the existence and uniqueness of a local strong solution. When the initial velocity is small and the initial phase-field function as well as the initial chemical density are small perturbations of a local minimizer of the free energy, we establish the existence of a unique global strong solution. Afterwards, we show the uniqueness of asymptotic limit for any global strong solution as time goes to infinity and provide an estimate on the convergence rate. The proofs for global well-posedness and long-time behavior are based on the  dissipative structure of the system and the {\L}ojasiewicz--Simon approach. Our analysis reveals the effects of chemotaxis, active transport and a long-range interaction of Oono's type on the global dynamics of the coupled system.
\medskip \\
\noindent%%%%%
\end{abstract}
%% 35A01 Existence problems for PDEs: global existence, local existence, non-existence
%% 35A02 Uniqueness problems for PDEs: global uniqueness, local uniqueness, non-uniqueness
%% 35K35 Initial-boundary value problems for higher-order parabolic equations
%% 35Q92 PDEs in connection with biology, chemistry and other natural sciences
%% 76D05 Navier-Stokes equations for incompressible viscous fluids

\tableofcontents

%%%%%%%%%%%%%%%%%%%%%%%%%%%%%%%%%%%%%%%%%%%%%%%%%%%%%%%%%%%%%%%%%%%%%%%%%%%%
\section{Introduction}
\setcounter{equation}{0}
\noindent
Let $\Omega \subset\mathbb{R}^3$ be a bounded domain with a smooth boundary $\partial\Omega$ and $T>0$. We consider the following  Navier--Stokes--Cahn--Hilliard type system
\begin{subequations}
	\begin{alignat}{3}
	&\partial_t  \bm{ v}+\bm{ v} \cdot \nabla  \bm {v}-\textrm{div} \big(2  \nu(\varphi) D\bm{v} \big)+\nabla p=(\mu+\chi \sigma)\nabla \varphi,\label{f3.c} \\
	&\textrm{div}\ \bm{v}=0,\label{f3.c1}\\
	&\partial_t \varphi+\bm{v} \cdot \nabla \varphi=\Delta \mu-\alpha(\overline{\varphi}-c_0),\label{f1.a} \\
	&\mu=-\varepsilon \Delta \varphi+ \frac{1}{\varepsilon} \varPsi'(\varphi)-\chi \sigma+\beta\mathcal{N}(\varphi-\overline{\varphi}),\label{f4.d}\\
	&\partial_t \sigma+\bm{v} \cdot \nabla \sigma= \Delta (\sigma-\chi\varphi), \label{f2.b}
	\end{alignat}
\end{subequations}
in $\Omega\times(0,T)$. The system \eqref{f3.c}--\eqref{f2.b} is subject to the classical boundary conditions
\begin{alignat}{3}
&\bm{v}=\mathbf{0},\quad {\partial}_{\bm{n}}\varphi={\partial}_{\bm{n}}\mu={\partial}_{\bm{n}}\sigma=0,\qquad\qquad &\textrm{on}& \   \partial\Omega\times(0,T),
\label{boundary}
\end{alignat}
and the initial conditions
\begin{alignat}{3}
&\bm{v}|_{t=0}=\bm{v}_{0}(x),\quad \varphi|_{t=0}=\varphi_{0}(x), \quad  \sigma|_{t=0}=\sigma_{0}(x), \qquad &\textrm{in}&\ \Omega.
\label{ini0}
\end{alignat}
Here, $\bm{n}$ is the unit outward normal vector on $\partial \Omega$, and $\partial_{\bm{n}}$ denotes the outer normal derivative on $\partial \Omega$.

The flow of a two-phase (or multi-phase) fluid mixture has attracted considerable attention in many application fields of science and engineering. The complicated interplay between the motion of moving and deforming free interfaces and the dynamics of binary fluids can be effectively described by the so-called diffuse interface models \cite{A2012,AMW,EG19jde,Gur,GLSS,HH,LS,LT98}. The system \eqref{f3.c}--\eqref{f2.b} under investigation is a simplified variant of the general thermodynamically consistent diffuse interface model that was derived in \cite{LW} for a mixture of two viscous incompressible fluids with a chemical species subject to the diffusion process as well as some important transport mechanism like chemotaxis (see also \cite{Sitka} and the references cited therein).
The fluid velocity $\bm{v}: \Omega \times(0,T)\to \mathbb{R}^3$ is taken as the volume-averaged velocity with $D\bm{v}=\frac{1}{2}(\nabla\bm{ v}+(\nabla\bm{ v}) ^ \mathrm{T})$ being the symmetrized velocity gradient, and the scalar function $p:\Omega \times(0,T)\to \mathbb{R}$ denotes the (modified) pressure of the mixture. The order parameter $\varphi:\Omega \times(0,T)\to \mathbb{R}$ with $\varphi\in [-1,1]$ denotes the difference in volume fractions of the binary fluid mixture such that the region $\{\varphi=1\}$ represents the  fluid 1 and $\{\varphi=-1\}$ represents the fluid 2 (i.e., the values $\pm 1$ represent the pure phases). The variable $\sigma: \Omega \times(0,T)\to \mathbb{R}$ denotes the concentration of an unspecified chemical species (e.g., a nutrient in the context of tumor growth modelling \cite{EG19jde,GLSS,GL17e,MRS}).

 For the sake of simplicity, we assume that the density difference between the two components of the mixture is negligible and set the density to be one. The mobilities for $\varphi$ and $\sigma$ are also assumed to be positive constants (set to be one), but we allow the fluid mixture to have unmatched viscosities. Assuming that $\nu_1$, $\nu_2>0$ are viscosities of the two homogeneous fluids, the mean viscosity is modeled by the concentration dependent term $\nu=\nu(\varphi)$, for instance, a typical form is given by the linear combination:
\be
\nu(\varphi)=\nu_1\frac{1+\varphi}{2}+\nu_2\frac{1-\varphi}{2}.
\label{vis}
\ee
The variable $\mu: \Omega \times(0,T)\to \mathbb{R}$ denotes the chemical potential associated to $(\varphi, \sigma)$. In the diffuse interface framework, the macroscopically immiscible fluids undergo a smooth and rapid transition in an interfacial region between the two
components, where the parameter $\varepsilon>0$ is related to the thickness of these interfacial layers. Since in this paper we do not consider the asymptotic behavior as $\varepsilon\to 0^+$ (i.e., the sharp-interface limit), hereafter we simply take $\varepsilon =1$.
The constant coefficient $\chi \geq 0$ is related to some specific transport mechanisms such as chemotaxis (in \eqref{f4.d}) and active transport (in \eqref{f2.b}), we refer to \cite{GLSS} for detailed explanations in the context of tumor growth modelling. In the subsequent analysis, the sign of $\chi$ does not play a role, so we allow $\chi\in \mathbb{R}$. The nonlinear function $\varPsi$ denotes the homogeneous free energy density for the binary fluid mixture, which has a double-well structure with two minima and a local unstable maximum in between. A physically significant example is the Flory--Huggins type (see \cite{CH, Gur}):
\be
\varPsi (r)=\frac{\theta}{2}\big[(1-r)\ln(1-r)+(1+r)\ln(1+r)\big]+\frac{\theta_{0}}{2}(1-r^2),\quad \forall\, r\in[-1,1],
\label{pot}
\ee
where the constant parameters $\theta$ and $\theta_0$ fulfill   $0<\theta<\theta_{0}$. In the literature, the above singular potential $\varPsi$ is often approximated by a fourth-order polynomial (e.g., in the regime of shallow quench)
\be
\varPsi(r)=\frac{1}{4}\big(1-r^2\big)^2,\quad \forall\, r\in\mathbb{R}, \label{regular}
\ee
or by some more general polynomial functions \cite{Mi19}. This approximation rules out possible singularities of the logarithmic potential \eqref{pot} (and its derivatives) at $\pm 1$ and brings great convenience in the mathematical analysis as well as numerical simulations for the Cahn--Hilliard type equations, see \cite{BGM,DFW,GG2010,JWZ,LS,RH99,ZWH} and the references therein. On the other hand, we refer to \cite{AACGV,PP21} for some different type of singular potentials in the modelling of living tissues (e.g., a tumor).

In \eqref{f1.a} and \eqref{f4.d} we denote by
$\overline{\varphi}=|\Omega|^{-1}\int_\Omega \varphi(x)\,\mathrm{d}x$ the mean of the phase function $\varphi$ over $\Omega$. Besides, $\mathcal{N}$ denotes the inverse of the minus Laplacian operator subject to a homogeneous Neumann boundary condition on some function space with zero mean constraint (see Section \ref{pm}).
The two terms $\alpha(\overline{\varphi}-c_0)$ and $\beta\mathcal{N}(\varphi-\overline{\varphi})$ with $\alpha\geq 0$, $\beta\in \mathbb{R}$ represent certain nonlocal interactions (i.e., reaction) between the two fluid components. Neglecting the coupling with $\bm{v}$ and $\sigma$, for $\alpha=\beta=0$, we simply recover the classical Cahn--Hilliard equation with constant mobility \cite{CH,CMZ,Mi19}. When taking $\alpha=\beta>0$, we obtain the well-known Cahn--Hilliard--Oono (CHO) equation
\begin{align}
\partial_t \varphi+\alpha(\varphi-c_0) =\Delta(-\Delta \varphi+ \varPsi'(\varphi)),
\label{CHO}
\end{align}
which was proposed to describe the dynamics of
microphase separation of diblock copolymers \cite{OP87}. See also \cite{Glo95} for a different physical origin from a binary mixture with reversible isomerization chemical reaction, where the parameters $\alpha$ and $c_0$ can be determined by the forward and backward reaction rates. If $\overline{\varphi_0}=c_0$, the CHO equation \eqref{CHO} can be viewed as a conserved gradient flow of the Ohta--Kawasaki functional (with $\varPsi$ given by \eqref{regular})
$$
\mathcal{F}_{\mathrm{OK}}(\varphi)= \int_{\Omega}
 \left(\frac{1} {2}|\nabla \varphi|^2+ \varPsi(\varphi) \right) \mathrm{d}x + \alpha \int_{\Omega\times \Omega}
 G(x-y)(\varphi(x)-\overline{\varphi}) (\varphi(y)-\overline{\varphi})\,\mathrm{d}x\mathrm{d}y,
$$
where $G$ denotes the Green function associated to the minus Laplacian subject to a homogeneous Neumann boundary condition \cite{OK}. For results on well-posedness, long-time behavior and optimal control of the CHO equation, we refer to \cite{CGRS22,GGM2017,Mi11}.
\smallskip

\textbf{Known results}. With a regular potential including \eqref{regular} and some more general reaction terms, the initial boundary value problem \eqref{f3.c}--\eqref{ini0} was first studied in \cite{LW}, where the authors proved the existence of global weak solutions in two and three dimensions, and the existence of a unique global strong solution in two dimensions. However, due to the coupling between $\varphi$ and $\sigma$ as well as the loss of maximum principle for the fourth order Cahn--Hilliard equation with regular potentials, they had to impose some technical assumption on the coefficients to achieve the existence of solutions. Later in \cite{H}, the author removed this restriction by considering a singular potential like \eqref{pot}, which is physically relevant for applications in materials science. By a semi-Galerkin approach, she was able to prove the existence of global weak solutions in both two and three dimensions as well as the uniqueness of global weak solutions in two dimensions. For the same system, the authors of \cite{H1} further established the global strong well-posedness for arbitrarily large and regular initial data in the two dimensional setting. The proofs carried out in \cite{H,H1} rely on the key property that the singular potential $\varPsi$ can guarantee the phase function $\varphi$ to always stay in the physical interval $[-1,1]$ along time evolution (thus, a uniform $L^\infty_tL^\infty_x$-bound on $\varphi$ is available). Moreover, in the two dimensional case, $\varphi$ can be strictly separated from the pure states $\pm 1$ once $t>0$. This separation property plays an essential role in the study of regularity and long-time behavior of solutions to the Cahn--Hilliard equation with a singular potential \cite{A2007,CMZ,GGM2017,Mi19,MZ04}, we also refer the reader to \cite{A2009,Boyer,CG,Gio2021,GGW,GMT} for studies of generalized systems with fluid coupling.
\smallskip

\textbf{Our contribution}. The well-posedness for problem \eqref{f3.c}--\eqref{ini0} with a singular potential including \eqref{pot} in three dimensions remains an open question so far. Our aim here is to make a first contribution in this aspect: under coupled effects of chemotaxis, active transport and a specific type of reaction, we first establish the global well-posedness of problem \eqref{f3.c}--\eqref{ini0} with a singular potential in a bounded smooth domain $\Omega$ in $\mathbb{R}^3$, then we characterize the long-time behavior of its global strong solutions as $t\to+\infty$. Since the system under consideration contains the Navier--Stokes equations driven by a capillary force term as a subsystem, it is naturally to expect that the global well-posedness result can only be obtained under suitable smallness assumptions on the initial data, that is, near a certain equilibrium. On the other hand, the nonconvexity of the free energy $\mathcal{F}(\varphi, \sigma)$ (see \eqref{fe1} below) implies that problem \eqref{f3.c}--\eqref{ini0} may admit a set of (nontrivial) equilibria with rather complicated structure. Thus, classical techniques to prove the existence of global small solutions (e.g., linearization combined with a small energy argument) do not apply.

Let us first summarize the main results of this paper. Details of their statements are presented in Section \ref{sec:mainr} below.
\begin{itemize}
\item[(A)] \textit{Global well-posedness: Theorem \ref{3main}}. When the initial velocity $\bm{v}_0$ is small, the initial phase-field function $\varphi_0$ and the chemical density $\sigma_0$ are small perturbations of a given local minimizer $(\varphi_*,\sigma_*)$ of the free energy $ \mathcal{F}(\varphi,\sigma)$, problem \eqref{f3.c}--\eqref{ini0} admits a unique global strong solution $(\bm{v},\varphi,\mu,\sigma)$ that is uniformly bounded for $t\geq 0$. Moreover, any energy minimizer of $\mathcal{F}$ is Lyapunov stable under the evolution governed by problem \eqref{f3.c}--\eqref{ini0}.
\item[(B)] \textit{Long-time behavior: Theorem \ref{3main1}}. Every global strong solution $(\bm{v},\varphi, \sigma)$ obtained in Theorem \ref{3main} converges to a single equilibrium $(\bm{0}, \varphi_\infty, \sigma_\infty)$ as $t\to +\infty$. Besides, we obtain an estimate on the convergence rate.
\end{itemize}

\textbf{Features of the problem}.
The hydrodynamic system \eqref{f3.c}--\eqref{f2.b} is thermodynamically consistent as shown in \cite{LW}. The free energy associated to it is given by (taking $\varepsilon=1$ for simplicity)
\be
 \mathcal{F}(\varphi,\sigma)=\int_{\Omega}
 \left(\frac{1} {2}|\nabla \varphi|^2+ \varPsi(\varphi)+\frac{1}{2}|\sigma|^2-\chi\sigma\varphi +\frac{\beta}{2}|\nabla\mathcal{N}(\overline{\varphi}-\varphi)|^2\right) \mathrm{d}x.\label{fe1}
\ee
Then a direct computation leads to the following energy balance equation for sufficiently smooth solutions $(\bm{v}, \varphi,\mu, \sigma)$:
\begin{align}
& \frac{\mathrm{d}}{\mathrm{d}t} \left( \int_{\Omega} \frac{1}{2}|\bm{v}|^2\, \mathrm{d}x+ \mathcal{F}(\varphi,\sigma)\right)   +\int_{\Omega} \left( 2\nu(\varphi)|D\bm{v}|^2 + |\nabla \mu|^2+|\nabla(\sigma-\chi\varphi)|^2\right) \mathrm{d}x\nonumber\\
&\quad  =-\alpha\int_\Omega (\overline{\varphi}-c_0)\mu\, \mathrm{d}x.
\label{BEL}
\end{align}
Besides, (formally) integrating  the equation \eqref{f1.a} over $\Omega$ and using the boundary condition \eqref{boundary}, we have
\begin{align}
\frac{\mathrm{d}}{\mathrm{d}t}(\overline{\varphi}-c_0)+\alpha(\overline{\varphi}-c_0)=0,
\label{mph1}
\end{align}
which implies
\begin{align}
\overline{\varphi}(t)-c_0=(\overline{\varphi_0}-c_0)e^{-\alpha t},\quad \forall\, t\geq 0.
\label{mph2}
\end{align}
Hence, if $\alpha=0$, or $\overline{\varphi_0}=c_0$ for $\alpha>0$, then the mass $\overline{\varphi}(t)$ is conserved in time, otherwise, $\overline{\varphi}(t)$ converges exponentially fast to $c_0$ provided that $\alpha>0$ (this is the so-called off-critical case, cf. \cite{BGM,GGM2017,MT}).

The energy balance \eqref{BEL} and the mass relation \eqref{mph2} play a crucial role in the subsequent analysis for problem \eqref{f3.c}--\eqref{ini0}. To this end, we observe that the well-known ``Model H'' for incompressible binary fluids (see \cite{HH,Gur}) can be recovered from the system \eqref{f3.c}--\eqref{f2.b} by simply neglecting $\sigma$ and taking $\alpha=\beta=0$. Then it follows from \eqref{BEL} and \eqref{mph2} that the resulting Navier--Stokes--Cahn--Hilliard (NSCH) system has a dissipative structure (i.e., its total energy is decreasing in time) and the mass $\overline{\varphi}$ is conserved for all $t\geq 0$. Based on these two facts, the NSCH system (with either a regular potential or a singular potential) has been extensively studied in the literature. In this respect, we may quote \cite{A2009,Boyer,GG2010,GMT,LS,ZWH, Zhou20} and the references therein for results on well-posedness, long-time behavior of global solutions and asymptotic analyses.
Next, in the system \eqref{f3.c}--\eqref{f2.b}, neglecting $\sigma$ and taking the convective CHO equation (cf. \eqref{CHO})
\begin{subequations}
	\begin{alignat}{3}
	&\partial_t \varphi+\bm{v} \cdot \nabla \varphi  =\Delta \widetilde{\mu} -   \alpha(\varphi-c_0),\label{f1.av} \\
	&\widetilde{\mu}=- \Delta \varphi + \varPsi'(\varphi),\label{f4.dv}
	\end{alignat}
\end{subequations}
we arrive at the so-called Navier--Stokes--Cahn--Hilliard--Oono (NSCHO) system that
has been analyzed in \cite{BGM,MT}. The NSCHO system  provides a description about the hydrodynamic effects on phase separation of binary mixtures with reversible chemical reaction \cite{H03}. In \cite{BGM}, the authors investigated global dynamical behavior of the NSCHO system in two dimensions with a regular potential (i.e., a polynomial) by constructing a family of exponential attractors that is continuous with respect to the parameter $\alpha$. Later, the NSCHO system with $c_0=0$ and a singular potential like \eqref{pot} was analyzed in \cite{MT}, where the authors proved the existence of global weak solutions in both two and three dimensions and showed the sequential convergence of solutions as $\alpha\to 0^+$.  However, several basic issues remain open for the NSCHO system studied therein (in particular, the case with a physically relevant singular potential \eqref{pot}), for instance, local well-posedness in two and three dimensions, instantaneous/eventual regularity of global weak solutions, asymptotic stabilization of a bounded global solution towards an equilibrium as $t\to +\infty$, and so on.

In view of \cite{BGM,MT}, we find that the analysis for the NSCHO system turns out to be more tricky, because of a combination of the hydrodynamic effect with the nonlocal reaction term of Oono's type $\alpha(\varphi-c_0)$. In particular, it is unclear whether the NSCHO system studied therein can admit a Lyapunov functional or not. Indeed, with \eqref{f1.av}--\eqref{f4.dv} instead of \eqref{f1.a}--\eqref{f4.d} and neglecting $\sigma$ again, we can derive the following energy identity:
\begin{align}
& \frac{\mathrm{d}}{\mathrm{d}t}
 \int_{\Omega} \left(\frac{1}{2}|\bm{v}|^2+\frac{1} {2}|\nabla \varphi|^2+ \varPsi(\varphi)\right) \mathrm{d}x
 +\int_{\Omega} \left( 2\nu(\varphi)|D\bm{v}|^2 + |\nabla \widetilde{\mu}|^2\right) \mathrm{d}x =-\alpha\int_\Omega (\varphi-c_0)\widetilde{\mu}\, \mathrm{d}x.
\notag
\end{align}
Comparing with the right-hand side of \eqref{BEL}, the NSCHO system in \cite{BGM,MT} has a higher-order energy production rate without a definite sign (recalling that $\overline{\varphi}$ only depends on time and decays to $c_0$ exponentially fast). This fact yields an essential difficulty to study its global well-posedness and long-time behavior, in particular, the convergence to a single equilibrium. For the single CHO equation \eqref{CHO} without fluid interaction, the authors of \cite{GGM2017} overcame the above mentioned difficulty by considering the Ohta–Kawasaki functional $\mathcal{F}_{\mathrm{OK}}$ and rewriting the equation as
 \begin{align}
\partial_t \varphi+\alpha(\overline{\varphi}-c_0) =\Delta\big(-\Delta \varphi+ \varPsi'(\varphi)+\alpha \mathcal{N}(\varphi-\overline{\varphi})\big)=\Delta \frac{\delta \mathcal{F}_{\mathrm{OK}}(\varphi)}{\delta \varphi}.
\label{CHOb}
\end{align}
When the hydrodynamic effect is taken into account as in our system \eqref{f3.c}--\eqref{f2.b},
inspired by \cite{GGM2017,H03}, we choose to replace the Oono's reaction term $ \alpha(\varphi-c_0)$ by a more general form $    \beta(\varphi-\overline{\varphi})+\alpha(\overline{\varphi}-c_0)$
with $\beta\in \mathbb{R}$ and $\alpha \geq 0$ (recall \eqref{f1.a}--\eqref{f4.d}). The first part $\beta(\varphi-\overline{\varphi})$ contributes not only  to the free energy $\mathcal{F}$ (see \eqref{fe1}), but also to the capillary force (see \eqref{f3.c}) at free interfaces between the binary fluids (cf. \cite{H03}).
In particular, the sign of the interaction strength $\beta$ can indicate the possible suppression or enhancement of chemical reaction for the phase separation and coarsening processes (see \cite{Choski,DFW}). Meanwhile, the second part $\alpha(\overline{\varphi}-c_0)$ yields a possible mass transfer between the two components as well as a ``lower-order'' energy production rate as in \eqref{BEL}.
 This simple modification not only agrees with the variational framework as in \cite{GLSS,LW}, but also provides a suitable correction to the NSCHO system investigated in the previous literature \cite{BGM,MT}. From the mathematical point of view, we are now able to show that problem \eqref{f3.c}--\eqref{ini0} admits a Lyapunov functional (see \eqref{BEL5} below), and as a consequence, the resulting PDE system enjoys a dissipative structure. This crucial fact enables us to study the existence of global strong solutions (with small initial data) and their long-time behavior by the \L ojasiewicz--Simon approach \cite{C2003,GG2006,LS83}. On the other hand, we note that more general reaction terms have been proposed and analyzed in the Cahn--Hilliard equation for biological applications \cite{Fa15,Lam22} and its extended systems for the tumor modelling \cite{EG19jde,GL17,GL17e,GLSS,GLRS,KS22,MRS}. It is an interesting and challenging task to show the global strong well-posedness and long-time behavior for the hydrodynamic system \eqref{f3.c}--\eqref{f2.b} with general reaction terms other than Oono's type.
\smallskip

\textbf{Strategy of proofs}. To prove Theorem \ref{3main}, we first show the existence and uniqueness of a local strong solution to problem \eqref{f3.c}--\eqref{ini0} in three dimensions (see Theorem \ref{ls}), based on a semi-Galerkin approximation scheme similar to that in \cite{H}. In this framework, we perform a Galerkin approximation only for the Navier--Stokes equations \eqref{f3.c}, but solve the coupled system for $(\varphi, \sigma)$ independently. Thus, we can take the advantage that the approximate velocity $\bm{v}^m$ is sufficiently regular in space and the approximate phase function $\varphi^m$ always takes its values in the physical interval $[-1,1]$. The latter helps us to handle the interaction with chemotaxis (see \eqref{fe1}) and remove the extra assumption on the coefficients as in \cite{LW}. Besides, in order to rigourously justify the higher-order estimates for the approximate solutions, we need to introduce an additional approximation involving the convective viscous Cahn--Hilliard equation (cf. \cite{MZ04,Gio2022}) together with some proper regularizations for the initial datum $(\varphi_0, \sigma_0)$ (cf. \cite{BGM,GMT}). For the viscous Cahn--Hilliard equation with a singular potential like \eqref{pot}, the viscous term allows one to derive a strict separation of the phase function $\varphi$ from the pure phases $\pm 1$ and thus yields sufficient spatial regularity thanks to the elliptic estimate \cite{H2,MZ04}. Finally, uniqueness of the local strong solution easily follows from an energy method similar to that in \cite{GMT,H}.

To prove the global existence, a naturally idea is to derive uniform-in-time estimates for the local strong solution so that it can be extended to the whole interval $[0,+\infty)$. As we have mentioned before, the method involving linearization does not work due to the nonconvexity of the free energy $\mathcal{F}$ and the complicated structure of the equilibrium set. Our proof is inspired by the argument in \cite{LL95} for the incompressible liquid crystal flow. The basic idea is, if one can show that the total energy (i.e., the sum of kinetic energy and free energy) does not drop ``too much'' during the evolution, then the strong solution will exist globally in time. This criterion on global well-posedness can be easily realized if the initial velocity $\bm{v}_0$ is small, and the initial phase-field function $\varphi_0$ as well as the chemical density $\sigma_0$ are small perturbations of a \emph{global} minimizer of the free energy $\mathcal{F}(\vp,\sigma)$, see e.g., \cite{ZWH} for related results on the NSCH system with a regularity potential. However, the situation is more involved if we consider the case near a \emph{local} minimizer $(\varphi_{*},\sigma_{*})$. Thanks to the existence of a Lyapunov functional, we observe that lower-order norm of the solution $(\bm{v}, \varphi,\sigma)$ in $\bm{L}^2\times H^1\times L^2$ remains uniformly bounded. However, its higher-order norm in $\bm{H}^1\times H^3\times H^1$ can grow fast and may even blow up in finite time. This situation can be avoided if we can guarantee the solution to stay in a suitable small neighbourhood of $(\bm{0},\varphi_{*},\sigma_{*})$ in $\bm{L}^2\times H^2\times L^2$ so that the total energy cannot drop too much as mentioned before. To achieve this goal, we apply the \L ojasiewicz--Simon approach \cite{LS83} that has been proved to be an efficient tool to study the long-time behavior of nonlinear evolution equations, see \cite{A2007,C2003,GG2006,GG2010,GGW,JWZ,RH99,ZWH} and the references therein.

As a preliminary step, we first investigate the existence and regularity of energy minimizers of the free energy $\mathcal{F}(\vp, \sigma)$ (see Proposition \ref{prop-le2}).
Then we derive an extended {\L}ojasiewicz--Simon type inequality for the phase function together with the chemical density (see Theorem \ref{LSmain}), which has its independent interest. Our result treats the interactions induced by chemotaxis/active transport and allows possible mass transfer, so that it generalizes the classical {\L}ojasiewicz--Simon inequality for the Cahn--Hilliard equation \cite{A2007,GG2006,RH99}. We note that the strict separation of any steady phase function from pure states $\pm1$ plays an important role in the derivation of Theorem \ref{LSmain}, since the singular potential $\varPsi$ can be regarded as a smooth (even analytic) function on some compact subset of $(-1,1)$ within its small neighbourhood in $H^2$ (recall the continuous embedding $H^2\hookrightarrow L^\infty$ in three dimensions).

However, for the evolution problem \eqref{f3.c}--\eqref{ini0} with a logarithmic potential \eqref{pot} in three dimensions, a uniform and strict separation of the solution $\varphi$ from $\pm 1$ remains an open question. This is the case even for the single Cahn--Hilliard equation \cite{Mi19,MZ04}. Thus, to prove the global existence, we need to control not only the growth of higher-order norm of $(\bm{v}, \varphi, \sigma)$ in  $\bm{H}^1\times H^3\times H^1$, but also the distance of $\varphi$ from the pure states $\pm 1$. Our strategy can be roughly summarized as follows. Thanks to Theorem \ref{ls}, we start with a (unique) local strong solution  $(\bm{v}, \varphi, \sigma)$ defined on a certain finite interval $[0,T_2]$ with $T_2>0$, where we have a finite growth of $\|\bm{v}\|_{\bm{H}^1}$, $\|\varphi\|_{H^3}$, $\|\sigma\|_{H^1}$ and $\|\varphi\|_{C(\overline{\Omega})}$. We choose the initial datum $(\bm{v}_0,\varphi_0, \sigma_0)$ to be sufficiently close to $(\bm{0}, \varphi_*, \sigma_*)$ in $\bm{L}^2\times H^2\times L^2$ so that the possible changes of the total energy as well as the free energy on $[0,T_2]$ are sufficiently small. Then applying the {\L}ojasiewicz--Simon approach (via a contradiction argument), we show that the small energy variation leads to a small change of distances in lower-order norms for $\|\varphi-\varphi_*\|_{(H^1)'}$ and $\|\sigma-\sigma_*\|_{(H^1)'}$, which can
compensate the growth of $\|\varphi\|_{H^3}$ and $\|\sigma\|_{H^1}$ on $[0,T_2]$. By interpolation, we find that $\|\varphi-\varphi_*\|_{H^2}$ and $\|\sigma-\sigma_*\|_{L^2}$ can be kept small on $[0,T_2]$ as required. From the Sobolev embedding theorem,  $\|\varphi-\varphi_*\|_{C(\overline{\Omega})}$ is also properly controlled so that $\varphi$ is strictly separated from $\pm 1 $ on $[0,T_2]$, with the same upper bound on the distance as the initial datum $\varphi_0$. Taking advantage of the dissipative structure of the system \eqref{f3.c}--\eqref{f2.b}, we can find some $t^*\in [T_2/2,T_2]$ and take $(\bm{v}(t^*),\varphi(t^*), \sigma(t^*))$ as the new initial datum to conclude the existence of a unique local strong solution $(\bm{v}, \varphi, \sigma)$ on $[t^*,t^*+T_2]$ by applying Theorem \ref{ls}. Thanks to the uniqueness, we indeed obtain a unique strong solution $(\bm{v}, \varphi, \sigma)$ on the enlarged interval $[0,3T_2/2]$, which turns out to satisfy the same estimates for its norms as on $[0,T_2]$. Then we can repeat the above procedure and extend the local strong solution with a fixed time step of $T_2/2$ to obtain a unique global strong solution that is uniformly bounded for all $t\geq 0$. This iterative argument has been successfully applied to some simpler diffuse-interface models for incompressible binary fluids with a singular potential like \eqref{pot}, for instance, the Hele--Shaw--Chan--Hilliard system (see \cite{GGW}) and the Model H (see \cite{Zhou20}).

Once the existence of a bounded global strong solution is established, we can study its long-time behavior by applying the \L ojasiewicz--Simon approach as in \cite{A2009,ZWH} for the Model H. Moreover, we can derive an estimate on the convergence rate by investigating the energy decay (cf. \cite{HJ2001}) and using an energy method (see \cite{W07,ZWH}).
\smallskip

\textbf{Plan of the paper}. The remaining part of this paper is organized as follows. In Section \ref{pm}, we introduce the functional settings and state the main results. Section \ref{gss} deals with the existence and uniqueness of a local strong solution.
Section \ref{sec:stapro} is devoted to the study of the  stationary problem. We show the existence and regularity of energy minimizers for the free energy $\mathcal{F}(\varphi,\sigma)$ and establish an extended {\L}ojasiewicz--Simon inequality.
In Section \ref{sec:glostr}, we prove the existence and uniqueness of a global strong solution, provided that the initial velocity $\bm{v}_0$ is small, while the initial phase-field function and the chemical density $(\varphi_0,\sigma_0)$ are small perturbations of a given local minimizer of $\mathcal{F}$. In Section \ref{lbg}, we prove the convergence to a single equilibrium for any bounded global strong solution as time goes to infinity and derive an estimate on the convergence rate.

%%%%%%%%%%%%%%%%%%%%%%%%%%%%%%%%%%%%%%
\section{Main Results}\label{pm}
\setcounter{equation}{0}
\subsection{Preliminaries}
First, we introduce some notations and basic tools that will be used in this paper. Let $X$ be a real Banach space. Its dual space is denoted by $X'$, and the duality pairing is denoted by
$\langle \cdot,\cdot\rangle_{X',X}$. Given an interval $J\subset [0,+\infty)$, $L^q(J;X)$ with $q\in [1,+\infty]$ denotes the space consisting of Bochner measurable $q$-integrable/essentially bounded functions with values in the Banach space $X$. For $q\in [1,+\infty]$, $W^{1,q}(J;X)$ is the space of all $f\in L^q(J;X)$ with $\partial_t f\in L^q(J;X)$. For $q=2$, we set $H^1(J;X)=W^{1,2}(J;X)$. Throughout the paper, we assume that $\Omega$ is a bounded domain in $\mathbb{R}^3$ with boundary $\partial\Omega$ of class $C^4$. For any $q \in [1,+\infty]$, $L^{q}(\Omega)$ denotes the Lebesgue space with norm $\|\cdot\|_{L^{q}}$. For $k\in \mathbb{Z}^+$, $q\in [1,+\infty]$, $W^{k,q}(\Omega)$ denotes the Sobolev space with norm $\|\cdot\|_{W^{k,q}}$.
For $q = 2$, we use the notation $H^{k}(\Omega)= W^{k,2}(\Omega)$ with the norm $\|\cdot\|_{H^{k}}$. The norm and inner product on $L^{2}(\Omega)$ are simply denoted by $\|\cdot\|$ and $(\cdot,\cdot)$, respectively. The boldface letter $\bm{X}$ denotes the vectorial space $X^3$ endowed with the usual product structure. For instance, the space
$\bm{L}^q(\Omega) := L^q (\Omega; \mathbb{R}^3)$ ($q\in [1,+\infty]$) consists of all $q$-integrable/essentially bounded vector-fields.

For every $f\in (H^1(\Omega))'$, we denote by $\overline{f}$ its generalized mean value over $\Omega$ such that
$\overline{f}=|\Omega|^{-1}\langle f,1\rangle_{(H^1)',H^1}$; if $f\in L^1(\Omega)$, then its mean value is simply given by $\overline{f}=|\Omega|^{-1}\int_\Omega f \,\mathrm{d}x$.
For the sake of convenience, we denote by $L^2_{0}(\Omega):=\{f\in L^2(\Omega)\ |\ \overline{f} =0\}$ the linear subspace of $L^2(\Omega)$ with zero mean. Besides, in view of the homogeneous Neumann boundary condition \eqref{boundary}, we introduce the space
$H^2_{N}(\Omega):=\{f\in H^2(\Omega)\ |\  \partial_{\bm{n}}f=0 \ \textrm{on}\  \partial \Omega\}$.
We denote by $\mathcal{A}_N\in \mathcal{L}\big(H^1(\Omega),(H^1(\Omega))'\big)$ the realization of $-\Delta$ subject to the homogeneous Neumann boundary condition such that
\begin{equation}\nonumber
   \langle \mathcal{A}_N u,v\rangle_{(H^1)',H^1} := \int_\Omega \nabla u\cdot \nabla v \, \mathrm{d}x,\quad \text{for }\,u,v\in H^1(\Omega).
\end{equation}
For linear spaces $V_0=\{ u \in H^1(\Omega)\ |\  \overline{u}=0\}$, $V_0'= \{ u \in (H^1(\Omega))'\ |\  \overline{u}=0 \}$, the restriction of $\mathcal{A}_N$ from $V_0$ onto $V_0'$ is an isomorphism. In particular, $\mathcal{A}_N$ is positively defined on $V_0$ and self-adjoint. We denote its inverse map by $\mathcal{N} =\mathcal{A}_N^{-1}: V_0'
\to V_0$. Then for every $f\in V_0'$, $u= \mathcal{N} f \in V_0$ is the unique weak solution of the Neumann problem
$$
\begin{cases}
-\Delta u=f, \quad \text{in} \ \Omega,\\
\partial_{\bm{n}} u=0, \quad \ \  \text{on}\ \partial \Omega.
\end{cases}
$$
Besides, we have
\begin{align}
&\langle \mathcal{A}_N u, \mathcal{N} g\rangle_{V_0',V_0} =\langle  g,u\rangle_{(H^1)',H^1}, \quad \forall\, u\in V, \ \forall\, g\in V_0',\label{propN1}\\
&\langle  g, \mathcal{N} f\rangle_{V_0',V_0}
=\langle f, \mathcal{N} g\rangle_{V_0',V_0} = \int_{\Omega} \nabla(\mathcal{N} g)
\cdot \nabla (\mathcal{N} f) \, \mathrm{d}x, \quad \forall \, g,f \in V_0',\label{propN2}
\end{align}
and the chain rule
\begin{align}
&\langle \partial_t u, \mathcal{N} u(t)\rangle_{V_0',V_0}
=\frac{1}{2}\frac{\mathrm{d}}{\mathrm{d}t}\|\nabla \mathcal{N} u\|^2,\quad \textrm{a.e. in}\ (0,T),\ \ \text{for any}\ u\in H^1(0,T; V_0').\nonumber
\end{align}
For any $f\in V_0'$, we set $\|f\|_{V_0'}=\|\nabla \mathcal{N} f\|$.
It is well-known that $f \to \|f\|_{V_0'}$ and $
f \to\big(\|f-\overline{f}\|_{V_0'}^2+|\overline{f}|^2\big)^\frac12$ are equivalent norms on $V_0'$ and $(H^1(\Omega))'$,
respectively (see \cite{MZ04}). Recall the well-known Poincar\'{e}--Wirtinger inequality:
\begin{equation}
\label{poincare}
\|f-\overline{f}\|\leq C \|\nabla f\|,\qquad \forall\,
f\in H^1(\Omega),
\end{equation}
where $C>0$ depends only on $\Omega$.
We then infer that $f\to \|\nabla f\|$ and  $f\to \big(\|\nabla f\|^2+|\overline{f}|^2\big)^\frac12$ are equivalent norms on $V_0$ and $H^1(\Omega)$.
We also report the following standard Hilbert interpolation inequality and elliptic estimates for the Neumann problem (assuming that $\Omega$ is sufficiently smooth):
\begin{align}
\|f\| &\leq \|f\|_{V_0'}^{\frac12} \| \nabla f\|^{\frac12},
\qquad \forall\, f \in V_0,\label{I}\\
\|\nabla \mathcal{N} f\|_{\bm{H}^{k}}& \leq C \|f\|_{H^{k-1}},
\qquad\quad\ \forall\, f\in H^{k-1}(\Omega)\cap L^2_0(\Omega),\quad k\in\mathbb{Z}^+.\label{N}
\end{align}

Next, we introduce the Hilbert spaces of solenoidal vector-valued functions (see e.g., \cite{G,S}). For a vector-valued Banach space $\bm{X}$ consisting of functions $\bm{f}: \Omega \rightarrow \mathbb{R}^3$ (e.g. $\bm{L}^2(\Omega)$ and $\bm{H}^1(\Omega)$), we denote $\bm{X}_{\mathrm{div}} $, $\bm{X}_{0,\mathrm{div}} $ by the closure of spaces  $C_{\mathrm{div}}^{\infty}(\Omega;\mathbb{R}^3)=\{\bm{f}\in C^{\infty}(\Omega;\mathbb{R}^3):\ \textrm{div}\bm{f}=0\}$,  $C_{0,\mathrm{div}}^{\infty}(\Omega;\mathbb{R}^3)=\{\bm{f}\in C_0^{\infty}(\Omega;\mathbb{R}^3):\ \textrm{div}\bm{f}=0\}$ with respect to the $\bm{X}$-norm, respectively. For the space $\bm{L}^2_{0,\mathrm{div}}(\Omega)$, we also use $(\cdot,\cdot)$ and $\|\cdot\|$ for its inner product and norm. It is well known that $\bm{L}^2(\Omega)$ can be decomposed into $\bm{L}^2_{0,\mathrm{div}}(\Omega)\oplus\bm{G}(\Omega)$, where $\bm{G}(\Omega):=\{\bm{f}\in\bm{L}^2(\Omega): \exists\, z\in H^1(\Omega),\ \bm{f}=\nabla z\}$. Then for any function $\bm{f} \in \bm{L}^2(\Omega)$, the Helmholtz--Weyl decomposition holds (see  \cite[Chapter \uppercase\expandafter{\romannumeral3}]{G}):
\be
\bm{f}=\bm{f}_{0}+\nabla z,\quad\text{where}\  \bm{f}_{0} \in \bm{L}^2_{0,\mathrm{div}}(\Omega),\ \nabla z \in \bm{G}(\Omega).\nonumber
 \ee
Consequently, we can define the Leray projection onto the space of divergence-free functions $\bm{P}:\bm{L}^2(\Omega)\to \bm{L}^2_{0,\mathrm{div}}(\Omega)$ such that $\bm{P}(\bm{f})=\bm{f}_{0}$. The space $\bm{H}^1_{0,\mathrm{div}}(\Omega)$ is equipped with the scalar product $
(\bm{u},\bm{v})_{\bm{H}^1_{0,\mathrm{div}}}:=(\nabla \bm{u},\nabla \bm{v})$ and the norm $\|\bm{u}\|_{\bm{H}^1_{0,\mathrm{div}}}=\|\nabla \bm{u}\|$.
From the well-known Korn's inequality
$$
\|\nabla \bm{u}\|\leq \sqrt{2}\|D\bm{u}\|,\quad \forall\, \bm{u}\in \bm{H}^1_{0,\mathrm{div}}(\Omega),
$$
we see that $\|D\bm{u}\|$ is an equivalent norm for $\bm{H}^1_{0,\mathrm{div}}(\Omega)$. Next, we introduce the Stokes operator $\bm{S}: D(\bm{S})= \bm{H}^1_{0,\mathrm{div}}(\Omega)\cap\bm{H}^2(\Omega) \to\bm{L}^2_{0,\mathrm{div}}(\Omega)$ such that $\bm{S}=\bm{P}(-\Delta)$ (see e.g., \cite[Chapter III]{S}).
The space $\bm{H}^1_{0,\mathrm{div}}(\Omega)\cap\bm{H}^2(\Omega)$ can be equipped with the inner product $(\bm{S}\bm{u},\bm{S}\bm{v})$ and the norm $\|\bm{S}\bm{u}\|$.
 For any $\bm{u}\in D(\bm{S})$ and $\bm{\zeta} \in \bm{H}^1_{0,\mathrm{div}}(\Omega)$, it holds
$(\bm{S}\bm{u},\bm{\zeta})=(\nabla \bm{u},\nabla\bm{\zeta})$. The operator $\bm{S}$ is a canonical isomorphism from $\bm{H}^1_{0,\mathrm{div}}(\Omega)$ to $(\bm{H}^1_{0,\mathrm{div}}(\Omega))'$. Denote its inverse map by $\bm{S}^{-1}$. For any $\bm{f}\in (\bm{H}^1_{0,\mathrm{div}}(\Omega))'$, there is a unique $\bm{u}=\bm{S}^{-1}\bm{f}\in\bm{H}^1_{0,\mathrm{div}}(\Omega)$ such that
\be
(\nabla\bm{S}^{-1}\bm{f},\nabla \bm{\zeta}) =\langle\bm{f},\bm{\zeta}\rangle_{(\bm{H}^1_{0,\mathrm{div}})', \bm{H}^1_{0,\mathrm{div}}},\quad \forall\, \bm{\zeta} \in \bm{H}^1_{0,\mathrm{div}}(\Omega).\nonumber
\ee
Then we see that $\|\nabla\bm{S}^{-1}\bm{f}\| =\langle\bm{f},\bm{S}^{-1}\bm{f}\rangle_{(\bm{H}^1_{0,\mathrm{div}})',\bm{H}^1_{0,\mathrm{div}}}^{\frac{1}{2}}$ is an equivalent norm on $(\bm{H}^1_{0,\mathrm{div}}(\Omega))'$ and the chain rule holds
\be
\langle\partial_t \bm{f}(t),\bm{S}^{-1}\bm{f}(t)\rangle_{(\bm{H}^1_{0,\mathrm{div}})',\bm{H}^1_{0,\mathrm{div}}} =\frac{1}{2}\frac{\mathrm{d}}{\mathrm{d}t}\|\nabla\bm{S}^{-1}\bm{f}\|^2,
\ee
for almost all $t \in (0,T)$ and any $\bm{f}\in H^1(0,T;\bm{H}^1_{0,\mathrm{div}}(\Omega)')$.

Finally, let us recall the Ladyzhenskaya, Agmon and Gagliardo--Nirenberg inequalities in three dimensions:
\begin{align*}
&\|f\|_{L^3}\leq C\|f\|_{H^1}^\frac12\|f\|^\frac12,\qquad\quad\   \forall\,f\in H^1(\Omega),\\
&\|f\|_{L^\infty}\leq  C\|f\|_{H^2}^\frac12\|f\|_{H^1}^\frac12,\qquad \ \ \forall\,f\in H^2(\Omega),\\
&\|\nabla f\|_{\bm{L}^4}\leq C\|f\|_{H^2}^\frac12\|f\|_{L^\infty}^\frac12,\qquad \forall\, f\in H^2(\Omega),
\end{align*}
where the constant $C>0$ only depends on $\Omega$.
We also recall the following regularity result for the Stokes operator (see e.g.,  \cite[Chapter III, Theorem 2.1.1]{S} and \cite[Lemma B.2]{GMT}):
\bl \label{stokes}
 \rm Let $\Omega$ be a bounded domain of class $C^2$ in $\mathbb{R}^3$. For any $\bm{f} \in \bm{L}^2_{0,\mathrm{div}}(\Omega)$,
there exists a unique pair $\bm{u}\in \bm{H}^1_{0,\mathrm{div}}(\Omega)\cap\bm{H}^2(\Omega)$ and $p\in H^1(\Omega)\cap L_0^2(\Omega)$ such that $-\Delta \bm{u}+\nabla p=\bm{f}$ a.e. in $\Omega$, that is, $\bm{u}=\bm{S}^{-1}\bm{f}$. Moreover, it holds
\begin{align*}
&\|\bm{u}\|_{\bm{H}^2}+\|\nabla p\|\le C\|\bm{f}\|,\quad
\|p\|\le C \|\bm{f}\|^\frac12\|\nabla \bm{S}^{-1}\bm{f}\|^\frac12,
 \end{align*}
where $C>0$ depends on $\Omega$ but is independent of $\bm{f}$.
 \el

Throughout the paper, the letters $C$, $C_i$ denote some generic positive constants, whose specific dependence will be pointed out if necessary.
%%%%%%%%%%%%%%%%%

\subsection{Main results}\label{sec:mainr}
\noindent
The following assumptions will be needed throughout the paper.
\begin{enumerate}
	\item[(H1)]\label{seta} The viscosity $\nu$ satisfies $\nu \in C^{2}(\mathbb{R})$ and \\
	\be
	\nu_{*} \leq \nu(r)\leq \nu^*,\quad |\nu'(r)|\leq \nu_{0},\quad|\nu''(r)|\leq \nu_{1},\quad \forall\, r \in \mathbb{R},\nonumber
	\ee
	where $\nu_{*}$, $\nu^*$, $\nu_{0}$ and $\nu_{1}$ are some positive constants.
	\item[(H2)]\label{item:as} The singular potential $\varPsi$ belongs to the class of functions $C\big([-1,1]\big)\cap C^{3}\big((-1,1)\big)$ and can be written into the following form
	\begin{equation}
	\varPsi(r)=\varPsi_{0}(r)-\frac{\theta_{0}}{2}r^2,\nonumber
	\end{equation}
	such that
	\begin{equation}
	\lim_{r\to \pm 1} \varPsi_{0}'(r)=\pm \infty\quad \text{and}\quad  \varPsi_{0}''(r)\ge \theta,\quad \forall\, r\in (-1,1),\nonumber
	\end{equation}
 with the strictly positive constants $\theta_{0},\  \theta$ satisfying $\theta_{0}-\theta>0$.
 There exists $\epsilon_0\in(0,1)$ such that $\varPsi_{0}''$ is nondecreasing in $[1-\epsilon_0,1)$ and nonincreasing in $(-1,-1+\epsilon_0]$.
	Besides, we make the extension $\varPsi_{0}(r)=+\infty$ for $r\notin[-1,1]$.
	 \item[(H3)] The function $\varPsi$ is real analytic in  $(-1,1)$.
	\item[(H4)]\label{sco} $\Omega$ is a bounded domain in $\mathbb{R}^3$, with boundary $\partial\Omega$ of class $C^4$. The coefficients $\chi$, $c_0$,  $\alpha$, $\beta$ are  prescribed  constants that satisfy
	\be
	\chi \in \mathbb{R},\quad c_0\in(-1,1),\quad \alpha\geq 0,\quad \beta\in\mathbb{R}. \nonumber
	\ee
\end{enumerate}
\begin{remark}
	 The logarithmic potential \eqref{pot} fulfills the assumptions (H2)--(H3). As indicated in \cite[Remark 2.1]{GMT}, one can easily extend the linear viscosity function \eqref{vis} to $\mathbb{R}$ in such a way to comply (H1). Indeed, since the singular potential guarantees that the phase function satisfies $\varphi\in [-1,1]$, the value of $\nu$ outside of $[-1,1]$ is not important and can be chosen as in (H1).
\end{remark}
\begin{remark}
We require the boundary $\partial \Omega$ to be sufficiently smooth in order to obtain some higher order spatial regularity of the solution. If $\partial \Omega$ belongs to the class $C^4$ as assumed in (H4), then by the classical elliptic estimate for the Neumann problem (see, e.g. \cite[Theorem 9.26]{Bre}), we can deduce that $\varphi \in  L_{\mathrm{loc}}^{2}\left(0,+\infty; H^{4}(\Omega)\right)$ as stated in Theorem \ref{3main} below.
\end{remark}

Our aim is to prove the existence and uniqueness of global strong solutions to problem \eqref{f3.c}--\eqref{ini0}, provided that the initial velocity $\bm{v}_0$ is suitably small, while the initial phase-field function $\varphi_0$ and the initial chemical density $\sigma_0$ are small perturbations of certain local minimizer of the free energy $\mathcal{F}(\varphi,\sigma)$. To this end, we first introduce the following definition.

\begin{definition}[Energy minimizer]\label{min:ene}
Let $m_1\in (-1,1)$ and $m_2\in \mathbb{R}$ be two given constants. We define the function space
$$
	\mathcal{Z}_{m_1,m_2}=\left\{(z_1,z_2) \in H^{1}(\Omega)\times L^{2}(\Omega)\ \big|\  \|  z_1 \|_{L^{\infty}} \le 1,\
	\overline{ z_1} = m_1,\ \overline{z_2}=m_2\right\}.
	$$
For all $(\varphi,\sigma)\in \mathcal{Z}_{m_1,m_2}$, we consider the energy functional
	\be
	\mathcal{F}(\varphi,\sigma)
=\int_{\Omega} \left( \frac{1}{2}|\nabla \varphi|^2+\varPsi(\varphi)+\frac{1}{2}|\sigma|^2
-\chi\sigma\varphi +\frac{\beta}{2}|\nabla\mathcal{N}(\varphi-\overline{\varphi})|^2\right) \mathrm{d}x.\label{fe}
	\ee
We say that a pair $(\varphi_{*},\sigma_{*})\in \mathcal{Z}_{m_1,m_2}$ is a local minimizer of $\mathcal{F}$ in $\mathcal{Z}_{m_1,m_2}$, if there exists a positive constant $\lambda_*$ such that for any $(\varphi,\sigma)\in \mathcal{Z}_{m_1,m_2}$ satisfying $\|\varphi-\varphi_*\|_{H^1}+\|\sigma-\sigma_*\|\le \lambda_*$, it holds
	\be
	\mathcal{F}(\varphi_*,\sigma_*)\le \mathcal{F}(\varphi,\sigma).\label{min}
	\ee
	If the inequality \eqref{min} is satisfied for all $(\varphi,\sigma)\in \mathcal{Z}_{m_1,m_2}$, we say that $(\varphi_{*},\sigma_{*})$ is a global minimizer of $\mathcal{F}$ in $\mathcal{Z}_{m_1,m_2}$.
\end{definition}

For any given constants $m_1\in (-1,1)$ and $m_2\in \mathbb{R}$, let us consider the following minimization problem:
\begin{align}
 \text{Find}\ (\varphi_*,\sigma_*)\in  \mathcal{Z}_{m_1,m_2}\quad \text{such that}\quad  \mathcal{F}(\varphi_*,\sigma_*) =\inf_{(\varphi,\sigma)\in  \mathcal{Z}_{m_1,m_2}} \mathcal{F}(\varphi,\sigma).
 \label{pro:min1}
\end{align}
By definition, $\mathcal{Z}_{m_1,m_2}$ is a closed subspace of $H^1(\Omega)\times L^2(\Omega)$. Under the hypothesis (H2), we can obtain the following result on the existence and regularity of a minimizer (see Section \ref{sec:stapro} for its proof).
\bp\label{prop-le2}
Let $m_1\in (-1,1)$ and $m_2,\,\beta,\,\chi \in \mathbb{R}$ be given constants. Moreover, we suppose that $\Omega$ is a bounded domain in $\mathbb{R}^3$ with boundary $\partial\Omega$ of class $C^3$ and (H2) is satisfied.

(1) Problem \eqref{pro:min1} has at least one (global) minimizer $(\varphi_*,\sigma_*)\in  \mathcal{Z}_{m_1,m_2}$.

 (2) Let $(\varphi_{*},\sigma_{*})\in  \mathcal{Z}_{m_1,m_2}$ be an energy minimizer. Then it  belongs to the class $H^2_N(\Omega)\times H^2_N(\Omega)$ and there exists a positive constant $\delta_*\in (0,1)$ such that
\be
	\|\varphi_*\|_{C(\overline{\Omega})}\le 1-\delta_*.
\label{ps1s}
\ee
\ep

%%%%%%%%%%%%%%%%%%%%%%%%%%%%
Based on Proposition \ref{prop-le2}, we now state the first main result in this paper.
\begin{theorem}[Global well-posedness near local energy minimizers] \label{3main}
 Let the hypotheses (H1)--(H4) be satisfied. Assume that   $(\varphi_{*},\sigma_{*})$ is a local minimizer of  $\mathcal{F}(\varphi,\sigma)$ in $\mathcal{Z}_{m_1,m_2}$ with $m_1=c_0$ and $m_2\in \mathbb{R}$ being given. For any $\epsilon>0$, there exist small constants $\eta_1,\, \eta_2,\, \eta_3 \in(0,1)$ such that for any initial data $\left(\boldsymbol{v}_{0}, \varphi_{0},\sigma_0\right)$ satisfying
	\begin{align*}
	&\boldsymbol{v}_{0} \in \bm{H}^1_{0,\mathrm{div}}(\Omega), \quad \varphi_{0} \in H^{3}(\Omega)\cap H^{2}_N(\Omega), \quad \sigma_{0} \in H^{1}(\Omega)\ \ \text{with}\ \   \overline{\sigma_0}=m_2, \\
	&\left\|\boldsymbol{v}_{0}\right\| \leq \eta_{1}, \quad\left\|\varphi_{0}-\varphi_{*}\right\|_{H^{2}} \leq \eta_{2}, \quad \|\sigma_{0}-\sigma_{*}\| \leq \eta_{3},
	\end{align*}
the initial boundary value problem \eqref{f3.c}--\eqref{ini0} admits a unique global strong solution $(\bm{v},\varphi,\mu,\sigma,p)$ such that
	\begin{align*}
	&\boldsymbol{v} \in C\left([0,+\infty) ;\bm{H}^1_{0,\mathrm{div}}(\Omega) \right) \cap L_{\mathrm{loc}}^{2}\left(0,+\infty ; \bm{H}^2(\Omega)\right) \cap H_{\mathrm{loc}}^{1}\left(0,+\infty ; \bm{L}^2_{0,\mathrm{div}}(\Omega)\right), \\
	&\varphi \in C\left([0,+\infty) ; H^{3}(\Omega)\right) \cap L_{\mathrm{loc}}^{2}\left(0,+\infty ; H^{4}(\Omega)\right) \cap H_{\mathrm{loc}}^{1}\left(0,+\infty ; H^{1}(\Omega)\right), \\
	&\mu \in C\left([0,+\infty) ; H^{1}(\Omega)\right) \cap L_{\mathrm{loc}}^{2}\left(0,+\infty ; H^{3}(\Omega)\right) \cap H_{\mathrm{loc}}^{1}\left(0,+\infty ;(H^{1}(\Omega))'\right), \\
	&\sigma \in C\left([0,+\infty) ; H^{1}(\Omega)\right) \cap L_{\mathrm{loc}}^{2}\left(0,+\infty ; H^{2}(\Omega)\right) \cap H_{\mathrm{loc}}^{1}\left(0,+\infty ; L^{2}(\Omega)\right), \\
	&p \in L_{\mathrm{loc}}^{2}\left(0,+\infty ; H^{1}(\Omega)\right).
	\end{align*}
The solution $(\bm{v},\varphi,\mu,\sigma,p)$ fulfills the system \eqref{f3.c}--\eqref{f2.b} almost everywhere in $\Omega\times (0,+\infty)$, the boundary conditions \eqref{boundary} almost everywhere on $\partial \Omega\times (0,+\infty)$ and the initial conditions \eqref{ini0} almost everywhere in $\Omega$.
Moreover, there exists $\delta\in (0,\delta_*)$ such that
\begin{align}
|\varphi(x,t)|\leq 1-\delta,\quad \forall\, (x,t)\in \overline{\Omega}\times [0,+\infty),
\label{glo-sep}
\end{align}
and
$(\varphi,\sigma)$ stays close to $(\varphi_{*},\sigma_{*})$ for all time, that is,
\begin{align}
	\|\varphi(t)-\varphi_{*}\|_{H^{2}} \leq \epsilon, \quad\|\sigma(t)-\sigma_{*}\| \leq \epsilon, \quad \forall\, t \geq 0.
\label{glo-lya}
\end{align}
\end{theorem}

Our second result characterizes the long-time behavior of the  global strong solution $(\bm{v},\varphi,\sigma)$.

\begin{theorem}[Long-time behavior]\label{3main1}
Suppose that the assumptions of Theorem \ref{3main} are satisfied. Let $(\bm{v},\varphi,\sigma)$ be a global strong solution obtained in Theorem \ref{3main}. Then the following results hold:

(1) Convergence to a single equilibrium: %
	\begin{equation}
	\lim_{t\to +\infty} \big(\|\bm{v}(t)\|_{\bm{H}^1} +\|\varphi(t)-\varphi_\infty\|_{H^{3}} +\|\sigma(t)-\sigma_\infty\|_{H^{1}}\big) =0.\label{convv}
	\end{equation}
Here, $(\varphi_{\infty},\sigma_{\infty})\in \big(H^3(\Omega)\cap H^2_N(\Omega)\big)\times H^2_N(\Omega)$ is a solution to the stationary problem such that
	\begin{subequations}
		\begin{alignat}{3}
		&-\Delta \varphi_\infty +\varPsi^{\prime}(\varphi_\infty) -\chi\sigma_\infty +\beta\mathcal{N}(\varphi_\infty-c_0) =\overline{\varPsi^{\prime}(\varphi_\infty)}
		-\chi\overline{\sigma_\infty},\quad
		&\ \text{in } \Omega,   \label{5bchv}\\
		& \Delta (\sigma_\infty-\chi\varphi_\infty)=0,
		&\ \text{in } \Omega,   \label{5f2.b}  \\
		&\partial_{\bm{n}} \varphi_\infty = \partial_{\bm{n}} \sigma_\infty=0,
		&\ \text{on } \partial \Omega, \label{5cchv}\\
		&\text{subject to the constraints}\quad \overline{\varphi_\infty}=c_0, \quad   \overline{\sigma_\infty}=\overline{\sigma_0}.
		& \label{5dchv}
		\end{alignat}
	\end{subequations}

(2) Estimate on the convergence rate: there exists some constant  $\kappa\in(0,1/2)$ depending on $(\varphi_{\infty},\sigma_{\infty})$ such that
\begin{align}
	\|\bm{v}(t)\|+ \|\varphi(t)-\varphi_\infty\|_{H^1} +\|\sigma(t)-\sigma_\infty\| \le C(1+t)^{-\frac{\kappa}{1-2 \kappa}},\quad \forall\,t\geq 0,
\label{con-rate}
\end{align}
where $C>0$ depends on $\|\bm{v}_0\|_{\bm{H}^1}$, $\|\varphi_0\|_{H^3}$, $\|\varphi_\infty\|_{H^2}$, $\|\sigma_0\|_{H^1}$, $\|\sigma_\infty\|_{H^1}$, $\delta_*$, coefficients of the system and $\Omega$.
\end{theorem}

\begin{remark}\label{rem:eq}
Theorem \ref{3main} yields the Lyapunov stability for every local energy minimizer of the free energy $\mathcal{F}$. Theorem \ref{3main1} further implies that for any global strong solution $(\bm{v}, \varphi, \sigma)$ obtained in Theorem \ref{3main}, $(\varphi, \sigma)$ not only stays close to the given local energy minimizer $(\varphi_*, \sigma_*)$, but also converges to an equilibrium $(\varphi_\infty, \sigma_\infty)$ as $t\to+\infty$. In general, $(\varphi_*, \sigma_*)$ and $(\varphi_\infty, \sigma_\infty)$ do not coincide. Nevertheless, if one can show that $(\varphi_*, \sigma_*)$ is an isolated minimizer for $\mathcal{F}$, then $(\varphi_\infty, \sigma_\infty)=(\varphi_*, \sigma_*)$, namely, the minimizer $(\varphi_*, \sigma_*)$ is (locally) asymptotically stable.
\end{remark}

%%%%%%%%%%%%%%%%%%%%%%%%%%%%%%%%%%%%
 \section{Local Well-posedness}\label{gss}
 \setcounter{equation}{0}
 In this section, we establish the existence and uniqueness of a local strong solution to problem \eqref{f3.c}--\eqref{ini0} subject to arbitrarily large and sufficiently regular initial data.
 \begin{theorem}[Local well-posedness] \label{ls}  Let the hypotheses (H1), (H2) and (H4) be satisfied.

 (1) For any initial data $\left(\boldsymbol{v}_{0}, \varphi_{0},\sigma_0\right)$ satisfying
 $\boldsymbol{v}_{0} \in \bm{H}^1_{0,\mathrm{div}}(\Omega)$, $\varphi_{0} \in H^{2}_N(\Omega)$, $\|\varphi_0\|_{L^\infty}\leq 1$, $|\overline{\varphi_0}|<1$, $\mu_0=-\Delta \varphi_0+ \varPsi'(\varphi_0)\in H^1(\Omega)$ and $\sigma_{0} \in H^{1}(\Omega)$, there exists some $T_0\in (0,+\infty)$, depending on $\left\|\boldsymbol{v}_{0}\right\|_{\boldsymbol{H}^{1}}$,
 $\|\varphi_0\|_{H^1}$, $\left\|\mu_0\right\|_{H^{1}}$,  $\left\|\sigma_{0}\right\|_{H^{1}}$, $\max_{r\in[-1,1]}|\varPsi(r)|$, coefficients of the system and $\Omega$, such that problem \eqref{f3.c}--\eqref{ini0} admits a unique local strong solution $(\bm{v},\varphi,\mu,\sigma,p)$ on $[0,T_0]$,  satisfying the following regularity properties
 \begin{align*}
 &\boldsymbol{v} \in C\big([0,T_0] ;\bm{H}^1_{0,\mathrm{div}}(\Omega) \big) \cap L^{2}\big(0,T_0 ; \bm{H}^2(\Omega)\big) \cap H^{1}\big(0,T_0 ; \bm{L}^2_{0,\mathrm{div}}(\Omega)\big), \\
 &\varphi \in L^\infty\big(0,T_0 ; W^{2,6}(\Omega)\big),\quad \partial_t \varphi \in L^{\infty}\big(0,T_0 ; (H^{1}(\Omega))'\big) \cap L^2\big(0,T_0 ; H^{1}(\Omega)\big), \\
 & \varphi\in L^\infty(\Omega\times(0,T_0))\ \ \text{such that}\ \  |\varphi(x,t)|<1\ \ \text{a.e. in}\ \ \Omega \times(0,T_0),\\
 &\mu \in L^\infty\big(0,T_0; H^{1}(\Omega)\big) \cap L^{2}\big(0,T_0 ; H^{3}(\Omega)\big), \\
 	&\sigma \in C\big([0,T_0] ; H^{1}(\Omega)\big) \cap L^{2}\big(0,T_0 ; H^{2}(\Omega)\big) \cap H^{1}\big(0,T_0 ; L^{2}(\Omega)\big), \\
 &p \in L^{2}\big(0,T_0 ; V_0\big).
 \end{align*}
 Besides, the solution $(\bm{v},\varphi,\mu,\sigma,p)$ fulfills the system \eqref{f3.c}--\eqref{f2.b} almost everywhere in $\Omega\times (0,T_0)$, the boundary conditions \eqref{boundary} almost everywhere on $\partial\Omega\times (0,T_0)$ and the initial conditions \eqref{ini0} almost everywhere in $\Omega$.

 (2) Assume in addition, $\left\|\varphi_{0}\right\|_{C(\overline{\Omega})} = 1-\delta_{0}$ for some given $\delta_{0} \in(0,1)$. The local strong solution obtained in (1) satisfies the further regularity
 \begin{align*}
 &\varphi \in C\big([0,T^{*}] ; H^{3}(\Omega)\big) \cap L^{2}\big(0,T^{*} ; H^{4}(\Omega)\big), \\
 &\mu \in C\big([0,T^{*}]; H^{1}(\Omega)\big) \cap H^{1}\big(0,T^{*} ;(H^{1}(\Omega))^{\prime}\big),
 \end{align*}
 and the strict separation property
 \be
 \|\varphi(t)\|_{C(\overline{\Omega})} \leq 1-\frac{1}{2} \delta_{0}, \quad \forall\, t \in [0, T^{*}],
 \ee
 for some time $T^{*} \in(0,T_0)$ depending on $\left\|\boldsymbol{v}_{0}\right\|_{\boldsymbol{H}^{1}}$,
 $\|\varphi_0\|_{H^1}$, $\left\|\mu_0\right\|_{H^{1}}$,  $\left\|\sigma_{0}\right\|_{H^{1}}$,  $\max_{r\in[-1,1]}|\varPsi(r)|$, coefficients of the system, $\Omega$ and $\delta_{0}$.
 \end{theorem}

\subsection{The approximate problem}\label{app-ini}
The existence part of Theorem \ref{ls} can be proved by using a semi-Galerkin scheme similar to that in \cite[Section 3]{H} for global weak solutions to a Navier--Stokes--Cahn--Hilliard system with
chemotaxis and singular potential. However, in order to derive higher-order Sobolev estimates for the approximate solutions in three dimensions, we need to adopt proper regularizations for the initial datum $(\varphi_0,\sigma_0)$ (cf. \cite{GMT,Gio2022}) and a viscous regularization for the Cahn--Hilliard equation (cf. \cite{Gio2022,H2}). When the spatial dimension is two, thanks to a global-in-time strict separation property for the phase-field function $\varphi$, the above mentioned regularizations on the initial datum and the equation can be skipped. We refer to \cite{H1} for further discussions in this aspect.

\medskip

  \textbf{Regularization of the initial datum}. Define
  $$
  \widetilde{\mu}_0=-\Delta \varphi_0+ \varPsi_0'(\varphi_0).
  $$
  As in \cite{GMT}, for any integer $k\geq 1$, we consider the cut-off $\widetilde{\mu}_{0,k}=h_k\circ \widetilde{\mu}_{0}$, where $h_k$ is a globally Lipschitz continuous function given by
  \begin{equation*}
  h_k(r)=\begin{cases}
  k,\quad\ \ \ \, r>k,\\
  r,\qquad   r\in[-k, k],\\
  -k,\quad \, r<-k.
  \end{cases}
  \end{equation*}
  Since $\mu_0\in H^1(\Omega)$ and $\varphi_0\in H^2_N(\Omega)$, it is straightforward to check that $$\widetilde{\mu}_0,\ \widetilde{\mu}_{0,k}\in H^1(\Omega),\quad  \|\widetilde{\mu}_{0,k}\|_{H^1}\leq \|\widetilde{\mu}_{0}\|_{H^1}\quad \text{and}\quad  \|\widetilde{\mu}_{0,k}-\widetilde{\mu}_{0}\|\to 0\quad \text{as}\quad k\to+\infty.
   $$
  Next, we define $\varphi_{0,k}$ as the unique solution to the nonlinear elliptic problem
  $$
  \begin{cases}
  -\Delta \varphi_{0,k} +\varPsi_0'(\varphi_{0,k})=\widetilde{\mu}_{0,k} \quad \ \  \text{in}\ \Omega, \\
   \partial_{\bm{n}}\varphi_{0,k} =0\qquad\qquad\qquad \qquad  \text{on}\ \partial \Omega.
   \end{cases}
  $$
  From \cite[Lemma A.1]{GMT} we infer that  $\varphi_{0,k}\in H^2_N(\Omega)$, $\varPsi_0'(\varphi_{0,k})\in L^2(\Omega)$, and moreover, it holds
   $$\|\varphi_{0,k}\|_{H^2}+\|\varPsi_0'(\varphi_{0,k})\|\leq C(1+\|\widetilde{\mu}_0\|)\quad \text{and}\quad  \|\varphi_{0,k}-\varphi_0\|_{H^1}\to 0\quad\text{as}\quad k\to +\infty.
    $$
  Thus, there exists an integer $\widehat{k}\gg 1$ such that
  $$
  \|\varphi_{0,k}\|_{H^1}\leq 1+\|\varphi_{0}\|_{H^1} \quad
  \text{and}\quad |\overline{\varphi_{0,k}}|\leq \frac{1+|\overline{\varphi_{0}}|}{2}<1,\qquad \forall\, k\geq \widehat{k}.
  $$
  On the other hand, from \cite[Lemma A.1]{CG}, we find
  $$
  \|\varPsi_0'(\varphi_{0,k})\|_{L^\infty}\leq \|\widetilde{\mu}_{0,k}\|_{L^\infty}\leq k,
  $$
  which together with (H2) further implies that
  $\|\varphi_{0,k}\|_{C(\overline{\Omega})}\leq 1-\delta_k$ for some constant  $\delta_k\in (0,1)$ depending on $k$. From the standard elliptic estimate, we also get $\varphi_{0,k}\in H^3(\Omega)$.

  Concerning the initial datum $\sigma_0$, we simply take $\sigma_{0,k}$ as the unique solution of
  $$
  \begin{cases}
  -\displaystyle{\frac{1}{k}} \Delta \sigma_{0,k} + \sigma_{0,k} =\sigma_{0}\qquad \text{in}\ \Omega, \\
   \partial_{\bm{n}}\sigma_{0,k} =0\qquad\qquad \qquad\ \ \,   \text{on}\ \partial \Omega.
   \end{cases}
  $$
  Since $\sigma_0\in H^1(\Omega)$, then it easily follows that  $\sigma_{0,k}\in H^2_N(\Omega)\cap H^3(\Omega)$, $\overline{\sigma_{0,k}}=\overline{\sigma_{0}}$ and  $\|\sigma_{0,k}\|_{H^1}\leq \|\sigma_{0}\|_{H^1}$ for all $k\geq 1$. Besides, we have
  $\|\sigma_{0,k}-\sigma_0\|\to 0$ as $k\to+\infty$. Hence, $\|\sigma_{0,k}\|\leq 1+\|\sigma_{0}\|$ holds for all $k\geq \widehat{k}$, provided that $\widehat{k}$ is large enough.
  \medskip

 \textbf{The semi-Galerkin scheme}. We perform a finite dimensional approximation of the Navier--Stokes equations \eqref{f3.c}. Let the family $\{\bm{y}_{i}\}_{i=1}^{\infty}$ be a basis of the Hilbert space $\bm{H}^1_{0,\mathrm{div}}(\Omega)$, which is given by eigenfunctions of the Stokes problem
  \be
  (\nabla \bm{y}_{i},\nabla \bm{w})=\lambda_{i}(\bm{y}_{i},\bm{w}),\quad  \forall\, \bm{w} \in {\bm{H}^1_{0,\mathrm{div}}(\Omega)},\quad \textrm{with}\ \|\bm{y}_{i}\|=1.\nonumber
  \ee
  Here, $\lambda_{i}$ is the eigenvalue corresponding to $\bm{y}_{i}$.  It is well-known that $0<\lambda_{1}< \lambda_{2}<\cdots $ is an unbounded monotonically increasing sequence, $\{\bm{y}_{i}\}_{i=1}^{\infty}$ forms a complete orthonormal basis in $\bm{L}^2_{0,\mathrm{div}}(\Omega)$ and it is also orthogonal in $\bm{H}^1_{0,\mathrm{div}}(\Omega)$.
  For every integer $m\geq 1$, we denote by
  $\bm{H}_{m}:=\textrm{span} \{\bm{y}_{1},\cdots,\bm{y}_{m}\}$ the finite-dimensional subspace of $\bm{L}^2_{0,\mathrm{div}}(\Omega)$.
  Moreover, we use $\bm{P}_{\bm{H}_m}$ for the corresponding orthogonal projections from $\bm{L}^2_{0,\mathrm{div}}(\Omega)$ onto $\bm{H}_m$.
  Thanks to (H4), the regularity theory of the Stokes
operator yields that $\bm{y}_{i} \in \bm{H}^4(\Omega)\cap \bm{H}^1_{0,\mathrm{div}}(\Omega)$ for all $i\geq 1$.

  Let $T>0$ be a given final time. For any $\gamma\in (0,1)$ and integers $m, k\geq 1$, we consider the following semi-Galerkin scheme for a regularized version of the original problem \eqref{f3.c}--\eqref{ini0}:
  $$
  \text{looking for functions} \quad\bm{v}^{m}(x,t):=\sum_{i=1}^{m}a_{i}^{m}(t)\bm{y}_{i}(x) \quad \text{and} \quad (\varphi^{m},\, \mu^{m},\, \sigma^{m}),
  $$
  which satisfy the equations
  \begin{equation}
  \begin{cases}
  (\partial_t  \bm{ v}^{m},\bm{w})+( \bm{ v}^{m} \cdot \nabla  \bm {v}^{m},\bm{ w})+ \big(  2\nu(\varphi^{m}) D\bm{v}^{m},\nabla \bm{w}\big)  =\big((\mu^{m}+\chi \sigma^{m})\nabla \varphi^{m},\bm {w}\big), \\
  \qquad \qquad \text{for all } \bm{w} \in \bm{H}_{m}\ \text{and almost all } t\in (0,T),\\
  \partial_t \varphi^m+ \bm{v}^m \cdot \nabla \varphi^m =\Delta  \mu^{m}-\alpha(\overline{\varphi^m}-c_0), \qquad \ \ \qquad \qquad \qquad \ \ \text{a.e. in }\ \ \Omega\times(0,T),\\
  \mu^{m}=\gamma\partial_t \varphi^m-\Delta \varphi^{m}+\varPsi'(\varphi^{m})-\chi \sigma^{m}+\beta\mathcal{N}(\varphi^m-\overline{\varphi^m}),   \ \quad  \quad    \text{a.e. in }\ \ \Omega\times(0,T),\\
  \partial_t \sigma^m+\bm{v}^m \cdot \nabla \sigma^m =\Delta (\sigma^m-\chi\varphi^m),  \qquad\qquad\qquad\quad\quad\quad\quad\quad \ \text{a.e. in }\ \ \Omega\times(0,T),
  \end{cases}
  \label{sapp1}
  \end{equation}
  as well as the boundary and initial conditions
  \begin{equation}
  \begin{cases}
  \bm{v}^m=\mathbf{0},\quad {\partial}_{\bm{n}}\varphi^m ={\partial}_{\bm{n}}\mu^m={\partial}_{\bm{n}}\sigma^m=0, \qquad\qquad\quad \quad \ \ \textrm{on} \   \partial\Omega\times(0,T),
  \\
  \bm{v}^{m}|_{t=0}=\bm{P}_{\bm{H}_{m}} \bm{v}_{0},\quad \varphi^{m}|_{t=0}=\varphi_{0,k},\quad  \sigma^{m}|_{t=0}=\sigma_{0,k},\quad \text{in}\ \Omega.
  \end{cases}
  \label{sapp2}
  \end{equation}
  For simplicity of the notation, we do not explicitly write down the approximating parameters $k$ and $\gamma$ in the three-level approximation scheme \eqref{sapp1}--\eqref{sapp2}.

   We can establish the following result on the existence and uniqueness of a local strong solution $(\bm{v}^m,\varphi^m,\mu^m,\sigma^m)$ to problem \eqref{sapp1}--\eqref{sapp2}.

\begin{proposition}\label{exe-app}
Suppose that the hypotheses (H1), (H2) and (H4) are satisfied. For any given $\gamma\in(0,1)$ and integers $m, k\geq 1$,  there exists a time $T_{m}\in (0,T]$ depending on $\|\bm{v}_{0}\|_{\bm{H}^1}$, $\|\varphi_{0,k}\|_{H^3}$, $\|\widetilde{\mu}_0\|_{H^1}$, $\|\sigma_{0,k}\|_{H^2}$, $\overline{\varphi_0}$, $\overline{\sigma_0}$, $\Omega$, and the coefficients of the system such that problem \eqref{sapp1}--\eqref{sapp2} admits a unique local strong solution $(\bm{v}^{m},\varphi^{m},\mu^{m},\sigma^{m})$ on $[0,T_{m}]$ satisfying
\begin{align}
&\bm{v}^{m} \in H^1(0,T_{m};\bm{H}_m(\Omega))\notag,\\
&\varphi^{m} \in L^\infty(0,T_{m};H^3(\Omega))\cap L^{2}(0,T_{m};H^4(\Omega)) \notag  \\
&\partial_t\varphi^m\in L^\infty(0,T_m,H^1(\Omega))\cap L^2(0,T_{m};H^2(\Omega))\cap H^1(0,T_m;L^2(\Omega)),\notag\\
&\mu^{m} \in L^{\infty}(0,T_{m};H^2(\Omega))\cap H^{1}(0,T_{m};L^2(\Omega)),\notag \\
&\sigma^{m} \in L^\infty(0,T_{m};H^2(\Omega))\cap L^2(0,T_{m};H^3(\Omega)), \notag  \\
&\partial_t\sigma^m \in L^{\infty}(0,T_{m};L^2(\Omega)) \cap L^2(0,T_{m};H^1(\Omega))\cap H^1(0,T_m;(H^1(\Omega))').\notag
\end{align}
Moreover, $|\varphi^m(x,t)|<1-\delta(\gamma,m,k)$ for all $(x,t)\in \overline{\Omega}\times[0,T_m]$,
where the constant $\delta(\gamma,m,k)\in (0,\delta_k)$ may depend on the parameters $\gamma$, $m$ and $k$.
\end{proposition}

\begin{remark}
Proposition \ref{exe-app} can be proved by using Schauder's fixed point argument similar to that for \cite[Proposition 3.1]{H}. We leave the detailed proof to the readers.
\end{remark}

\subsection{Uniform estimates}
For the approximation parameters
$$
\gamma\in (0,1),\quad  m\geq 1,\quad k\geq \widehat{k},
$$
we proceed to derive uniform estimates of the corresponding approximate solutions $(\bm{v}^{m},\varphi^{m},\mu^{m},\sigma^{m})$ obtained in Proposition \ref{exe-app}. In what follows, we denote by $C$, $C_i$ generic positive constants that are independent of  $\gamma$, $m$ and $k$.

As a starting point, we infer from Proposition \ref{exe-app} that
	\be
	\sup_{t\in[0,T_m]}\|\varphi^m(t)\|_{L^\infty}\leq 1.\label{inif}
	\ee
The uniform estimate \eqref{inif} turns out to be crucial in the subsequent analysis. Here, we are not able to make use of the strict separation property for $\varphi^m$, since the distance $\delta(\gamma,m,k)$ depends on $\gamma$, $m$, $k$ and it may tend to zero as $\gamma\to 0$ or $k\to +\infty$.
\medskip

\textbf{First estimate: control of the mass}. Integrating \eqref{sapp1}$_2$ over $\Omega$ and using the boundary conditions in \eqref{sapp2}, we obtain (cf. \cite{GGM2017,H})
	\be
	\frac{\mathrm{d}}{\mathrm{d}t} \big(\overline{\varphi^m}(t)-c_0\big) +\alpha\big(\overline{\varphi^m}(t)-c_0\big)=0.
	\label{conver1}
	\ee
Integrating \eqref{conver1} with respect to time, we easily find that
	\be
	\overline{\varphi^m}(t) = c_0+e^{-\alpha t}\left(\overline{\varphi_{0,k}}-c_0\right),\quad\forall\,t\in [0,T_m].\label{aver-phi}
	\ee
Thus, the mass of $\varphi^m$ is conserved if $\alpha =0$ or $\overline{\varphi_{0,k}}=c_0$. When $\alpha>0$, $\overline{\varphi^m}(t)$ converges exponentially fast to $c_0$ as time increases. The above observations imply that for $k\geq \widehat{k}$,
\be
|\overline{\varphi^m}(t)|\leq C(|\overline{\varphi_0}|,c_0)<1,\qquad \forall\,t\in [0,T_m],\label{aver-phi2}
\ee
that is, independent of $\gamma$, $m$ and $k$. Similarly, integrating \eqref{sapp1}$_4$ over $\Omega$, we get
 	\be
	\frac{\mathrm{d}}{\mathrm{d}t}\int_\Omega \sigma^m\, \mathrm{d}x =0,\notag
	\ee
which yields
	\be
	\overline{\sigma^m}(t)=\overline{\sigma_{0}},\quad \forall\, t\in [0,T_m].
	\label{aver-sig}
	\ee

\textbf{Second estimate: lower-order energy estimate}.
Multiplying \eqref{sapp1}$_1$ by $\bm{v}^m$, \eqref{sapp1}$_2$ by $\mu^m$, \eqref{sapp1}$_3$ by  $\partial_t \varphi^m$ and \eqref{sapp1}$_4$ by $\sigma^m-\chi\varphi^m$, integrating over $\Omega$ and adding the resultants together, we obtain the energy identity
	\begin{align}
	\frac{\mathrm{d}}{\mathrm{d}t}\mathcal{E}_m(t)
	+\mathcal{D}_m(t)
&= -\alpha\big(\overline{\varphi^m}(t)-c_0\big)\int_\Omega \mu^m(t) \,\mathrm{d}x,
	\label{BEL2}
	\end{align}
with
	\begin{align}
	\mathcal{E}_m(t)&=\frac{1}{2}\|\bm{v}^m(t)\|^2
    + \frac{1}{2}\|\nabla \varphi^m(t)\|^2
	+ \int_\Omega \varPsi(\varphi^m(t))\, \mathrm{d}x
	+\frac{1}{2}\|\sigma^m(t)\|^2\notag\\
	&\quad
	-\int_\Omega \chi\sigma^m(t)\varphi^m(t)\, \mathrm{d}x +\frac{\beta}{2}\|\nabla\mathcal{N}\big(\varphi^m(t)-\overline{\varphi^m(t)}\big)\|^2,
\label{E}\\
	\mathcal{D}_m(t)& =\int_\Omega 2\nu(\varphi^m(t))|D\bm{v}^m(t)|^2\, \mathrm{d}x + \|\nabla \mu^m(t)\|^2+\|\nabla\big(\sigma^m(t)-\chi\varphi^m(t)\big)\|^2\notag\\
&\quad +\gamma\|\partial_t\varphi^m(t)\|^2.
\label{D}
	\end{align}
From the construction of the approximated initial data (recall Section \ref{app-ini}), (H2) and H\"{o}lder's inequality, we see that the initial energy can be estimated by
\begin{align}
\mathcal{E}_m(0)&\leq \frac12\|\bm{v}_0\|^2
+ \frac{1}{2}\left(1+\|\varphi_0\|_{H^1}^2\right)
+ |\Omega|\max_{r\in[-1,1]}|\varPsi(r)| \notag\\
&\quad + \frac12(1+\chi^2)\|\sigma_0\|^2+ C(|\beta|+1)|\Omega|,
\label{iniE1}
\end{align}
where $C$ only depends on $\Omega$. On the other hand, it follows from \eqref{inif} that
\begin{align}
\left|\int_\Omega \chi\sigma^m(t)\varphi^m(t)\, \mathrm{d}x\right|\leq |\chi|\|\sigma^m(t)\|\|\varphi^m(t)\|\leq \frac14\|\sigma^m(t)\|^2+ \chi^2|\Omega|,\quad \forall\,
t\in [0,T_m].
\label{couple1}
\end{align}
According to (H2), we can obtain the uniform lower bound
\begin{align}
\mathcal{E}_m(t)&\geq \frac{1}{2}\|\bm{v}^m(t)\|^2
	+ \frac{1}{2}\|\varphi^m(t)\|_{H^1}^2
	+\frac{1}{4}\|\sigma^m(t)\|^2 - C_1, \quad \forall\,
t\in [0,T_m].
\label{low-bd1}
\end{align}
where $C_1$ only depends on $\chi$, $\beta$ and $\Omega$.
 	
We proceed to estimate the right-hand side of \eqref{BEL2}.  Integrating \eqref{sapp1}$_3$ over $\Omega$, we infer from (H2), \eqref{inif}--\eqref{aver-phi} and \eqref{aver-sig} that
	\begin{align}
	\left|\int_\Omega \mu^m(t)\,\mathrm{d}x\right|
	&= \left|\int_\Omega\big(\gamma\partial_t \varphi^m(t)+ \varPsi'(\varphi^m(t)) -\chi \sigma^m(t)\big) \,\mathrm{d}x\right| \notag \\
	& \le  \| \varPsi_0'(\varphi^m(t))\|_{L^1}  +|\Omega|\theta_0 +|\chi|\left|\int_\Omega \sigma_0\,\mathrm{d}x\right|  + \gamma \alpha \left|\int_\Omega(\varphi_{0,k}-c_0)\,\mathrm{d}x\right|\notag\\
&\leq  \| \varPsi_0'(\varphi^m(t))\|_{L^1}  +C, \qquad \forall\, t \in [0,T_m],
	\label{average valuea}
	\end{align}
where $C$ depends on the coefficients of the system, $\overline{\sigma_0}$ and $\Omega$.
Thanks to the assumption (H2), the term $\| \varPsi'_0(\varphi^m)\|_{L^1}$ can be controlled as follows (cf. \cite[Proposition A.1]{MZ04})
$$
\| \varPsi_0'(\varphi^m)\|_{L^1}\leq C\int_\Omega \varPsi_0'(\varphi^m)(\varphi^m-\overline{\varphi^m}) \,\mathrm{d}x+C,
$$
where $C$ depends on $\overline{\varphi^m}$ (thus only on $\overline{\varphi_{0}}$, $c_0$ in view of \eqref{aver-phi2}) and $\Omega$. Next, multiplying \eqref{sapp1}$_3$ by $\varphi^m-\overline{\varphi^m}$ and integrating over $\Omega$, we obtain
\begin{align*}
& \int_\Omega \varPsi_0'(\varphi^m)(\varphi^m-\overline{\varphi^m}) \,\mathrm{d}x + \int_\Omega |\nabla \varphi^m|^2\mathrm{d}x \\
&\quad = \int_\Omega (\mu^m-\overline{\mu^m})(\varphi^m-\overline{\varphi^m})\, \mathrm{d}x  - \gamma\int_\Omega \partial_t \varphi^m (\varphi^m-\overline{\varphi^m})\, \mathrm{d}x \\
&\qquad + \int_\Omega \big(\chi (\sigma^{m} - \overline{\sigma^{m}}) -\beta\mathcal{N}(\varphi^m-\overline{\varphi^m})+
\theta_0 (\varphi^m- \overline{\varphi^m})\big)(\varphi^m-\overline{\varphi^m}) \,\mathrm{d}x.
\end{align*}
By H\"{o}lder's inequality and the Poincar\'{e}--Wirtinger inequality, we find that
	\begin{align}
	\| \varPsi'_0(\varphi^m)\|_{L^1}
	&\le C\big(\|  \nabla \mu^m\| + \|\sigma^m-\overline{\sigma^m}\|+ \| \varphi^m - \overline{\varphi^m}\|+ \gamma\|\partial_t \varphi^m \|\big)\|\nabla \varphi^m\|+C\notag \\
&\leq C\big(\|  \nabla \mu^m\|  + \|\nabla \sigma^m\|
+ \|\nabla \varphi^m\|+\gamma\|\partial_t \varphi^m \|\big)\|\nabla \varphi^m\|+C,
	\label{varPsi}
	\end{align}
 where $C$ depends on the coefficients of the system,  $\overline{\varphi_0}$, $\overline{\sigma_0}$ and $\Omega$.
 On the other hand, it follows from the boundary condition $\partial_{\bm{n}}\varphi^m=0$ on $\partial \Omega$, the  Cauchy--Schwarz inequality and \eqref{inif} that
	\be
	\|\nabla \varphi^m\|^2\leq \|\Delta \varphi^m\|\|\varphi^m\|\leq a\|\Delta\varphi^m\|^2+\frac{1}{4a}|\Omega|,\quad \forall\, a >0.\label{q1}
	\ee
Now multiplying \eqref{sapp1}$_3$ by $-\Delta \varphi^m$, integrating over $\Omega$, we infer from  \eqref{q1} (with some properly chosen $a$)  that
\begin{align}
& \int_\Omega \varPsi_0''(\varphi^m)|\nabla \varphi^m|^2\, \mathrm{d}x + \|\Delta \varphi^m\|^2
\notag
\\
&\quad =
\int_\Omega \nabla \mu^m\cdot \nabla \varphi^m\, \mathrm{d}x  +\gamma\int_\Omega \partial_t\varphi^m\Delta \varphi^m\,\mathrm{d}x\notag\\
&\qquad  + \int_\Omega \nabla  \big(\chi (\sigma^{m}- \overline{\sigma^{m}}) -\beta\mathcal{N}(\varphi^m-\overline{\varphi^m}) +
\theta_0 (\varphi^m-\overline{\varphi^m})\big)\cdot\nabla \varphi^m \,\mathrm{d}x \notag\\
&\quad \leq \|\nabla \mu^m\|\| \nabla \varphi^m\|+
C(\|\nabla \sigma^m\| +\|\nabla \varphi^m\|)\|\nabla \varphi^m\| +\gamma\|\partial_t\varphi^m\|\|\Delta \varphi^m\|\notag \\
&\quad \leq \frac{\gamma}{4}\|\Delta \varphi^m\|^2+ \|\nabla \mu^m \|^2 + \|\nabla \sigma^m\|^2+C\|\nabla \varphi^m\|^2+ \gamma \|\partial_t\varphi^m\|^2
\notag \\
&\quad \leq \frac{1}{4}\|\Delta \varphi^m\|^2+ \|\nabla \mu^m \|^2+ \|\nabla(\sigma^m-\chi\varphi^m)\|^2 + C\|\nabla \varphi^m\|^2 + \gamma \|\partial_t\varphi^m\|^2\notag\\
&\quad \leq \frac{1}{2}\|\Delta \varphi^m\|^2 + \|\nabla \mu^m \|^2+ \|\nabla(\sigma^m-\chi\varphi^m)\|^2+ \gamma \|\partial_t\varphi^m\|^2+C.
\label{esH2}
\end{align}
Combining the estimates \eqref{average valuea}--\eqref{esH2} and and Young's inequality, we assert that
	\begin{align}
		\left|\int_\Omega \mu^m\,\mathrm{d}x\right|
&\leq  C\big(\|  \nabla \mu^m\|  + \|\nabla \sigma^m\|
+ \|\nabla \varphi^m\|+\gamma\|\partial_t \varphi^m \|\big)\|\nabla \varphi^m\|+C\notag\\
&\leq \frac{a}{2} \big(\|\nabla \mu^m\|^2+\|\nabla(\sigma^m-\chi\varphi^m)\|^2+\gamma\|\partial_t \varphi^m \|^2\big) \notag\\
&\quad + C(1+a^{-1})\|\nabla \varphi^m\|^2+C\notag\\
& \le a \big(\|\nabla \mu^m\|^2+\|\nabla(\sigma^m-\chi\varphi^m)\|^2+\gamma\|\partial_t \varphi^m \|^2\big)+C\left[\frac{(1+a)^2}{a^3}+1\right],
	\label{q2}
	\end{align}
for any $a\in (0,1)$, where the positive constant $C$ depends on the coefficients of the system, $\overline{\varphi_0}$, $\overline{\sigma_0}$, and $\Omega$, but it is independent of $a$. 	

When $\alpha=0$, we simply have $-\alpha(\overline{\varphi^m}-c_0)\int_\Omega \mu^m \, \mathrm{d}x=0$. When $\alpha> 0$, in view of \eqref{aver-phi} and \eqref{BEL2}, we can choose $a=1/4\alpha$ in  \eqref{q2} to conclude that
	\begin{align}
	\left|\alpha(\overline{\varphi^m}-c_0)\int_\Omega \mu^m \, \mathrm{d}x\right|
& \leq  2\alpha a \big(\|\nabla \mu^m\|^2+\|\nabla(\sigma^m-\chi\varphi^m)\|^2+\gamma\|\partial_t \varphi^m \|^2\big)\notag\\
&\quad + C\left[\frac{(1+a)^2}{a^3}+1\right] \alpha  e^{-\alpha t}|\overline{\varphi_{0,k}}-c_0|\notag \\
& \le \frac12\left(\|\nabla \mu^m\|^2 +\|\nabla(\sigma^m-\chi\varphi^m)\|^2+\gamma\|\partial_t \varphi^m \|^2\right) \notag\\
&\quad  +C_2\big(1+\alpha^3\big)\alpha  e^{-\alpha t}|\overline{\varphi_{0,k}}-c_0|,
\label{q3}
	\end{align}
where $C_2$  depends on $\overline{\varphi_0}$, $\overline{\sigma_0}$, $\Omega$ and the coefficients of the system except $\alpha$.

Define
	\begin{align}
	&\widetilde{\mathcal{E}}_m(t)=\mathcal{E}_m(t)+\widetilde{C}_2  e^{-\alpha t}|\overline{\varphi_{0,k}}-c_0|,\label{EE}
	\end{align}
with $\widetilde{C}_2=C_2(1+\alpha^3)\alpha>0$.
From \eqref{BEL2} and \eqref{q3}, we deduce that
	\begin{align} &\frac{\mathrm{d}}{\mathrm{d}t}\widetilde{\mathcal{E}}_m(t)
	+\frac12\mathcal{D}_m(t) \leq  0,\quad \forall\, t\in (0,T_m).
	\label{BEL4}
	\end{align}
An integration in time on $[0,t]$, with $t\in (0,T_m]$, yields
\begin{align}
\widetilde{\mathcal{E}}_m(t)+ \frac12\int_0^t\mathcal{D}_m(\tau)\,\mathrm{d}\tau\leq  \widetilde{\mathcal{E}}_m(0), \quad \forall\, t\in (0,T_m].
\label{low-es1}
\end{align}
In view of the definitions of $\widetilde{\mathcal{E}}_m(t)$ and $\mathcal{D}_m(t)$, we conclude from \eqref{iniE1}, \eqref{low-bd1}, \eqref{low-es1} and Korn's inequality that
	\begin{align}
	&\|\bm{v}^m(t)\|^2 +\|\varphi^m(t)\|_{H^1}^2+\|\sigma^m(t)\|^2 \notag \\
	&\qquad +\int_{0}^{t}\big(\|\nabla\bm{v}^m(\tau)\|^2 + \|\nabla \mu^m(\tau)\|^2+ \|\nabla (\sigma^m-\chi\varphi^m)(\tau)\|^2+\gamma\|\partial_t \varphi^m (\tau)\|^2\big)\,  \mathrm{d} \tau \notag \\
	& \quad\le C,\quad \forall\, t\in [0,T_m].
	\label{low-es2}
	\end{align}
 From \eqref{q2}, \eqref{low-es2}  and the Poincar\'{e}--Wirtinger inequality, we also find
\be
\int_{0}^{t}\big(\| \mu^m(\tau)\|^{2}_{H^1}+ \|\sigma^m(\tau)\|_{H^1}^2 \big) \, \mathrm{d} \tau \leq C,\quad \forall\, t\in [0,T_m]. \label{low-es3}
\ee
The positive constant $C$ in \eqref{low-es2} and \eqref{low-es3} depends on the coefficients of the system, $\Omega$, $\overline{\varphi_0}$, $\overline{\sigma_0}$, $\max_{r\in[-1,1]}|\varPsi(r)|$,  $\|\bm{v}_0\|$, $\|\varphi_0\|_{H^1}$ and $\|\sigma_0\|$.
\medskip

\textbf{Third estimate: higher-order energy estimate}.
Since $\bm{v}_m\in \bm{H}_m$, $\bm{v}_m$ and $\partial_t\bm{v}_m$ have sufficient spatial regularity to be used as test functions. Testing \eqref{sapp1}$_1$ with $\partial_{t} \bm{v}^m$, we obtain
\begin{align}
\|\partial_{t} \bm{v}^m\|^2
& = -( \bm{ v}^{m} \cdot \nabla  \bm {v}^{m},\partial_{t} \bm{v}^m)+ 2\big(  \mathrm{div}(\nu(\varphi^{m}) D\bm{v}^{m}),\partial_{t} \bm{v}^m\big) \notag \\ &\quad +\big((\mu^{m}+\chi \sigma^{m})\nabla \varphi^{m},\partial_{t} \bm{v}^m\big).\notag
\end{align}
Recalling \eqref{varPsi} and applying the estimate \eqref{low-es2}, we now have
\begin{align}
	\| \varPsi'_0(\varphi^m)\|_{L^1}
	&\le C\big(\|  \nabla \mu^m\| + \|\sigma^m-\overline{\sigma^m}\|+ \| \varphi^m - \overline{\varphi^m}\| +\gamma\|\partial_t \varphi^m \|\big)\|\nabla \varphi^m\|+C\notag \\
&\leq C\big(\|  \nabla \mu^m\|+ \gamma\|\partial_t \varphi^m \|+1\big),
	\notag
	\end{align}
which combined with \eqref{average valuea}
further implies
\begin{align}
	\left|\int_\Omega \mu^m\,\mathrm{d}x\right|
	\leq C\big(\|  \nabla \mu^m\|+\gamma\|\partial_t \varphi^m \|+ 1\big).
	\label{muave}
	\end{align}
Then from the Poincar\'{e}--Wirtinger inequality, we get
\begin{align}
	\|\mu^m\|_{H^1}
	\leq C\big(\|  \nabla \mu^m\|+ \gamma\|\partial_t \varphi^m \|+1\big).
	\label{muH1}
	\end{align}
Recalling the definition of the Stokes operator $\boldsymbol{S}$ and following an argument similar to that for \cite[(5.3)]{GMT}, we can deduce from \eqref{low-es2} and \eqref {muH1} that
\begin{align}
\left\|\partial_{t} \bm{v}^m\right\|^{2} & \leq C_3\left[\left\|\boldsymbol{S} \bm{v}^m\right\|^{2}+\left\|\nabla \bm{v}^m\right\|^{6} +\left\|\varphi^m\right\|_{H^{2}}^{4}\left\|\nabla \bm{v}^m\right\|^{2}\right]\notag\\
&\quad   +C_3 \left\|\nabla \varphi^m\right\|_{\bm{L}^{3}}^{2}\big(1+\left\|\nabla \mu^m\right\|^{2}+ \|\nabla(\sigma^m-\chi\varphi^m)\|^2
+\gamma\|\partial_t \varphi^m \|^2\big),
\label{3v}
\end{align}
where $C_3$ depends on the coefficients of the system, $\Omega$, $\overline{\varphi_0}$, $\overline{\sigma_0}$, $\max_{r\in[-1,1]}|\varPsi(r)|$,  $\|\bm{v}_0\|$, $\|\varphi_0\|_{H^1}$ and $\|\sigma_0\|$.

Next, testing \eqref{sapp1}$_1$  with $\bm{S} \bm{v}^m$, we get
\begin{align*}
& \frac12\frac{\mathrm{d}}{\mathrm{d}t}\|\nabla \bm{ v}^{m}\|^2 + 2\int_\Omega  \mathrm{div}\big(\nu(\varphi^{m}) D\bm{v}^{m}\big) \cdot\bm{S} \bm{v}^m \,\mathrm{d}x \\
&\quad =-\int_\Omega (\bm{ v}^{m} \cdot \nabla  \bm {v}^{m})  \cdot\bm{S} \bm{v}^m \,\mathrm{d}x + \int_\Omega (\mu^{m}+\chi \sigma^{m})\nabla \varphi^{m}\cdot\bm{S} \bm{v}^m \,\mathrm{d}x.
\end{align*}
From the decomposition $\bm{S} \bm{v}^m = -\Delta \bm{v}^{m} + \nabla p^m $ with some $p^m\in L^2(0,T_m; H^1(\Omega))$, similar to \cite[(5.4)]{GMT}, we can apply Lemma \ref{stokes}, (H1),  \eqref{low-es2} and \eqref{muH1} to conclude that
\begin{align}
&\frac{1}{2} \frac{\mathrm{d}}{\mathrm{d} t}\left\|\nabla \bm{v}^m\right\|^{2}+\nu_{*}\left\|\boldsymbol{S} \bm{v}^m\right\|^{2}\notag\\
&\quad \leq C_{4}\left[(1+\left\|\varphi^m\right\|^{8}_{H^2})\left\|\nabla \bm{v}^m\right\|^{2}+\left\|\nabla \bm{v}^m\right\|^{6}\right]\notag\\
&\qquad
+C_4 \left\|\nabla \varphi^m\right\|_{\boldsymbol{L}^{3}}^{2}\big(1+\|\nabla \mu^m\|^{2}+ \|\nabla(\sigma^m-\chi\varphi^m)\|^2 +\gamma\|\partial_t \varphi^m \|^2\big),
\label{3vt}
\end{align}
where $C_4$ has a similar dependence like that for $C_3$.

To control the $H^2$-norm of $\varphi^m$, we recall   \eqref{esH2} and apply \eqref{low-es2} to obtain
\begin{align}
& \int_\Omega \varPsi_0''(\varphi^m)|\nabla \varphi^m|^2\, \mathrm{d}x + \|\Delta \varphi^m\|^2
\notag
\\
&\quad =
\int_\Omega \nabla \mu^m\cdot \nabla \varphi^m\, \mathrm{d}x  +\gamma\int_\Omega \partial_t\varphi^m\Delta \varphi^m\,\mathrm{d}x\notag\\
&\qquad  - \int_\Omega \big(\chi \sigma^{m} -\beta\mathcal{N}(\varphi^m-\overline{\varphi^m}) +
\theta_0 \varphi^m\big)\Delta \varphi^m \,\mathrm{d}x \notag\\
&\quad \leq \|\nabla \mu^m\|\| \nabla \varphi^m\|+\gamma\|\partial_t\varphi^m\|\|\Delta \varphi^m\|+
C(\|\sigma^m\| +\|\varphi^m\|)\|\Delta \varphi^m\|\notag \\
&\quad \leq \frac{1}{2} \|\Delta \varphi^m\|^2+
C\big(\|\nabla \mu^m\| +\gamma\|\partial_t\varphi^m\|^2 +1\big).
\notag
\end{align}
As a result, we infer from the elliptic estimate for the Neumann problem that
\begin{align}
\|\varphi^m\|_{H^2}^2 \leq C\big(\|\nabla \mu^m\|+\gamma\|\partial_t\varphi^m\|^2+1\big).
\label{H2es}
\end{align}
Multiplying \eqref{3v} by $\nu_{*}/(2 C_{3})>0$ and summing up with \eqref{3vt}, using the Sobolev embedding $H^2(\Omega)\hookrightarrow W^{1,3}(\Omega)$ as well as the estimates \eqref{low-es2}, \eqref{H2es}, we arrive at
\begin{align}
&\frac{1}{2} \frac{\mathrm{d}}{\mathrm{d} t}\left\|\nabla \bm{v}^m\right\|^{2}+\frac{\nu_{*}}{2}\left\|\boldsymbol{S} \bm{v}^m\right\|^{2}+\frac{\nu_{*}}{2 C_{3}}\left\|\partial_{t} \bm{v}^m\right\|^{2}\notag \\
&\quad \leq C \left[\big(1+\left\|\varphi^m\right\|_{H^{2}}^{8}\big)\left\|\nabla \bm{v}^m\right\|^{2}+\left\|\nabla \bm{v}^m\right\|^{6} \right.\notag\\
&\qquad \left.+\left\| \varphi^m\right\|_{H^2}^{2}\big(1+\left\|\nabla \mu^m\right\|^{2}+ \|\nabla(\sigma^m-\chi\varphi^m)\|^2 +\gamma\|\partial_t\varphi^m\|^2 \big)\right]\notag\\
 &\quad \leq C \left[\big(1+\left\|\nabla \mu^m\right\|^{4}+\gamma^4\|\partial_t\varphi^m\|^8\big)\left\|\nabla \bm{v}^m\right\|^{2}+\left\|\nabla \bm{v}^m\right\|^{6}\right.\notag\\
&\qquad \left.+\big(1+\left\|\nabla \mu^m\right\|+\gamma\|\partial_t\varphi^m\|^2\big)\big(1+\left\|\nabla \mu^m\right\|^{2}+ \|\nabla(\sigma^m-\chi\varphi^m)\|^2 +\gamma\|\partial_t\varphi^m\|^2\big)\right],
\label{vt}
\end{align}
where $C$ depends on the coefficients of the system, $\Omega$, $\overline{\varphi_0}$, $\overline{\sigma_0}$, $\max_{r\in[-1,1]}|\varPsi(r)|$,  $\|\bm{v}_0\|$, $\|\varphi_0\|_{H^1}$ and $\|\sigma_0\|$.

The regularity property obtained in Proposition \ref{exe-app} also allows us to test \eqref{sapp1}$_2$  by $\partial_{t} \mu^m$. Integrating over $\Omega$ and after integration by parts, we obtain
\begin{align*}
\frac{1}{2}\frac{\mathrm{d}}{\mathrm{d}t} \|\nabla\mu^m\|^2
+ \int_\Omega \partial_t\varphi^m\partial_{t} \mu^m\,\mathrm{d}x
+ \int_\Omega (\bm{v}^m\cdot\nabla \varphi^m)\partial_{t} \mu^m\,\mathrm{d}x
+\alpha(\overline{\varphi^m}-c_0) \int_\Omega \partial_{t} \mu^m\,\mathrm{d}x=0.
\end{align*}
 A direct computation yields that
\begin{align*}
\int_\Omega \partial_t\varphi^m\partial_{t} \mu^m\,\mathrm{d}x
& = \gamma\int_\Omega \partial_{t} \varphi^m\partial_{t}^2 \varphi^m  \,\mathrm{d}x+
\|\nabla\partial_t\varphi^m\|^2 +
\int_\Omega \varPsi_0''(\varphi^m)|\partial_t\varphi^m|^2\,\mathrm{d}x \\
&\quad -\theta_0\|\partial_t\varphi^m\|^2
-\chi\int_\Omega \partial_{t} \varphi^m\partial_t\sigma^m  \,\mathrm{d}x
+\beta \int_\Omega \partial_{t} \varphi^m \mathcal{N}(\partial_{t} \varphi^m -\overline{\partial_{t} \varphi^m}) \,\mathrm{d}x,
\end{align*}
\begin{align*}
\int_\Omega (\bm{v}^m\cdot\nabla \varphi^m)\partial_{t} \mu^m\,\mathrm{d}x
&= \frac{\mathrm{d}}{\mathrm{d}t}\int_\Omega (\bm{v}^m\cdot \nabla \varphi^m)\mu^m\, \mathrm{d}x- \int_\Omega (\partial_t\bm{v}^m\cdot \nabla \varphi^m)\mu^m\,\mathrm{d}x\notag\\
&\quad -
\int_\Omega (\bm{v}^m\cdot \nabla \partial_t \varphi^m)\mu^m\,\mathrm{d}x.
\end{align*}
Besides, from \eqref{conver1}, it follows that
\begin{align*}
 \alpha(\overline{\varphi^m}-c_0) \int_\Omega \partial_{t} \mu^m\,\mathrm{d}x
&=   \frac{\mathrm{d}}{\mathrm{d}t} \left(\alpha(\overline{\varphi^m}-c_0)  \int_\Omega \mu^m\,\mathrm{d}x \right)- \alpha \left(\frac{\mathrm{d}}{\mathrm{d}t}\overline{\varphi^m} \right) \int_\Omega \mu^m\,\mathrm{d}x
\notag\\
&=  \frac{\mathrm{d}}{\mathrm{d}t} \left(\alpha(\overline{\varphi^m}-c_0)  \int_\Omega \mu^m\,\mathrm{d}x \right) +
\alpha^2 \big(\overline{\varphi^m}(t)-c_0\big) \int_\Omega \mu^m\,\mathrm{d}x.
\end{align*}
As a consequence, we deduce that
	\begin{align}
&\frac{\mathrm{d}}{\mathrm{d}t} \left[\frac{1}{2}\|\nabla\mu^m\|^2+\frac{\gamma}{2}\|\partial_t \varphi^m\|^2
+\int_\Omega (\bm{v}^m\cdot\nabla\varphi^m)\mu^m\,\mathrm{d}x + \alpha(\overline{\varphi^m}-c_0) \int_\Omega   \mu^m\,\mathrm{d}x\right]
+\|\nabla\partial_t\varphi^m\|^2\notag \\
&\quad = -\int_\Omega \varPsi_0''(\varphi^m)|\partial_t\varphi^m|^2\,\mathrm{d}x +\theta_0\|\partial_t\varphi^m\|^2 +\chi\int_\Omega \partial_{t} \varphi^m\partial_t\sigma^m  \,\mathrm{d}x
\notag\\
&\qquad
-\beta \int_\Omega \partial_{t} \varphi^m\mathcal{N}(\partial_{t} \varphi^m -\overline{\partial_{t} \varphi^m}) \,\mathrm{d}x + \int_\Omega (\partial_t\bm{v}^m\cdot \nabla \varphi^m)\mu^m\,\mathrm{d}x \notag\\
&\qquad +\int_\Omega (\bm{v}^m\cdot \nabla \partial_t \varphi^m)\mu^m\,\mathrm{d}x -\alpha^2 \big(\overline{\varphi^m}(t)-c_0\big) \int_\Omega \mu^m\,\mathrm{d}x\notag\\
&\quad =: \sum_{j=1}^7 I_j. \notag
\end{align}
From (H2), \eqref{sapp1}$_2$, \eqref{conver1} and \eqref{low-es2}, we find $I_1\leq 0$ and
\begin{align*}
I_2 &\leq 2\theta_0\big(\|\partial_t\varphi^m- \overline{\partial_t\varphi^m}\|^2 + |\Omega| |\overline{\partial_t\varphi^m}|^2\big)\notag\\
&\leq C\|\nabla \partial_t\varphi^m\|\|\partial_t\varphi^m- \overline{\partial_t\varphi^m}\|_{V_0'} +C|\overline{\partial_t\varphi^m}|^2\notag\\
&\leq \|\nabla \partial_t\varphi^m\|\big(
\|\nabla \mu^m\|+ \|\bm{v}^m\cdot \nabla \varphi^m\|_{L^\frac65}\big) +C\\
&\leq  \frac{1}{8} \|\nabla \partial_t\varphi^m\|^2
+C\|\nabla \mu^m\|^2+C\|\bm{v}^m\|_{\bm{L}^3}^2\|\nabla \varphi^m\|^2 +C\\
&\leq  \frac{1}{8} \|\nabla \partial_t\varphi^m\|^2
+C\|\nabla \mu^m\|^2+C\|\nabla \bm{v}^m\|^2  +C.
\end{align*}
Using \eqref{aver-sig}, \eqref{low-es2}, H\"{o}lder's inequality and the Poincar\'{e}--Wirtinger inequality, we estimate $I_3$ as follows
\begin{align*}
I_3&= \chi\int_\Omega (\partial_{t} \varphi^m -\overline{\partial_{t} \varphi^m})\partial_t\sigma^m  \,\mathrm{d}x\notag\\
&\leq C\|\partial_t\varphi^m-\overline{\partial_{t} \varphi^m}\|_{V_0}\|\partial_t\sigma^m\|_{V_0'}\notag\\
&\leq C\|\nabla \partial_t\varphi^m\|\big(\|\nabla(\sigma^m-\chi\varphi^m)\|+\|\bm{v}^m \cdot \nabla \sigma^m\|_{L^\frac65}\big)\\
&\leq \frac{1}{8} \|\nabla \partial_t\varphi^m\|^2 +  C\|\nabla (\sigma^m-\chi\varphi^m)\|^2+ C\|\bm{v}^m\|_{\bm{L}^3}^2\|\nabla \sigma^m\|^2
\\
&\leq \frac{1}{8} \|\nabla \partial_t\varphi^m\|^2 +  C\|\nabla (\sigma^m-\chi\varphi^m)\|^2+ C\|\nabla\bm{v}^m\|^2\big(\|\nabla (\sigma^m-\chi\varphi^m)\|^2+1\big).
\end{align*}
Concerning $I_4$, we infer from \eqref{sapp1}$_2$ and the Poincar\'{e}--Wirtinger inequality that
\begin{align*}
I_4& =-\beta \int_\Omega (\partial_{t} \varphi^m - \overline{\partial_{t} \varphi^m})\mathcal{N}(\partial_{t} \varphi^m -\overline{\partial_{t} \varphi^m})\,\mathrm{d}x\\
&\leq C \|\nabla \partial_t \varphi^m\| \|\mathcal{N}(\partial_{t} \varphi^m -\overline{\partial_{t}\varphi^m}) \|\\
&\leq C\|\nabla \partial_t \varphi^m\|\big(\|\bm{v}^m \cdot \nabla \varphi^m\|_{L^\frac65}+\|\mu^m-\overline{\mu^m}\|\big)\\
&\leq \frac{1}{8} \|\nabla \partial_t\varphi^m\|^2
+ C\|\nabla \bm{v}^m\|^2 + C\|\nabla \mu^m\|^2.
\end{align*}
By \eqref{low-es2}, \eqref{muH1}, \eqref{H2es},  H\"{o}lder's inequality and Young's inequality, we find (cf. \cite[Section 4]{GMT})
\begin{align*}
I_5 &\leq \|\partial_t\bm{v}^m\|\| \nabla \varphi^m\|_{\bm{L}^3}\|\mu^m\|_{L^6}\\
&\leq \frac{\nu_{*}}{8 C_{3}}\left\|\partial_{t} \bm{v}^m\right\|^{2}+ C \|\varphi^m\|_{H^2}^2\|\mu^m\|_{H^1}^2\\
&\leq \frac{\nu_{*}}{8 C_{3}}\left\|\partial_{t} \bm{v}^m\right\|^{2}+ C \big(1+\|\nabla \mu^m\|^2+ \gamma\|\partial_t \varphi^m \|^2\big)^2,
\end{align*}
and
\begin{align*}
I_6
&\leq \|\bm{v}^m\|_{\bm{L}^3}\|\nabla \partial_t \varphi^m\|\|\mu^m\|_{L^6}\\
&\leq \frac{1}{8} \|\nabla \partial_t\varphi^m\|^2
+ C \|\bm{v}^m\|\|\nabla \bm{v}^m\|\|\mu^m\|_{H^1}^2\\
&\leq \frac{1}{8} \|\nabla \partial_t\varphi^m\|^2
+ C  \|\nabla \bm{v}^m\|\big(1+\|\nabla \mu^m\|^2+\gamma\|\partial_t \varphi^m \|^2\big)
\end{align*}
The last term $I_7$ can be estimated by simply using \eqref{q3} such that
$$
I_7\leq C\left|\int_\Omega \mu^m\,\mathrm{d}x\right|\leq   C \big(1+\|\nabla \mu^m\|^2+\|\nabla(\sigma^m-\chi\varphi^m)\|^2+\gamma\|\partial_t \varphi^m \|^2\big).
$$
Combining the above estimates, we  obtain
	\begin{align}
&\frac{\mathrm{d}}{\mathrm{d}t} \left[\frac{1}{2}\|\nabla\mu^m\|^2
+ \frac{\gamma}{2}\|\partial_t \varphi^m\|^2+ \int_\Omega (\bm{v}^m\cdot\nabla\varphi^m)\mu^m\,\mathrm{d}x + \alpha(\overline{\varphi^m}-c_0) \int_\Omega   \mu^m\,\mathrm{d}x\right] \notag\\
&\qquad
+\frac{1}{2}\|\nabla\partial_t\varphi^m\|^2
-\frac{\nu_{*}}{8 C_{3}}\left\|\partial_{t} \bm{v}^m\right\|^{2}\notag \\
&\quad\le C \big(1+\|\nabla \bm{v}^m\|^2+\|\nabla \mu^m\|^2+ \|\nabla(\sigma^m-\chi\varphi^m)\|^2+\gamma\|\partial_t \varphi^m \|^2\big)^2,
\label{i}
\end{align}
where $C$ depends on the coefficients of the system, $\Omega$, $\overline{\varphi_0}$, $\overline{\sigma_0}$, $\max_{r\in[-1,1]}|\varPsi(r)|$,  $\|\bm{v}_0\|$, $\|\varphi_0\|_{H^1}$ and $\|\sigma_0\|$.
%%%%%%%%%%%%%%%%%%%%%%%%%%%%%%%%%

Multiplying \eqref{sapp1}$_4$ with $-\Delta(\sigma^m-\chi\varphi^m)$ and integrating over $\Omega$, we find
\begin{align}
&\frac{1}{2} \frac{\mathrm{d}}{\mathrm{d} t}\left\|\nabla(\sigma^m-\chi\varphi^m)\right\|^{2} +\|\Delta(\sigma^m-\chi\varphi^m)\|^2\notag\\
&\quad=\int_\Omega (\bm{v}^m \cdot \nabla \sigma^m)\Delta(\sigma^m-\chi\varphi^m)\,\mathrm{d}x -\chi\int_\Omega \nabla\partial_t\varphi^m\cdot \nabla(\sigma^m-\chi\varphi^m)\,\mathrm{d}x\notag\\
&\quad =: I_8+I_9.\notag
\end{align}
It follows from H\"{o}lder's inequality, the Gagliardo--Nirenberg inequality, Young's inequality, the elliptic estimate for the Neumann problem and \eqref{H2es} that
\begin{align}
I_8&\le \|\bm{v}^m\|_{\bm{L}^{6}}\| \nabla\sigma^m\|_{\bm{L}^{3}} \|\Delta(\sigma^m-\chi\varphi^m)\|\notag\\
&\le\frac{1}{4}\|\Delta(\sigma^m-\chi\varphi^m)\|^2+ C\|\nabla\bm{v}^m\|^2\|\nabla \sigma^m\|\big(\|\Delta\sigma^m\|+ \|\sigma^m\|\big)\notag\\
&\leq \frac{1}{4}\|\Delta(\sigma^m-\chi\varphi^m)\|^2+ C\|\nabla\bm{v}^m\|^2\big(\|\Delta \sigma^m\|^\frac32\|\sigma^m\|^\frac12+ \|\nabla \sigma^m\|\|\sigma^m\|\big)\notag\\
&\leq \frac{1}{4}\|\Delta(\sigma^m-\chi\varphi^m)\|^2\notag\\
&\quad + C\|\nabla\bm{v}^m\|^2 \big(\|\Delta(\sigma^m-\chi\varphi^m)\|^\frac32+
\|\Delta \varphi^m\|^\frac32+\left\|\nabla(\sigma^m-\chi\varphi^m)\right\| + 1\big)\notag\\
&\le\frac{1}{2}\|\Delta(\sigma^m-\chi\varphi^m)\|^2 +C\|\Delta\varphi^m\|^2 +C\|\nabla\bm{v}^m\|^8+C\|\nabla\bm{v}^m\|^2 \big(\|\nabla(\sigma^m-\chi\varphi^m)\| + 1\big)\notag \\
&\le\frac{1}{2}\|\Delta(\sigma^m-\chi\varphi^m)\|^2  + C\big(1+\|\nabla \mu^m\|^2 +\left\|\nabla(\sigma^m-\chi\varphi^m)\right\|^2
+\gamma\|\partial_t\varphi^m\|^2\big)\notag\\
&\quad +C\|\nabla\bm{v}^m\|^8.
\end{align}
Besides, applying H\"{o}lder's and Young's inequalities we get
\be
I_9\le \frac{a_1}{4}\|\nabla\partial_t\varphi^m\|^2 +\frac{\chi^2}{a_1}\|\nabla(\sigma^m-\chi\varphi^m)\|^2,
\ee
where $a_1\in(0,1)$ will be determined later. As a consequence, we obtain
\begin{align}
&\frac{1}{2} \frac{\mathrm{d}}{\mathrm{d} t}\left\|\nabla(\sigma^m-\chi\varphi^m)\right\|^{2} +\frac12\|\Delta(\sigma^m-\chi\varphi^m)\|^2\notag\\
&\quad\le C\|\nabla\bm{v}^m\|^8 +\frac{a_1}{4}\|\nabla\partial_t\varphi^m\|^2 +\frac{\chi^2}{a_1}\|\nabla(\sigma^m-\chi\varphi^m)\|^2 \notag\\
&\qquad
+C\big(1+\|\nabla \mu^m\|^2 +\left\|\nabla(\sigma^m-\chi\varphi^m)\right\|^2
+\gamma\|\partial_t\varphi^m\|^2\big).
 \label{2sig}
\end{align}

Multiplying \eqref{i} with $a_1\in (0,1)$
and adding the resultant with \eqref{vt}, \eqref{2sig}, then applying Young's inequality, we arrive at
\begin{align}
&\frac{\mathrm{d}}{\mathrm{d} t} \Lambda_m(t)+\frac{\nu_{*}}{2}\left\|\boldsymbol{S} \bm{v}^m\right\|^{2} +\frac{\nu_{*}}{8 C_{3}}(4-a_1)\left\|\partial_{t} \bm{v}^m\right\|^{2} +\frac{a_1}{4}\left\|\nabla \partial_{t} \varphi^m\right\|^{2} +\frac12\|\Delta(\sigma^m-\chi\varphi^m)\|^2 \notag\\
&\quad
\leq C \big(1+\|\nabla \bm{v}^m\|^2+\|\nabla \mu^m\|^2+ \|\nabla(\sigma^m-\chi\varphi^m)\|^2+\gamma\|\partial_t \varphi^m \|^2\big)^5\notag\\
&\qquad + \frac{\chi^2}{a_1}\|\nabla(\sigma^m-\chi\varphi^m)\|^2, \label{lm}
\end{align}
where
\begin{align}
\Lambda_m(t)& =\frac{1}{2}\left\|\nabla \bm{v}^m(t)\right\|^{2}
+\frac{a_1}{2}\left\|\nabla \mu^m(t)\right\|^{2}
+ \frac{a_1\gamma}{2}\|\partial_t \varphi^m(t)\|^2
\notag\\
&\quad
+\frac{1}{2} \|\nabla(\sigma^m(t)-\chi\varphi^m(t))\|^{2}
 +a_1\int_\Omega \big(\bm{v}^m(t)\cdot\nabla\varphi^m(t)\big)\mu^m(t)\,\mathrm{d}x \notag \\
&\quad
+ a_1 \alpha\big(\overline{\varphi^m}(t)-c_0\big)\int_\Omega \mu^m(t)\,\mathrm{d}x,
\label{h1}
\end{align}
and the positive constant  $C$ in \eqref{lm} depends on  $\Omega$, $\overline{\varphi_0}$, $\overline{\sigma_0}$, $\max_{r\in[-1,1]}|\varPsi(r)|$,  $\|\bm{v}_0\|$, $\|\varphi_0\|_{H^1}$, $\|\sigma_0\|$ and the coefficients of the system.

Let us choose the constant $a_1$ in the definition of $\Lambda_m$.
From \eqref{inif} and Poincar\'{e}'s inequality, we can see that
\begin{align}
\left|a_1\int_\Omega \big(\bm{v}^m\cdot\nabla\varphi^m\big)\mu^m\,\mathrm{d}x\right|
&=\left|a_1\int_\Omega \big(\bm{v}^m\cdot\nabla\mu^m\big)\varphi^m\,\mathrm{d}x\right|
  \notag\\
&\le a_1 C_\mathrm{P}\|\nabla\bm{v}^m\|\|\nabla \mu^m\|\|\varphi^m\|_{L^\infty}\notag\\
&\le \frac{a_1}{8}\|\nabla\mu^m\|^2+ 2a_1 C^2_\mathrm{P}\|\nabla\bm{v}^m\|^2,
\label{qqqb}
\end{align}
where $C_\mathrm{P}>0$ only depends on $\Omega$. Besides, by \eqref{q2} (with $a=1/16\alpha$ therein), we get
\begin{align}
& \left|a_1\alpha\big(\overline{\varphi^m}-c_0\big) \int_\Omega \mu^m\,\mathrm{d}x\right|\notag\\
& \quad \leq a_1 \alpha  e^{-\alpha t}|\overline{\varphi_{0,k}}-c_0| \left|\int_\Omega
\mu^m\,\mathrm{d}x\right| \notag\\
&\quad \leq \frac{a_1}{8}\left(\|\nabla\mu^m\|^2+\|\nabla(\sigma^m-\chi\varphi^m)\|^2+\gamma\|\partial_t \varphi^m \|^2\right) +C_5a_1\alpha  e^{-\alpha t}|\overline{\varphi_{0,k}}-c_0| ,\label{qqqa}
\end{align}
where $C_5$ depends on the coefficients of the system, $\Omega$, $\overline{\varphi_0}$, $\overline{\sigma_0}$, $\max_{r\in[-1,1]}|\varPsi(r)|$,  $\|\bm{v}_0\|$, $\|\varphi_0\|_{H^1}$ and $\|\sigma_0\|$. Set
\begin{align}
a_1 =\min\left\{\frac{1}{2},\ \frac{1}{8C_\mathrm{P}^2}\right\}>0.
\label{a1}
\end{align}
Then from the above estimates we see that
\begin{align}
\Lambda_m(t) &\geq \frac{1}{4}\left\|\nabla \bm{v}^m(t)\right\|^{2}
+\frac{a_1}{4}\left\|\nabla \mu^m(t)\right\|^{2}
+ \frac{a_1\gamma}{4}\|\partial_t \varphi^m\|^2 \notag \\
&\quad +\frac{1}{4} \|\nabla(\sigma^m(t)-\chi\varphi^m(t))\|^{2}-\widetilde{C}_5,\label{m3b}\\
\Lambda_m(t) & \leq \frac34\left\|\nabla \bm{v}^m(t)\right\|^{2} +\frac{3a_1}{4}\left\|\nabla \mu^m(t)\right\|^{2}
+ \frac{3a_1\gamma}{4}\|\partial_t \varphi^m\|^2 \notag  \\
&\quad +\frac{3}{4} \|\nabla(\sigma^m(t)-\chi\varphi^m(t))\|^{2}+\widetilde{C}_5,\label{m3}
\end{align}
with the constant $\widetilde{C}_5=C_5\alpha$. Therefore, exploiting \eqref{lm}--\eqref{m3}, we are led to
\begin{align}
&\frac{\mathrm{d}}{\mathrm{d} t} \widetilde{\Lambda}_m(t) +\frac{\nu_{*}}{2}\left\|\boldsymbol{S} \bm{v}^m\right\|^{2} +\frac{\nu_{*}}{4 C_{3}} \left\|\partial_{t} \bm{v}^m\right\|^{2} +\frac{ a_1}{4}\left\|\nabla \partial_{t} \varphi^m\right\|^{2} +\frac12\|\Delta(\sigma^m-\chi\varphi^m)\|^2\notag\\
&\quad \leq C_6\big[\widetilde{\Lambda}_m(t)\big]^5,
 \label{m5}
\end{align}
where $\widetilde{\Lambda}_m(t)= \Lambda_m(t)+2\widetilde{C}_5$ and the constant $C_6>0$ depends on the coefficients of the system, $\Omega$, $\overline{\varphi_0}$, $\overline{\sigma_0}$, $\max_{r\in[-1,1]}|\varPsi(r)|$,  $\|\bm{v}_0\|$, $\|\varphi_0\|_{H^1}$ and $\|\sigma_0\|$.

The differential inequality  \eqref{m5} yields that for any $\widetilde{T}>0$ satisfing
\begin{equation}
1-4C_6\widetilde{T}\big[\widetilde{\Lambda}_m (0)\big]^4>0,\label{m5bb}
\end{equation}
it holds
\be
0<\widetilde{\Lambda}_m(t)\leq \frac{\widetilde{\Lambda}_m (0)}{\big(1-4C_6t\big[\widetilde{\Lambda}_m (0)\big]^4\big)^\frac14},\qquad \forall\, t\in [0,\widetilde{T}].\label{m5b}
\ee
Using  \eqref{sapp1}$_2$, \eqref{sapp1}$_3$  and integration by parts, we find that
\begin{align*}
\gamma\|\partial_t \varphi^m\|^2
& = \int_\Omega \mu^m\partial_t \varphi^m\,\mathrm{d}x- \int_\Omega \big(-\Delta \varphi^m+\varPsi'(\varphi^m)-\chi \sigma^m+\beta\mathcal{N}(\varphi^m-\varphi^m)\big)\partial_t \varphi^m\,\mathrm{d}x
\\
& = -\|\nabla \mu^m\|^2+ \int_\Omega \nabla \mu^m \cdot (\varphi^m\bm{v}^m) \,\mathrm{d}x   -  \alpha (\overline{\varphi^m}-c_0)\int_\Omega \gamma\partial_t \varphi^m\,\mathrm{d}x\\
&\quad +  \int_\Omega \nabla \big(-\Delta \varphi^m+\varPsi'(\varphi^m)\big)\cdot(\nabla \mu^m-\varphi^m\bm{v}^m)\,\mathrm{d}x\\
&\quad + \int_\Omega \nabla \big(-\chi \sigma^m+\beta\mathcal{N}(\varphi^m-\varphi^m)\big)\cdot(\nabla \mu^m-\varphi^m\bm{v}^m)\,\mathrm{d}x.
\end{align*}
Thus, by continuity, \eqref{conver1}, H\"{o}lder's inequality and Poincar\'{e}'s inequality,  we have
\begin{align*}
& \gamma\|\partial_t \varphi^m(0)\|^2 + \|\nabla \mu^m(0) \|^2\\
&\quad = \int_\Omega \nabla \mu^m(0) \cdot (\varphi_{0,k}\bm{P}_{\bm{H}_{m}} \bm{v}_0) \,\mathrm{d}x
+\gamma \alpha^2 |\Omega| (\overline{\varphi_{0,k}}-c_0)^2\\
&\qquad  +  \int_\Omega \nabla \big(-\Delta \varphi_{0,k}+ \varPsi'(\varphi_{0,k})\big)\cdot\big(\nabla \mu^m(0)-\varphi_{0,k}\bm{P}_{\bm{H}_{m}} \bm{v}_0\big)\,\mathrm{d}x\\
&\qquad + \int_\Omega \nabla \big(-\chi \sigma_{0,k}+\beta\mathcal{N}(\varphi_{0,k} -\overline{\varphi_{0,k}})\big)\cdot\big(\nabla \mu^m(0)-\varphi_{0,k}\bm{P}_{\bm{H}_{m}} \bm{v}_0\big)\,\mathrm{d}x\\
&\quad \leq \frac12\|\nabla \mu^m(0) \|^2+  C\big(1+\|\bm{v}_0\|_{\bm{H}^1}^2+\|\widetilde{\mu}_0\|_{H^1}^2 +\|\varphi_0\|_{H^1}^2 +\|\sigma_0\|_{H^1}^2\big),
\end{align*}
which implies
\begin{align}
\widetilde{\Lambda}_m(0)
& =\frac{1}{2}\left\|\nabla \bm{P}_{\bm{H}_{m}} \bm{v}_{0}\right\|^{2}
+\frac{a_1}{2}\left\|\nabla \mu^m(0)\right\|^{2}
+ \frac{a_1\gamma}{2}\|\partial_t \varphi^m(0)\|^2 \notag\\
&\quad
+\frac{1}{2} \|\nabla(\sigma_{0,k}-\chi\varphi_{0,k})\|^{2}
 -a_1\int_\Omega \big( \varphi_{0,k}\bm{P}_{\bm{H}_{m}} \bm{v}_{0}\big)\cdot\nabla \mu^m(0)\,\mathrm{d}x \notag \\
&\quad
+ a_1 \alpha\big(\overline{\varphi_{0,k}}-c_0\big)\int_\Omega \mu^m(0)\,\mathrm{d}x+2\widetilde{C}_5\notag\\
&\leq C\big(1+ \|\bm{v}_{0}\|_{\bm{H}^1}^2
+ \|\widetilde{\mu}_{0}\|_{H^1}^2 + \|\varphi_{0}\|_{H^1}^2 +\|\sigma_0\|_{H^1}^2\big)+3\widetilde{C}_5\notag\\
&=:\widetilde{\Lambda}_0,
\label{h1-0}
\end{align}
such that $\widetilde{\Lambda}_0>0$ is independent of $\gamma$, $m$ and $k$ (for $k\geq \widehat{k}$).

Therefore, setting
$$
T_0=\frac{1}{5C_6(\widetilde{\Lambda}_0)^4}>0,
$$
which satisfies \eqref{m5bb} and is independent of $\gamma$, $m$ as well as $k$ (for $k\geq \widehat{k}$), we can deduce from \eqref{m5b} and \eqref{h1-0} that
\be
\widetilde{\Lambda}_m(t)\leq \frac{\widetilde{\Lambda}_0}{\big[1-4C_6t(\widetilde{\Lambda}_0)^4\big]^\frac14}, \qquad \forall\, t\in [0,T_0].
\label{m5c}
\ee
Indeed, the uniform estimates \eqref{low-es2} and \eqref{m5c} enable us to prove the existence of a unique local strong solution $(\bm{v}^{m},\varphi^{m},\mu^{m},\sigma^{m})$ on the uniform interval $[0,T_0]$ instead of $[0,T_m]$ (cf. Proposition \ref{exe-app}) for any given parameters $\gamma>0$, $m\geq 1$ and $k\geq \widehat{k}$. Moreover, thanks to \eqref{low-es2} and \eqref{m3b}, it holds
\begin{align}
& \sup_{t \in [0,T_0]} \big(\left\|\nabla \bm{v}^m(t)\right\|+\left\|\nabla \mu^m(t)\right\|+
\sqrt{\gamma}  \|\partial_t\varphi^m(t)\| + \|\nabla\sigma^m(t)\| \big)\leq C,
\label{sv}
\end{align}
which together with \eqref{sapp1}$_3$,
\eqref{low-es2}, \eqref{muH1} and \eqref{H2es} yields
\begin{align}
\sup_{t \in [0,T_0]} \left(\|\varphi^m(t)\|_{H^2}+ \| \mu^m(t) \|_{H^1}+ \|\varPsi_0'(\varphi^m(t))\|\right)\leq C.\label{sv2}
\end{align}
Next, integrating \eqref{m5} on $[0,T_0]$ and using \eqref{sv}, \eqref{sv2}, we can obtain
\begin{align}
\int_{0}^{T_0}\left(\left\|\boldsymbol{S} \bm{v}^m(t)\right\|^{2}+\left\|\partial_{t} \bm{v}^m(t)\right\|^{2}+\left\|\nabla \partial_{t} \varphi^m(t)\right\|^{2}+\|\sigma^m (t)\|_{H^2}^2\right) \, \mathrm{d} t \leq C.\label{sv1}
\end{align}
The positive constant $C$ in \eqref{sv}--\eqref{sv1} depends on the coefficients of the system, $\Omega$, $\overline{\varphi_0}$, $\overline{\sigma_0}$, $\|\bm{v}_0\|_{\bm{H}^1}$, $\|\varphi_0\|_{H^1}$, $\|\widetilde{\mu}_0\|_{H^1}$, $\|\sigma_0\|_{H^1}$ and $\max_{r\in[-1,1]}|\varPsi(r)|$. From \eqref{sapp1}$_2$, we observe that
$$
\|\Delta \mu^m\|\leq \|\partial_t\varphi^m\|+\|\bm{v}^m\|_{\bm{L}^3}\|\nabla\varphi^m\|_{\bm{L}^6} +\alpha\|\overline{\varphi^m}-c_0\|,
$$
and
\begin{align*}
\|\nabla \Delta \mu^m\|^2& =\int_\Omega \nabla \partial_t\varphi^m\cdot\nabla \Delta \mu^m\,\mathrm{d}x +\int_\Omega \nabla (\bm{v}^m\cdot \nabla \varphi^m) \cdot\nabla \Delta \mu^m\,\mathrm{d}x\\
&\leq \frac12 \|\nabla \Delta \mu^m\|^2 +C\|\nabla \partial_t\varphi^m\|^2+ C\|\nabla \bm{v}^m\|_{\bm{L}^3}^2\|\nabla \varphi^m\|_{\bm{L}^6}^2+ C\|\bm{v}^m\|_{\bm{L}^\infty}^2\|\varphi^m\|_{H^2}^2.
\end{align*}
Hence, it follows from \eqref{low-es3}, \eqref{sv}--\eqref{sv1} and the Sobolev embedding theorem that
\begin{align}
&\sup_{t\in[0,T_0]}\|\mu^m(t)\|_{H^2}\leq C \left(1+\frac{1}{\sqrt{\gamma}}\right),\qquad \int_{0}^{T_{0}}\left\|\mu^m(t)\right\|_{H^3}^{2}  \, \mathrm{d}t \leq C,\label{v}
\end{align}
where $C>0$ is independent of $\gamma$, $m$ and $k$ (for $k\geq \widehat{k}$).\medskip

\textbf{Fourth estimate: control of time derivatives}.
We have obtained the estimate for $\|\partial_t\bm{v}^m\|$ in \eqref{sv1}. Next, from \eqref{conver1}, \eqref{sv1} and the Poincar\'{e}--Wirtinger inequality, we see that
\begin{equation}
\left \|  \partial_{t}\varphi^m\right \|_{L^{2}(0,T_0;H^1(\Omega))} \le C.
\label{phimt1}
\end{equation}
Finally, a comparison in the equation for $\sigma^m$ yields
\begin{align}
\int_0^{T_0} \|\partial_t\sigma^m(t)\|^2\,\mathrm{d}t
&\leq C\int_0^{T_0} \big(\|\Delta\sigma^m(t)\|^2
+\|\Delta\varphi^m(t)\|^2 + \|\bm{v}^m(t)\|_{\bm{L}^6}^2 \|\nabla\sigma^m(t)\|_{\bm{L}^3}^2\big)\,\mathrm{d}t\notag \\
&\leq C\int_0^{T_0}\big(\|\sigma^m(t)\|_{H^2}
+\|\varphi^m(t)\|_{H^2}^2+ \|\bm{v}^m(t)\|_{\bm{H}^1}^2 \|\sigma^m(t)\|_{H^2}^2\big)\,\mathrm{d}t\notag\\
&\leq C,\label{vsigmt2d}
\end{align}
where the positive constant $C$ in \eqref{phimt1} and \eqref{vsigmt2d} depends on the coefficients of the system, $\Omega$, $\overline{\varphi_0}$, $\overline{\sigma_0}$, $\|\bm{v}_0\|_{\bm{H}^1}$, $\|\varphi_0\|_{H^1}$, $\|\widetilde{\mu}_0\|_{H^1}$, $\|\sigma_0\|_{H^1}$, $\max_{r\in[-1,1]}|\varPsi(r)|$,
but is independent of $\gamma$, $m$ and $k$ (for $k\geq \widehat{k}$). \medskip

\textbf{Fifth estimate: separation property for $\varphi^m$}.
It follows from   \eqref{sv}, \eqref{sv2}  that
\begin{align}
\|\Delta \sigma^m(t)\|
&\leq
\|\partial_t\sigma^m(t)\|
+ |\chi|\|\Delta\varphi^m(t)\| +  \|\bm{v}^m(t)\|_{\bm{L}^6} \|\nabla\sigma^m(t)\|_{\bm{L}^3}\notag \\
&\leq \|\partial_t\sigma^m(t)\|
+ |\chi|\|\Delta\varphi^m(t)\| +  C \|\Delta \sigma^m(t)\|^\frac12\|\nabla \sigma^m(t)\|^\frac12 \notag \\
&\leq \frac12\|\Delta \sigma^m(t)\|+ \|\partial_t\sigma^m(t)\|
+C.
\label{sigmH2a}
\end{align}
Next, differentiating \eqref{sapp1}$_4$ with respect to time, multiplying the resultant by
$\partial_t \sigma^m$ and integrating over $\Omega$, we infer from \eqref{low-es2}, \eqref{sv} and \eqref{sigmH2a} that
\begin{align}
&\frac12 \frac{\mathrm{d}}{\mathrm{d}t} \|\partial_t \sigma^m\|^2 + \|\nabla \partial_t \sigma^m\|^2\notag\\
&\quad = - \int_\Omega (\partial_t \bm{v}^m\cdot \nabla \sigma^m) \partial_t \sigma^m\,\mathrm{d}x
- \int_\Omega  (\bm{v}^m \cdot \nabla \partial_t \sigma^m ) \partial_t \sigma^m\,\mathrm{d}x
+ \chi \int_\Omega
\nabla \partial_t  \varphi^m\cdot \nabla \partial_t  \sigma^m\,\mathrm{d}x\notag\\
& \quad \leq \|\partial_t \bm{v}^m\|\| \nabla (\sigma^m- \overline{\sigma^m}) \|_{\bm{L}^3}\|\partial_t (\sigma^m-\overline{\sigma^m})\|_{L^6}
+|\chi|\|\nabla \partial_t  \varphi^m\|\| \nabla \partial_t  \sigma^m\|\notag\\
 &\quad \leq \frac12\|\nabla \partial_t  \sigma^m\|^2
 + C\|\partial_t  \bm{v}^m\|^2 \|\nabla \sigma^m\|\|\Delta \sigma^m\|
  +\chi^2\|\nabla \partial_t\varphi^m\|^2\notag\\
 &\quad \leq  \frac12\|\nabla \partial_t \sigma^m\|^2
 + C\|\partial_t \bm{v}^m\|^2 \|\partial_t\sigma^m\|^2 +C\big( \|\partial_t \bm{v}^m\|^2+ \|\nabla \partial_t\varphi^m\|^2).
\label{sigmH2}
\end{align}
By continuity, we have
\begin{align*}
\|\partial_t\sigma^m(0)\|
&\leq C\big(\|\Delta\sigma_{0,k}\|
+\|\Delta\varphi_{0,k}\|+ \|\bm{P}_{\bm{H}_{m}}\bm{v}_0\|_{\bm{L}^3} \|\nabla\sigma_{0,k}\|_{\bm{L}^6}\big)\\
&\leq C\big(\|\sigma_{0,k}\|_{H^2}
+\|\varphi_{0,k}\|_{H^2}+ \|\bm{v}_0\|_{\bm{H}^1} \|\sigma_{0,k}\|_{H^2}\big),
\end{align*}
which depends on $k$ but is independent of $m$, $\gamma$.
Then it follows from \eqref{sv1}, \eqref{sigmH2} and Gronwall's lemma that
\begin{align}
\|\partial_t\sigma^m(t)\|^2 +\int_0^t\|\nabla \partial_t\sigma^m(\tau)\|^2\mathrm{d}\tau \leq C(k),\quad \forall\, t\in [0,T_0],\label{sigtL2}
\end{align}
where $C(k)>0$ depends on $k$, but is independent of $m$ and $\gamma$. This estimate combined with \eqref{low-es2} and \eqref{sigmH2a} implies that
\begin{align}
\sup_{t\in[0,T_0]} \|\sigma^m(t)\|_{H^2}\leq \widetilde{C}(k).\label{sigmLi}
\end{align}
Thanks to the $H^2$-estimates \eqref{v}, \eqref{sigmLi} for $\mu^m$ and $\sigma^m$, we can check the arguments in \cite{Gio2022,H2,MZ04} and find that
the following strict separation property
\begin{align}
 |\varphi^m(x,t)|<1-\delta(\gamma,k),\quad \forall\,(x,t)\in \overline{\Omega}\times[0,T_0],\label{sep}
\end{align}
holds with some constant $\delta(\gamma, k)\in (0,1)$ depending on $\gamma$ as well as $k$, but it is independent of the parameter $m$ for the semi-Galerkin approximation scheme. \medskip

\subsection{Proof of Theorem \ref{ls}}
\textbf{Existence.}
Thanks to the uniform estimates \eqref{sv}--\eqref{vsigmt2d}, we are able to pass to the limit in the approximate problem \eqref{sapp1}--\eqref{sapp2}, first as $m\to +\infty$, then $\gamma\to 0^+$ and finally $k\to +\infty$. The procedure follows a standard compactness argument (cf. \cite{Gio2022}) and thus its details are omitted here. Up to a subsequence, we can find a set of limit functions   $(\bm{v}, \varphi, \mu, \sigma)$ that
satisfies the equations
 \begin{equation}
  \begin{cases}
  (\partial_t  \bm{ v},\bm{w})+( \bm{ v} \cdot \nabla  \bm {v},\bm{ w})-\big( \mathrm{div}(2\nu(\varphi) D\bm{v}),  \bm{w}\big)  =\big((\mu+\chi \sigma)\nabla \varphi,\bm {w}\big), \\
  \qquad \qquad \qquad \qquad \qquad \ \,\text{for all } \bm{w} \in \bm{L}^2_{0,\mathrm{div}}(\Omega)\ \text{and almost all } t\in (0,T_0),\\
  \partial_t \varphi+ \bm{v} \cdot \nabla \varphi =\Delta  \mu-\alpha(\overline{\varphi}-c_0), \quad  \qquad \qquad \qquad \ \ \text{a.e. in }\ \ \Omega\times(0,T_0),\\
  \mu=-\Delta \varphi +\varPsi'(\varphi) -\chi \sigma+\beta\mathcal{N}(\varphi-\overline{\varphi}),
   \qquad \qquad \ \ \     \text{a.e. in }\ \ \Omega\times(0,T_0),\\
  \partial_t \sigma+\bm{v} \cdot \nabla \sigma =\Delta (\sigma-\chi\varphi),  \qquad\qquad\qquad\quad\quad\quad \ \ \ \text{a.e. in }\ \ \Omega\times(0,T_0),
  \end{cases}
  \notag
  \end{equation}
  as well as the boundary and initial conditions
  \begin{equation}
  \begin{cases}
  \bm{v}=\mathbf{0},\quad {\partial}_{\bm{n}}\varphi ={\partial}_{\bm{n}}\mu={\partial}_{\bm{n}}\sigma=0, \qquad\qquad\quad \,     \textrm{a.e on} \   \partial\Omega\times(0,T_0),
  \\
  \bm{v}|_{t=0}=\bm{v}_{0},\quad \varphi|_{t=0}=\varphi_{0},\quad  \sigma|_{t=0}=\sigma_{0},\qquad\ \text{a.e. in}\ \Omega,
  \end{cases}
  \notag
  \end{equation}
with the following regularity properties
 \begin{align*}
 &\boldsymbol{v} \in L^\infty\big(0,T_0 ;\bm{H}^1_{0,\mathrm{div}}(\Omega) \big) \cap L^{2}\big(0,T_0 ; \bm{H}^2(\Omega)\big) \cap H^{1}\big(0,T_0 ; \bm{L}^2_{0,\mathrm{div}}(\Omega)\big), \\
 &\varphi \in L^\infty\big(0,T_0 ; H^2(\Omega)\big)\cap  H^1\big(0,T_0 ; H^{1}(\Omega)\big), \\
 & \varphi\in L^\infty(\Omega\times(0,T_0)):\ \ |\varphi(x,t)|<1\ \ \text{a.e. in}\ \ \Omega \times(0,T_0),\\
 &\mu \in L^\infty\big(0,T_0; H^{1}(\Omega)\big) \cap L^{2}\big(0,T_0 ; H^{3}(\Omega)\big), \\
 	&\sigma \in L^\infty\big(0,T_0; H^{1}(\Omega)\big) \cap L^{2}\big(0,T_0 ; H^{2}(\Omega)\big) \cap H^{1}\big(0,T_0 ; L^{2}(\Omega)\big).
 \end{align*}

From the regularity theory for the Cahn--Hilliard equation with singular potential (see e.g., \cite{A2009,GGW,GMT}), we can deduce that
$\varphi \in L^\infty\big(0,T_0;W^{2,6}(\Omega)\big)$ and  $\varPsi'(\varphi)\in L^\infty\big(0,T_0;L^6(\Omega)\big)$.
By comparison in the $\varphi$-equation, we also get $\partial_t \varphi \in L^{\infty}\big(0,T_0 ; (H^{1}(\Omega))'\big)$. From the regularity of time derivatives of $\bm{v}$, $\sigma$ and the interpolation theory \cite{simon}, it follows that $\boldsymbol{v} \in C\big([0,T_0];\bm{H}^1_{0,\mathrm{div}}(\Omega) \big)$ and $\sigma\in C\big([0,T_0]; H^{1}(\Omega)\big)$. Moreover, by H\"{o}lder's inequality, we have
$(\mu+\chi \sigma)\nabla \varphi \in L^2(0,T_0;\bm{L}^2(\Omega))$ and $\bm{ v} \cdot \nabla  \bm {v}\in  L^2(0,T_0;\bm{L}^2(\Omega))$. Then applying the classical theory for the Navier--Stokes equations (see e.g., \cite{S}) and the regularity theory for the Stokes operator with variable viscosity (see e.g., \cite[Lemma 4]{A2009}), we can find a unique function $p\in L^2\big(0,T_0;V_0\big)$ such that
$$
\nabla p=-\partial_t  \bm{ v}-\bm{ v} \cdot \nabla  \bm {v}+\mathrm{div} \big(2  \nu(\varphi) D\bm{v} \big) +(\mu+\chi \sigma)\nabla \varphi,\quad \text{a.e. in}\ \ \Omega\times(0,T_0).
$$

\textbf{Uniqueness}. Let $(\bm{v}_{1},\varphi_{1},\mu_1,\sigma_{1})$ and $(\bm{v}_{2},\varphi_{2},\mu_2,\sigma_{2})$ be two strong solutions subject to the initial data $(\bm{v}_{01},\sigma_{01},\varphi_{01})$ and $(\bm{v}_{02},\sigma_{02},\varphi_{02})$, respectively.   We denote their differences by
$$
(\bm{v},\varphi,\mu,\sigma)=(\bm{v}_{1}-\bm{v}_{2}, \varphi_{1}-\varphi_{2},\mu_1-\mu_2,\sigma_{1}-\sigma_{2}).
$$
The regularity properties of strong solutions obtained above allow us to apply the argument in \cite[Section 4]{H} (cf. also \cite[Section 5]{GMT} for the Navier--Stokes--Cahn--Hilliard system). Using the corresponding Sobolev embedding theorems in three dimensions, we can check the procedure in \cite[Section 4]{H} step by step and deduce that for the quantity
 \begin{align}
 W(t):=\|\nabla\bm{S}^{-1}\bm{v}(t)\|^2 +\|\varphi(t)\|_{(H^1)'}^2+\|\sigma(t)\|_{(H^1)'}^2 +|\overline{\varphi}(t)|,
 \label{W}
 \end{align}
the following differential inequality holds
\begin{align}
\frac{\mathrm{d}}{\mathrm{d}t}W(t)+ \frac12\big( \nu_* \|\bm{v}\|^2 + \|\nabla\varphi\|^2 + \|\sigma\|^2\big)
\leq C W(t),\quad \forall\, t\in (0,T_0).
\label{uniA}
\end{align}
Applying Gronwall's lemma, we thus obtain the continuous dependence on initial data for local strong solutions
\begin{align}
W(t)\leq e^{Ct}W(0),\quad \forall\, t\in[0,T_0],\label{uniA1}
\end{align}
which easily yields the uniqueness result. We remark that higher order estimates for both two solutions are essentially used in the derivation of \eqref{uniA}. Theorefore, this argument cannot be used to prove a weak-strong uniqueness result.

We have thus proved the first part of Theorem \ref{ls}.
\medskip

\textbf{Separation from pure states and further regularity properties.}
Let us assume in addition that $\|\varphi_0\|_{L^\infty}=1-\delta_0$ for some $\delta_0\in (0,1)$. Since $\varphi \in L^{\infty}\big(0, T_0 ; H^{2}(\Omega)\big)$ and $\partial_{t} \varphi \in L^{2}\big(0, T_0 ; H^{1}(\Omega)\big)$, then it follows from the interpolation theory \cite{simon} that  $\varphi\in C\big([0,T_0];C(\overline{\Omega})\big)$ and thus $\|\varphi(t)\|_{C(\overline{\Omega})} = \|\varphi(t)\|_{L^\infty}$. On the other hand, for any $t_{1}, t_{2} \in\left[0, T_0\right]$, $t_1>t_2$, by Agmon's and H\"{o}lder's inequalities, we have
\begin{align}
\left\|\varphi\left(t_{1}\right)-\varphi\left(t_{2}\right)\right\|_{L^{\infty}}  \notag
&  \leq C\left\|\varphi\left(t_{1}\right)-\varphi\left(t_{2}\right)\right\|_{H^{2}}^{\frac{1}{2}}
\left\|\varphi\left(t_{1}\right)-\varphi\left(t_{2}\right)\right\|_{H^{1}}^{\frac{1}{2}}
\notag\\
&  \leq C\left(\left\|\varphi\left(t_{1}\right)\right\|_{H^{2}}^{\frac{1}{2}} +\left\|\varphi\left(t_{2}\right)\right\|_{H^{2}}^{\frac{1}{2}}\right) \left(\int_{t_{2}}^{t_{1}}\left\|\partial_{t} \varphi(t)\right\|_{H^{1}} \mathrm{d} t\right)^{\frac{1}{2}} \notag\\
&  \leq C_7\left|t_{1}-t_{2}\right|^{\frac{1}{4}},
\label{Li-distance}
\end{align}
where the constant $C_7>0$ depends on the estimates \eqref{sv2} and \eqref{sv1}.
Thanks to the strict separation property of the initial datum $\varphi_0$, there must exist some $T^{*} \in\left(0, T_0\right]$ such that
\begin{equation}
\left\|\varphi(t)\right\|_{L^\infty} \leq 1-\frac{1}{2} \delta_{0}, \quad \forall\, t \in\left[0, T^{*} \right].
\label{sepa}
\end{equation}
Indeed, we can take
$$
T^{*}=\sup_{t\in[0,T_0]}
\left\{\left\|\varphi(s)\right\|_{L^\infty} < 1-\frac{1}{2}\delta_{0},\ \ \forall\,s\in [0,t)\right\}.
$$
By continuity, it holds $T^*>0$. If $T^*=T_0$, then the conclusion \eqref{sepa} holds. If $T^*\in (0,T_0)$, we infer from the definition that
$\|\varphi\left(T^*\right)\|_{L^\infty}=1 - \delta_{0}/2$. On the other hand, it follows from \eqref{Li-distance} that
$$
C_7 |T^*|^{\frac{1}{4}}\geq \|\varphi\left(T^*\right)-\varphi_0\|_{L^{\infty}} \geq
\|\varphi\left(T^*\right)\|_{L^{\infty}}-\|\varphi_0\|_{L^{\infty}} =\frac{1}{2}\delta_0,
$$
which implies
\begin{align}
T^*\geq \left(\frac{\delta_0}{2C_7}\right)^4.\label{Tss}
\end{align}
Therefore, $T^{*}$ has a strictly positive lower bound that depends on the coefficients of the system, $\Omega$, $\overline{\varphi_0}$, $\overline{\sigma_0}$, $\|\bm{v}_0\|_{\bm{H}^1}$, $\|\varphi_0\|_{H^1}$, $\|\widetilde{\mu}_0\|_{H^1}$, $\|\sigma_0\|_{H^1}$, $\max_{r\in[-1,1]}|\varPsi(r)|$ and $\delta_0$.

From (H2), \eqref{f4.d}, \eqref{sv2}, the strict separation property \eqref{sepa} and the elliptic estimate, we can deduce that
\begin{align}
	\left\|\varphi(t)\right\|_{L^\infty(0,T^*;H^{3}(\Omega))}
	&  \leq C \mathop{\mathrm{ess\ sup}}_{t\in[0,T^*]} \left(\left\|\varPsi^{\prime}\left(\varphi(t)\right)\right\|_{H^{1}} +\left\|\mu(t)\right\|_{H^{1}}+\left\|\sigma(t)\right\|_{H^{1}}+\left\|\varphi(t)\right\|_{H^{1}}\right) \notag\\
& \leq C.\notag
\end{align}
Thus, we have $\varphi \in L^{\infty}\big(0, T^*; H^{3}(\Omega)\big)$.
Besides, in view of \eqref{f4.d}, we infer from (H4), \eqref{sv1}, \eqref{v} and \eqref{sepa} that $\varphi \in L^{2}\big(0, T^* ; H^{4}(\Omega)\big)$ and $\mu \in H^{1}\big(0, T^* ;(H^1(\Omega))'\big)$. By the interpolation theory \cite{simon}, we further obtain $\varphi \in C\left([0,T^{*}] ; H^{2}(\Omega)\right)$ and $\mu \in C\left([0,T^{*}]; H^{1}(\Omega)\right)$.
 Then for any $0\leq t_1\leq t_2\leq T^*$, we can deduce from (H2) and \eqref{sepa} that
\begin{align}
&\|\varPsi'(\vp(t_1))-\varPsi'(\vp(t_2))\|_{H^1}\notag \\
&\quad  \leq \int_0^1 \|\varPsi''(s\vp(t_1)+(1-s)\vp(t_2))(\vp(t_1)-\vp(t_2))\|_{H^1}\,\mathrm{d}s\notag\\
&\quad \leq \max_{r\in[-1+\frac{\delta_0}{2},1-\frac{\delta_0}{2}]}|\varPsi''(r)|\|\vp(t_1)-\vp(t_2)\|_{H^1}\notag\\
&\qquad + \max_{r\in[-1+\frac{\delta_0}{2},1-\frac{\delta_0}{2}]}|\varPsi^{(3)}(r)| \big(\|\nabla \vp(t_1)\|_{\bm{L}^4} +\|\nabla \vp(t_2)\|_{\bm{L}^4}\big)\|\vp(t_1)-\vp(t_2)\|_{L^4}\notag\\
&\quad \leq C \|\vp(t_1)-\vp(t_2)\|_{H^1}.
\end{align}
This yields $\varPsi'(\vp)\in C([0,T^*];H^1(\Omega))$. Then by comparison in the equation \eqref{f4.d} and using the continuity properties of $\mu$ and $\sigma$, we find that $\Delta \vp \in C([0,T^*];H^1(\Omega))$ as well. Finally, from the elliptic estimate for the Neumann problem, we can deduce that $\vp\in C([0,T^*];H^3(\Omega))$.

The proof of Theorem \ref{ls} is complete.
\hfill $\square$

%%%%%%%%%%%%%%%%%%%%%%%%%%%%%%%%%%%%%%%%%%%%%%%%%
\section{The Stationary Problem and \L ojasiewicz--Simon Inequality}\label{sec:stapro}
\setcounter{equation}{0}
In this section, we first prove Proposition \ref{prop-le2} on the existence and regularity of energy minimizers of the free energy $\mathcal{F}(\varphi, \sigma)$. Next, we derive an extended {\L}ojasiewicz--Simon type inequality for the coupled system. The assumption (H3) on the analyticity of the potential function $\varPsi$ plays an essential role in the proof.

\subsection{Proof of Proposition \ref{prop-le2}}

\textbf{Part 1. Existence.} The proof for the existence of an energy  minimizer can be done by applying the direct method of the calculus of variations. Using (H2) and an estimate similar to \eqref{couple1}, we find that $\mathcal{F}(\varphi,\sigma)$ (recall \eqref{fe}) is bounded from below on $\mathcal{Z}_{m_1,m_2}$ with a lower bound only depending on  $\chi$, $\beta$, $\theta_0$, $\varPsi_0$ and $\Omega$. On the other hand, it is easy to check that  $\mathcal{F}(\varphi,\sigma)\to+\infty$ if $\|(\varphi,\sigma)\|_{H^1\times L^2}\to+\infty$. As a consequence, $\mathcal{F}$ has a bounded minimizing sequence $\{(\varphi_j,\sigma_j)\}_{j\in \mathbb{Z}^+} \subset \mathcal{Z}_{m_1,m_2}$  such that
$$
\mathcal{F}(\varphi_j,\sigma_j)\to  \inf_{(\varphi,\sigma)\in  \mathcal{Z}_{m_1,m_2}} \mathcal{F}(\varphi,\sigma),\quad \text{as}\ j\to+\infty.
$$
Hence, there exists a pair $(\varphi_*,\sigma_*)\in  \mathcal{Z}_{m_1,m_2}$  such that $(\varphi_j,\sigma_j)\rightharpoonup (\varphi_*,\sigma_*)$ weakly in $H^1(\Omega)\times L^2(\Omega)$ as
$j\to +\infty$ along a non-relabeled subsequence.
From  \eqref{fe}, (H2), the weak lower
semi-continuity of norms and \cite[Lemma 4.1]{A2007}, we find that
\be
	\mathcal{F}_{\mathrm{conv}}(\varphi,\sigma)
\triangleq\int_{\Omega} \left(\frac{1}{2}|\nabla \varphi|^2 + \varPsi_0(\varphi)+\frac{1}{2}|\sigma|^2
\right) \mathrm{d}x,\label{fe-1}
	\ee
is a proper, lower semi-continuous and convex functional on $\mathcal{Z}_{m_1,m_2}$. Besides, it follows from the compact embedding $H^1(\Omega)\hookrightarrow\hookrightarrow L^2(\Omega)$ that $\varphi_j\to \varphi_*$ strongly in $L^2(\Omega)$ as $j\to +\infty$. Therefore, as $j\to +\infty$, the remaining (possibly non-convex) part of $\mathcal{F}$ satisfies
\begin{align*}
\left|\chi\int_{\Omega} (\sigma_j\varphi_j-\sigma_*\varphi_*)\,\mathrm{d}x \right|
& \leq |\chi|\| \sigma_j\|\|\varphi_j-\varphi_*\| +|\chi|\left|\int_\Omega \varphi_*(\sigma_j-\sigma_*)\,\mathrm{d}x\right|\to 0,
\end{align*}
$$
\left|\frac{\theta_0}{2}\int_{\Omega} (\varphi_j^2-\varphi_*^2)\,\mathrm{d}x\right|
\leq \frac{\theta_0}{2}\big(\|\varphi_j\|+\|\varphi_*\|\big) \|\varphi_j-\varphi_*\|\to 0,
$$
$$
\left|\frac{\beta}{2}\int_\Omega \big(|\nabla\mathcal{N}(\varphi_j-\overline{\varphi_j})|^2 -|\nabla\mathcal{N}(\varphi_*-\overline{\varphi_*})|^2\big)\mathrm{d}x\right|
\leq   C|\beta| \big(\|\varphi_j\|+\|\varphi_*\|\big) \|\varphi_j-\varphi_*\|\to 0.
$$
As a consequence, we can conclude that
$$
\mathcal{F}(\varphi_*,\sigma_*)
\leq \liminf_{j\to+\infty}\mathcal{F}(\varphi_j,\sigma_j)
=
\inf_{(\varphi,\sigma)\in  \mathcal{Z}_{m_1,m_2}} \mathcal{F}(\varphi,\sigma).
$$
Namely, $\mathcal{F}$ attains its
minimum at $(\varphi_*,\sigma_*)$ in $\mathcal{Z}_{m_1,m_2}$. \medskip

\textbf{Part 2. Regularity.} The regularity of a minimizer can be derived by a dynamic approach (cf. \cite{MZ04,A2007} for the Cahn--Hilliard equation).
To this end, let us consider the following viscous Cahn--Hilliard--Oono system with chemotaxis:
\begin{subequations}
\begin{alignat}{3}
&\partial_t \varphi=\Delta \mu,\ &\textrm{in}\ \Omega\times(0,+\infty),\label{of1.a} \\
&\mu=\gamma\partial_t\varphi - \Delta \varphi+\varPsi'(\varphi)-\chi \sigma+ \beta\mathcal{N}(\varphi-\overline{\varphi}), \quad\quad &\textrm{in}\ \Omega\times(0,+\infty),\label{of4.d} \\
&\partial_t \sigma= \Delta (\sigma-\chi\varphi),\ &\textrm{in}\ \Omega\times(0,+\infty), \label{of2.b}
\end{alignat}
\end{subequations}
subject to the boundary and initial conditions
\begin{alignat}{3}
&{\partial}_{\bm{n}}\varphi={\partial}_{\bm{n}}\mu={\partial}_{\bm{n}}\sigma=0,\qquad\qquad &\textrm{on}& \   \partial\Omega\times(0,+\infty),\label{oboundary}\\
&\varphi|_{t=0}=\varphi_{0},\ \ \ \sigma|_{t=0}=\sigma_{0}, \qquad &\textrm{in}&\ \Omega.
\label{oini0}
\end{alignat}
In \eqref{of4.d}, $\gamma>0$ is a fixed constant for the additional viscous term in the chemical potential.
For any initial datum $(\varphi_0, \sigma_0)\in \mathcal{Z}_{m_1,m_2}$, it follows from \cite[Theorem 2.1]{H2} that problem \eqref{of1.a}--\eqref{oini0} admits a unique global weak solution $(\varphi,\sigma)$ on $[0,+\infty)$. Moreover, the weak solution generates a strongly continuous
semigroup $\mathcal{S}(t): \mathcal{Z}_{m_1,m_2}\to \mathcal{Z}_{m_1,m_2}$ such that $\mathcal{S}(t)(\varphi_0,\sigma_0) =(\varphi(t),\sigma(t))$ for all $t\geq 0$. Noticing that for global weak solutions, the free energy  $\mathcal{F}(\varphi,\sigma)$ satisfies the energy identity
\begin{align}
  \mathcal{F}(\varphi(t),\sigma(t))
+ \int_0^t\left(\|\nabla \mu(\tau)\|^2 +\|\nabla(\sigma-\chi\varphi)(\tau)\|^2 +\gamma\|\partial_t\varphi(\tau)\|^2\right) \mathrm{d}\tau= \mathcal{F}(\varphi_0,\sigma_0),
\label{Lya-1}
\end{align}
for all $t> 0$ and $\|\partial_t\sigma\|_{V_0'}=\|\nabla(\sigma-\chi\varphi)\|$, we can deduce that $\mathcal{F}(\varphi,\sigma)$ is a strict Lyapunov function.

Let $(\varphi_*,\sigma_*)$ be a local minimizer of $\mathcal{F}$ in $\mathcal{Z}_{m_1,m_2}$. We take the initial datum $(\varphi_0, \sigma_0)=(\varphi_*,\sigma_*)$ in \eqref{oini0}. Since the mapping $t \mapsto \mathcal{F}(\varphi(t), \sigma(t))$ is continuous, it holds $\mathcal{F}(\varphi(t),\sigma(t)) =\mathcal{F}(\varphi_*,\sigma_*)$ for all $t\geq 0$, which together with \eqref{Lya-1} yields that $\|\partial_t\varphi(t)\|=\|\partial_t\sigma(t)\|_{V_0'}=0$ for almost all $t>0$. Thus, $(\varphi_*,\sigma_*)$ is a stationary point to the evolution problem \eqref{of1.a}--\eqref{oini0}, that is, $\mathcal{S}(t)(\varphi_*,\sigma_*) =(\varphi_*,\sigma_*)$ for all $t\geq 0$. As a consequence, from the regularity property of global weak solutions to problem \eqref{of1.a}--\eqref{oini0} (see \cite[Corollary 2.1]{H2}), we can conclude that $(\varphi_{*},\sigma_{*})\in H^2_N(\Omega)\times H^2_N(\Omega)$ and the strict separation property \eqref{ps1s} holds.

The proof of Proposition \ref{prop-le2} is complete.
\hfill $\square$

\medskip
Since the energy functional $\mathcal{F}$ may be non-convex, the uniqueness of its minimizer is in general not expected.
On the other hand, the proof of Proposition \ref{prop-le2} implies that every energy minimizer of $\mathcal{F}$ in $\mathcal{Z}_{m_1,m_2}$ belongs to the $\omega$-limit set of itself. Thus, recalling \cite[Theorem 2.5]{H2} on the characterization of $\omega$-limit sets for problem \eqref{of1.a}--\eqref{oini0}, we can draw the following conclusion:
\begin{corollary}\label{sta-str}
Let the assumptions of Proposition \ref{prop-le2} be satisfied. Any local minimizer of $\mathcal{F}$ in $\mathcal{Z}_{m_1,m_2}$ is a strong solution to the following stationary problem
	\begin{subequations}
		\begin{alignat}{3}
		&-\Delta \varphi_\mathrm{s} +\varPsi^{\prime}(\varphi_\mathrm{s}) -\chi\sigma_\mathrm{s} +\beta\mathcal{N}(\varphi_\mathrm{s}-\overline{\varphi_\mathrm{s}}) =\overline{\varPsi^{\prime}(\varphi_\mathrm{s})}
		-\chi\overline{\sigma_\mathrm{s}},\quad
		&\ \text{in } \Omega,   \label{s5bchv}\\
		& \Delta (\sigma_\mathrm{s}-\chi\varphi_\mathrm{s})=0,
		&\ \text{in } \Omega,   \label{s5f2.b}  \\
		&\partial_{\bm{n}} \varphi_\mathrm{s} = \partial_{\bm{n}} \sigma_\mathrm{s}=0,
		&\ \text{on } \partial \Omega, \label{s5cchv}\\
		&\text{subject to the constraints}\quad \overline{\varphi_\mathrm{s}}=m_1, \quad   \overline{\sigma_\mathrm{s}}=m_2.
		& \label{s5dchv}
		\end{alignat}
	\end{subequations}
\end{corollary}

\subsection{An extended {\L}ojasiewicz--Simon inequality}

The main result of this section reads as follows.

\begin{theorem}[\L ojasiewicz--Simon inequality]
\label{LSmain}
Let $m_1\in (-1,1)$ and $m_2,\,\beta,\,\chi \in \mathbb{R}$ be given constants. Besides, we suppose that $\Omega$ is a bounded domain in $\mathbb{R}^3$ with boundary $\partial\Omega$ of class $C^3$ and the assumptions (H2) and (H3) are satisfied. Define the set
$$
\mathcal{X}_{m_1,m_2}=\big\{(\varphi,\sigma)\in \big(H^2_N(\Omega)\times H^2_N(\Omega)\big)\cap \mathcal{Z}_{m_1,m_2}\ \big|\ (\varphi,\sigma)\ \text{satisfies}\ \eqref{s5bchv}-\eqref{s5dchv}\big\}.
$$
For every $(\varphi_{\mathrm{s}},\sigma_{\mathrm{s}}) \in	\mathcal{X}_{m_1,m_2}$, there exist constants $\kappa\in(0,1/2)$, $\ell \in (0,1)$ and $C\geq 1$ such that
\begin{align}
|\mathcal{F}(\varphi,\sigma) -\mathcal{F}(\varphi_\mathrm{s},\sigma_\mathrm{s})|^{1-\kappa}
 & \leq C  \|-\Delta \varphi+ \varPsi'(\varphi)
- \chi(\sigma-\overline{\sigma})
+{\beta}\mathcal{N}(\varphi-\overline{\varphi})-\overline{ \varPsi'(\varphi)}\|\notag\\
&\quad +C  \|\sigma-\chi\varphi- \overline{\sigma -\chi\varphi}\| +C|\overline{\varphi}-m_1|^{1-\kappa},
\label{lojas}
\end{align}
for any $(\varphi,\sigma)\in H^2_N(\Omega)\times L^2(\Omega)$ satisfying $\overline{\sigma}=m_2$ and $\|\varphi-\varphi_{\mathrm{s}}\|_{H^2}+\|\sigma-\sigma_\mathrm{s}\|<\ell$.
Here, the constants $\kappa$, $\ell$ and $C$ depend on $\varphi_\mathrm{s}$, $\sigma_\mathrm{s}$, $m_1$, $m_2$, $\beta$, $\chi$ and $\Omega$, but not on $(\varphi, \sigma)$.
\end{theorem}

\begin{proof} Since every stationary solution $(\varphi_\mathrm{s},\sigma_\mathrm{s})\in \mathcal{X}_{m_1,m_2}$ is also a (strong) solution to the evolution problem \eqref{of1.a}--\eqref{oini0} with the initial datum given by itself, from the uniqueness of solutions to \eqref{of1.a}--\eqref{oini0}, we again have  $\mathcal{S}(t)(\varphi_\mathrm{s},\sigma_\mathrm{s}) =(\varphi_\mathrm{s},\sigma_\mathrm{s})$ for all $t\geq 0$. Moreover, by \cite[Corollary 2.1]{H2}, there exists some $\delta_\mathrm{s}\in (0,1)$ such that
\begin{align}
\|\varphi_\mathrm{s}\|_{C(\overline{\Omega})}\leq 1-\delta_\mathrm{s}.
\label{sta-sepa}
\end{align}
Consider the energy
\begin{align}
\mathcal{F}_1(\varphi)=\int_{\Omega} \Big( \frac{1}{2}|\nabla \varphi|^2 + \varPsi_0(\varphi)\Big) \mathrm{d}x,\quad \label{F2}
\end{align}
where the potential function $\varPsi_0$ satisfies (H2).
From \cite[Lemma 4.1]{A2007}, we find that $\mathcal{F}_1$ is a proper, lower semi-continuous and convex functional on $\mathcal{D}(\mathcal{F}_1)=\{\vp\in H^1(\Omega)\ |\ \varPsi_0(\varphi) \in L^1(\Omega),\ \overline{\varphi}=m_1\}$. Besides, it follows from \cite[Theorem 4.3]{A2007} that the subgradient of $\mathcal{F}_1$ is given by
$$
\partial \mathcal{F}_1(\varphi) = -\Delta \varphi + \varPsi_0'(\varphi)- \overline{\varPsi_0'(\varphi)}
\in L_0^2(\Omega),
$$
for all
$$
\varphi\in \mathcal{D}(\partial\mathcal{F}_1)\triangleq
\big\{\varphi\in H_N^2(\Omega)\ \big|\ \varPsi'_0(\varphi)\in L^2(\Omega),\ \varPsi''_0(\varphi)|\nabla \varphi|^2\in L^1(\Omega),\ \overline{\varphi}=m_1\big\}.
$$
Thanks to the strict separation property \eqref{sta-sepa} and (H2),  for any $(\varphi_\mathrm{s},\sigma_\mathrm{s})\in \mathcal{X}_{m_1,m_2}$, we have
$\varphi_\mathrm{s}\in  \mathcal{D}(\partial\mathcal{F}_1)$.
From \eqref{s5bchv}--\eqref{s5dchv}, we also find that $(\varphi_\mathrm{s},\sigma_\mathrm{s})$
satisfies the reduced system:
 \begin{subequations}
		\begin{alignat}{3}
		&-\Delta \varphi_\mathrm{s} +  \big(\varPsi^{\prime}(\varphi_\mathrm{s}\big) - \overline{\varPsi^{\prime}(\varphi_\mathrm{s})}\big) -\chi(\sigma_\mathrm{s} - \overline{\sigma_\mathrm{s}}) +\beta\mathcal{N}(\varphi_\mathrm{s}-\overline{\varphi_\mathrm{s}}) =0,\quad
		&\ \text{a.e. in } \Omega,   \label{crit-1}\\
		& \sigma_{\mathrm{s}}-\chi\varphi_\mathrm{s}= m_2-\chi m_1,\quad
		&\ \text{a.e. in } \Omega. \label{crit-2}
		\end{alignat}
	\end{subequations}

We now divide the proof of Theorem \ref{LSmain} into several steps. \medskip

\textbf{Step 1. {\L}--S inequality for a reduced energy with mass constraint}.
In view of \eqref{crit-1}--\eqref{crit-2}, we define the following reduced energy
\begin{align}
\widetilde{\mathcal{F}}(\varphi)
&=\int_\Omega \Big( \frac{1}{2}|\nabla \varphi|^2+ \varPsi(\varphi)\Big) \mathrm{d}x
-\frac{\chi^2}{2}\|\varphi\|^2
+\frac{\beta}{2}\|\nabla\mathcal{N}(\varphi-\overline{\varphi})\|^{2}. \label{f1}
\end{align}
From the strict separation property \eqref{sta-sepa} for $\vp_\mathrm{s}$ and the Sobolev embedding $H^2(\Omega)\hookrightarrow L^\infty(\Omega)$, we infer  that for any $\varphi\in H^2(\Omega)$ satisfying $\|\varphi-\varphi_\mathrm{s}\|_{H^2}\ll 1$, a slightly relaxed strict separation property still holds for $\vp$, for instance,
\begin{align}
\|\varphi\|_{C(\overline{\Omega})}\leq 1-\frac{2}{3}\delta_\mathrm{s}.\label{re-sep1}
\end{align}
The uniform strict separation property \eqref{re-sep1} rules out the possible singularities of $\varPsi'(\varphi)$ at $\pm 1$ for functions in a small $H^2$-neighbourhood of $\vp_\mathrm{s}$. Hence, we are able to apply the abstract result \cite[Theorem 6]{GG2006} or \cite[Corollary 5.2]{Rupp} (cf. also \cite[Proposition 6.3]{A2007} for the Cahn--Hilliard equation) to  derive the following gradient inequality of {\L}ojasiewicz--Simon type for $\widetilde{\mathcal{F}}$.
\bl\label{LS}  \label{ls2}
Under the assumptions of Theorem \ref{LSmain}, for any given $ (\varphi_{\mathrm{s}},\sigma_{\mathrm{s}}) \in	\mathcal{X}_{m_1,m_2}$,
there exist constants $\widetilde{\kappa}\in(0,1/2)$, $\widetilde{\ell}\in (0,1)$ and $\widetilde{C}\geq 1$ such that the following inequality holds
\begin{align}
&|\widetilde{\mathcal{F}}(\widetilde{\varphi}) -\widetilde{\mathcal{F}}(\varphi_{\mathrm{s}})|^{1-\widetilde{\kappa}}
 \le
\widetilde{C}\|-\Delta \widetilde{\varphi}+\varPsi'(\widetilde{\varphi})
- \chi^2(\widetilde{\varphi}-\overline{\widetilde{\varphi}})
+{\beta}\mathcal{N}(\widetilde{\varphi}-\overline{\widetilde{\varphi}})-\overline{ \varPsi'(\widetilde{\varphi})}\|_{V_0'},
\label{loj}
\end{align}
for any $\widetilde{\varphi}\in H_N^{2}(\Omega)$ satisfying the mass constraint $\overline{\widetilde{\varphi}}=m_1$ and $\|\widetilde{\varphi}-\varphi_{\mathrm{s}}\|_{H^2}<\widetilde{\ell} $.
The constants $\widetilde{\kappa}$, $\widetilde{\ell}$, $\widetilde{C}$ depend on $\varphi_\mathrm{s}$, $m_1$, $m_2$, $\beta$, $\chi$ and $\Omega$, but not on $\widetilde{\varphi}$.
\el

\textbf{Step 2. Recover the case of two variables}. We observe that for any $(\widetilde{\varphi}, \sigma)\in \mathcal{Z}_{m_1,m_2}$, it holds
\begin{align}
\mathcal{F}(\widetilde{\varphi},\sigma)
& =\widetilde{\mathcal{F}}(\widetilde{\varphi})
+ \frac{1}{2}  \int_\Omega (\sigma -\chi\widetilde{\varphi})^2\,\mathrm{d}x,\notag\\
& =\widetilde{\mathcal{F}}(\widetilde{\varphi})
+ \frac{1}{2}  \int_\Omega \big[(\sigma -\chi\widetilde{\varphi})- \overline{\sigma -\chi\widetilde{\varphi}}\,\big]^2\,\mathrm{d}x
+ \frac{1}{2} |\Omega| \big(\overline{\sigma -\chi\widetilde{\varphi}}\big)^2 \notag  \\
& =\widetilde{\mathcal{F}}(\widetilde{\varphi})
+ \frac{1}{2}  \int_\Omega \big[(\sigma -\chi\widetilde{\varphi})- \overline{\sigma -\chi\widetilde{\varphi}}\,\big]^2\,\mathrm{d}x
+ \frac{1}{2} |\Omega| \big(\overline{\sigma_\mathrm{s}  -\chi \varphi_\mathrm{s}}\big)^2. \label{f2}
\end{align}
Besides, \eqref{crit-2} implies that $\| \sigma_{\mathrm{s}}-\chi\varphi_\mathrm{s} - \overline{\sigma_{\mathrm{s}} -\chi \varphi_\mathrm{s}}\|=0$.
As a consequence, we can apply Lemma \ref{ls2} to conclude that for any $(\widetilde{\varphi},\sigma)\in  \big(H^2_N(\Omega)\times L^2(\Omega)\big) \cap \mathcal{Z}_{m_1,m_2}$ satisfying the assumption $\|\widetilde{\varphi}-\varphi_{\mathrm{s}}\|_{H^2}+ \|\sigma-\sigma_{\mathrm{s}}\|<\widetilde{\ell}$,
it holds
\begin{align}
&|\mathcal{F}(\widetilde{\varphi},\sigma) -\mathcal{F}(\varphi_\mathrm{s},\sigma_\mathrm{s})|^{1-\widetilde{\kappa}}\notag\\
&\quad  \leq |\widetilde{\mathcal{F}}(\widetilde{\varphi}) -\widetilde{\mathcal{F}}(\varphi_{\mathrm{s}})|^{1-\widetilde{\kappa}}
  + \| (\sigma -\chi\widetilde{\varphi})- \overline{\sigma -\chi\widetilde{\varphi}}\|^{2(1-\widetilde{\kappa})}\notag\\
  &\quad \leq \widetilde{C}\|-\Delta \widetilde{\varphi} + \varPsi'(\widetilde{\varphi})
- \chi^2(\widetilde{\varphi}-\overline{\widetilde{\varphi}})
+{\beta}\mathcal{N}(\widetilde{\varphi}-\overline{\widetilde{\varphi}})-\overline{ \varPsi'(\widetilde{\varphi})}\|_{V_0'}\notag \\
&\qquad  + C \| (\sigma -\chi\widetilde{\varphi})- \overline{\sigma -\chi\widetilde{\varphi}}\|\notag \\
&\quad \leq \widetilde{C}\|-\Delta \widetilde{\varphi} + \varPsi'(\widetilde{\varphi})
- \chi(\sigma-\overline{\sigma})
+{\beta}\mathcal{N}(\widetilde{\varphi}-\overline{\widetilde{\varphi}})-\overline{ \varPsi'(\widetilde{\varphi})}\|_{V_0'} \notag \\
&\qquad + \widetilde{C}|\chi| \| (\sigma -\chi\widetilde{\varphi})- \overline{\sigma -\chi\widetilde{\varphi}}\|_{V_0'} + C \| (\sigma -\chi\widetilde{\varphi})- \overline{\sigma -\chi\widetilde{\varphi}}\|\notag \\
&\quad \leq C_8 \|-\Delta \widetilde{\varphi} + \varPsi'(\widetilde{\varphi})
- \chi(\sigma-\overline{\sigma})
+{\beta}\mathcal{N}(\widetilde{\varphi}-\overline{\widetilde{\varphi}})-\overline{ \varPsi'(\widetilde{\varphi})}\|\notag\\
&\qquad +C_8 \|\sigma-\chi\widetilde{\varphi}- \overline{\sigma -\chi\widetilde{\varphi}}\|,
\label{ls-4}
\end{align}
where $\widetilde{\kappa}$, $\widetilde{\ell}$ are determined as in Lemma \ref{ls2} and $C_8\geq \widetilde{C}\geq 1$ depends on $\varphi_\mathrm{s}$, $\sigma_\mathrm{s}$, $m_1$, $m_2$, $\beta$, $\chi$ as well as $\Omega$.
\medskip

\textbf{Step 3. The general case with possible change of mass}. Now for any $\varphi\in H^2_N(\Omega)$, we make the following change of variable (shifting up to a constant)
$$
\widetilde{\varphi}=\varphi-(\overline{\varphi}-m_1).
$$
Thus, $\overline{\widetilde{\varphi}}=m_1=\overline{\varphi_\mathrm{s}}$ as required in Step 2 and $\varphi-\overline{\varphi}= \widetilde{\varphi}-\overline{\widetilde{\varphi}}$.
Set $\ell=\widetilde{\ell}/2$, where $\widetilde{\ell}\in (0,1)$ is determined as in Lemma \ref{ls2}. For any $(\varphi, \sigma)\in H^2_N(\Omega)\times L^2(\Omega)$ satisfying $\|\varphi-\varphi_\mathrm{s}\|_{H^2}+ \|\sigma-\sigma_{\mathrm{s}}\|<\ell$, it holds
\begin{align*}
\|\widetilde{\varphi}-\varphi_\mathrm{s}\|_{H^2}+ \|\sigma-\sigma_{\mathrm{s}}\|
& \leq \|\varphi-\varphi_\mathrm{s}\|_{H^2}+ \|\overline{\varphi}-m_1\|_{H^2}+ \|\sigma-\sigma_{\mathrm{s}}\| \notag\\
&\leq 2\|\varphi-\varphi_\mathrm{s}\|_{H^2}+ \|\sigma-\sigma_{\mathrm{s}}\| < \widetilde{\ell}.
\end{align*}
Besides, we may adjust the positive constant $\widetilde{\ell}$ properly smaller if necessary so that (recall \eqref{sta-sepa})
\begin{align}
\|\varphi\|_{C(\overline{\Omega})}\leq 1-\frac{2}{3}\delta_\mathrm{s} \quad \text{and}\quad \|\widetilde{\varphi}\|_{C(\overline{\Omega})}\leq 1-\frac{2}{3}\delta_\mathrm{s}.\label{re-sep2}
\end{align}
On one hand, for $(\widetilde{\varphi}, \sigma)$ we can derive the inequality \eqref{ls-4} as in Step 2. On the other hand, using the strict separation property \eqref{re-sep2}, (H2) and the constraint $\overline{\sigma}=m_2$, we can obtain the following estimates
\begin{align}
|\mathcal{F}(\varphi,\sigma) -\mathcal{F}(\widetilde{\varphi},\sigma)|
& = \left|\int_{\Omega} \big( \varPsi(\varphi)-\varPsi(\widetilde{\varphi})
-\chi\sigma(\varphi-\widetilde{\varphi}) \big) \mathrm{d}x\right|
\notag\\
&\leq |\Omega|\max_{r\in[-1+\frac{\delta_\mathrm{s}}{3}, 1-\frac{\delta_\mathrm{s}}{3}]}|\varPsi'(r)| |\overline{\varphi}-m_1| +|\chi||\Omega||m_2||\overline{\varphi}-m_1|
\notag\\
&\leq C|\overline{\varphi}-m_1|,
\label{diffE}
\end{align}
and
\begin{align}
\|\varPsi'(\varphi)-\overline{ \varPsi'(\varphi)}-\varPsi'(\widetilde{\varphi})+\overline{ \varPsi'(\widetilde{\varphi})}\|
& \leq 2 \max_{r\in[-1+\frac{\delta_\mathrm{s}}{2}, 1-\frac{\delta_\mathrm{s}}{2}]}|\varPsi''(r)|\|\varphi -\widetilde{\varphi}\|
\notag\\
&\leq C|\overline{\varphi}-m_1|.
\notag
\end{align}
From the above estimates and \eqref{ls-4}, we can deduce that
\begin{align}
&|\mathcal{F}(\varphi,\sigma) -\mathcal{F}(\varphi_\mathrm{s},\sigma_\mathrm{s})|^{1-\widetilde{\kappa}}\notag\\
&\quad \leq  |\mathcal{F}(\varphi,\sigma) -\mathcal{F}(\widetilde{\varphi},\sigma)|^{1-\widetilde{\kappa}}
+ |\mathcal{F}(\widetilde{\varphi},\sigma) -\mathcal{F}(\varphi_\mathrm{s},\sigma_\mathrm{s})|^{1-\widetilde{\kappa}}\notag\\
&\quad \leq  C_8 \|-\Delta \widetilde{\varphi} + \varPsi'(\widetilde{\varphi})
- \chi(\sigma-\overline{\sigma})
+{\beta}\mathcal{N}(\widetilde{\varphi} -\overline{\widetilde{\varphi}})-\overline{ \varPsi'(\widetilde{\varphi})}\|\notag\\
&\qquad +C_8 \|\sigma-\chi\widetilde{\varphi}- \overline{\sigma -\chi\widetilde{\varphi}}\| + C|\overline{\varphi}-m_1|^{1-\widetilde{\kappa}}
\notag\\
&\quad \leq C_9 \|-\Delta \varphi + \varPsi'(\varphi)
- \chi(\sigma-\overline{\sigma})
+{\beta}\mathcal{N}(\varphi-\overline{\varphi})-\overline{ \varPsi'(\varphi)}\|\notag\\
&\qquad +C_9 \|\sigma-\chi\varphi- \overline{\sigma -\chi\varphi}\| + C_9|\overline{\varphi}-m_1|^{1-\widetilde{\kappa}},
\label{ls-5}
\end{align}
for some constant $C_9\geq C_8$, depending also on $\delta_\mathrm{s}$ but not on $(\varphi, \sigma)$. Thus, we can draw the conclusion \eqref{lojas} by taking $\kappa=\widetilde{\kappa}$ and $C=C_9$ as in \eqref{ls-5}.

The proof of Theorem \ref{LSmain} is complete.
\end{proof}
\begin{remark}\label{rem:kappa}
In general, we only have $\kappa\in (0,1/2)$. In some special cases, one can achieve $\kappa=1/2$ (see e.g., \cite{C2003}). Besides, it is easy to verify that if the inequality \eqref{lojas} holds for some $\kappa$, then it is valid for any $\kappa'\in (0,\kappa)$ (with possible modification for the constant $C$).
\end{remark}

%%%%%%%%%%%%%%%%%%%%%%%%%%%%%%%%%%%%%%%%%%%%%%%%%
\section{Global Strong Solutions}
\setcounter{equation}{0} \label{sec:glostr}

In this section, we prove Theorem \ref{3main}. Based on the local well-posedness result Theorem \ref{ls}, we complete our proof by exploiting the dissipative structure of the coupled system \eqref{f3.c}--\eqref{f2.b} and deriving uniform-in-time estimates under suitable smallness assumptions on the initial data. In this manner, we are able to extend the local strong solution step by step and eventually obtain a global one that is uniformly bounded on $[0,+\infty)$. In what follows, we focus on the more involved case with an arbitrary but fixed parameter $\alpha>0$, while the case $\alpha=0$ (that leads to the mass conservation $\overline{\varphi(t)}\equiv \overline{\varphi_0}$) can be treated with some minor modifications.

\subsection{Local strong solution revisited} \label{loc-re}

Since $(\varphi_{*},\sigma_{*})$ is a given local minimizer of the free energy $\mathcal{F}(\varphi,\sigma)$ in $\mathcal{Z}_{m_1,m_2}$ with $m_1=c_0$, it follows from Definition \ref{min:ene} that there exists some $\lambda_*>0$, for any $(\widetilde{\varphi},\sigma)\in \mathcal{Z}_{c_0,m_2}$ satisfying $\|\widetilde{\varphi}-\varphi_*\|_{H^1} +\|\sigma-\sigma_*\|\le \lambda_*$, it holds
\be
\mathcal{F}(\varphi_*,\sigma_*)\le \mathcal{F}(\widetilde{\varphi},\sigma). \label{mini1}
\ee
Besides, Proposition \ref{prop-le2} yields that
$(\varphi_{*},\sigma_{*})\in \big(H^2_N(\Omega)\times H^2_N(\Omega)\big)\cap \mathcal{Z}_{c_0,m_2}$ and
\be
\|\varphi_{*}\|_{C(\overline{\Omega})} \le 1-\delta_*,  \quad \text{for some}\ \ \delta_*\in (0,1).
\label{m1}
\ee
From the Sobolev embedding theorem in three dimensions, we have
\begin{equation}
\|\varphi\|_{C(\overline{\Omega})} \le C_{S}\|\varphi\|_{H^{2}},\quad \forall\, \varphi\in H^2(\Omega),
\label{m}
\end{equation}
where the constant $C_{S}>0$ only depends on $\Omega$.
Concerning the initial data, we impose the following upper bounds
\begin{align}
&\left\|\boldsymbol{v}_{0}\right\|_{\boldsymbol{H}^{1}} \le M_0, \quad\left\|\varphi_{0}\right\|_{H^{3}} \le M_0,\quad\left\|\sigma_{0}\right\|_{H^{1}} \le M_0, \label{l}\\
&\left\|\boldsymbol{v}_{0}\right\| \le \eta_{1}, \quad\left\|\varphi_{0}-\varphi_{*}\right\|_{H^{2}} \le \eta_{2},\quad\left\|\sigma_{0}-\sigma_{*}\right\| \le \eta_{3},\label{hh}
\end{align}
where $M_0>0$ (allowed to be arbitrarily large) and $\eta_{1}$, $\eta_{2}$, $\eta_{3} \in (0,1)$ are small constants that will be specified later. Hereafter, we simply use the equivalent norm $\left\|\boldsymbol{v}\right\|_{\boldsymbol{H}^{1}} =\left\|\nabla \boldsymbol{v}\right\|$ for any $\bm{v}\in \bm{H}^1_{0,\mathrm{div}}(\Omega)$.
Assuming that
\begin{equation*}
0<\eta_{2}<\min \left\{1, \frac{\delta_*}{3 C_{S}}\right\},
\end{equation*}
then from \eqref{m} and \eqref{hh}, we have the strict separation property for the initial datum:
\begin{align}
\|\varphi_0\|_{C(\overline{\Omega})} &\le\|\varphi_{*}\|_{C(\overline{\Omega})} +\|\varphi_0-\varphi_{*}\|_{C(\overline{\Omega})} \le 1-\delta_*+C_S\eta_2 \le 1-\frac{2\delta_*}{3}.
\label{tt}
\end{align}
According to (H2), we introduce the following notations
\begin{equation}
K_1=\underset{r\in[-1+\frac{\delta_*}{3},\,1-\frac{\delta_*}{3}]}{\max}|\varPsi''(r)|,\quad K_2=\underset{r\in[-1+\frac{\delta_*}{3},\,1-\frac{\delta_*}{3}]}{\max}|\varPsi'(r)|,\quad
K_3=\underset{r\in[-1,\,1]}{\max}|\varPsi(r)|.
\label{ii}
\end{equation}
Below we denote by $L_i$ ($i=1,2,\cdots$) positive constants that may depend on upper bounds of the norms of the initial data, norms of the energy minimizer $(\varphi_*, \sigma_*)$, coefficients of the system and $\Omega$.
\medskip

\textbf{Local strong solution}. In view of \eqref{l} and \eqref{tt}, we can apply Theorem \ref{ls}-(2) to conclude that problem \eqref{f3.c}--\eqref{ini0} admits a unique local strong solution $(\bm{v},\varphi,\mu,\sigma,p)$ on a certain time interval $[0,T_1]$ with the corresponding regularity stated therein. In particular, it holds
\begin{equation}
\|\varphi(t)\|_{C(\overline{\Omega})} \le1-\frac{\delta_*}{3},\quad \forall\,t\in[0,T_1].
\label{r}
\end{equation}
The existence time $T_1>0$ depends on $\left\|\boldsymbol{v}_{0}\right\|_{\boldsymbol{H}^{1}}$,
 $\|\varphi_0\|_{H^1}$, $\left\|\mu_0\right\|_{H^{1}}$ (or equivalently $\|\varphi_0\|_{H^3}$ and $\delta_*$),  $\left\|\sigma_{0}\right\|_{H^{1}}$,  $K_3$, coefficients of the system, $\Omega$ and $\delta_{*}$ (cf. \eqref{Tss}). After checking the argument in Section \ref{gss}, we find that $T_1$ only depends on upper bounds of the norms $\left\|\boldsymbol{v}_{0}\right\|_{\boldsymbol{H}^{1}}$,
  $\left\|\varphi_0\right\|_{H^{3}}$,  $\left\|\sigma_{0}\right\|_{H^{1}}$ and the distance of $\vp_0$ from the pure states $\pm1$ characterized by $\delta_*$, but not on the detailed choice of initial data.
 \medskip

 \textbf{Lower-order estimates}.
 Similar to \eqref{conver1}, we have
\be
	\frac{\mathrm{d}}{\mathrm{d}t} \big(\overline{\varphi}(t)-c_0\big) +\alpha\big(\overline{\varphi}(t)-c_0\big)=0,
	\label{conver1b}
	\ee
which implies
	\be
	\overline{\varphi}(t) = c_0+e^{-\alpha t}\left(\overline{\varphi_{0}}-c_0\right),\quad\forall\,t\in [0,T_1].\label{aver-phib}
	\ee
Besides, the mass of $\sigma$ is conserved (cf. \eqref{aver-sig}), that is,
	\be
	\overline{\sigma}(t)=\overline{\sigma_{0}}=m_2,\quad \forall\, t\in [0,T_1].
	\label{aver-sigb}
	\ee
Recalling \eqref{E}, for the total energy functional
$$
\mathcal{E}(t)= \frac12\|\bm{v}(t)\|^2 + \mathcal{F}(\varphi(t),\sigma(t)),
$$
with $\mathcal{F}$ given as in \eqref{fe}, we can deduce from (H2) the uniform lower bound (cf. \eqref{low-bd1})
\begin{align}
\mathcal{E}(t)&\geq \frac{1}{2}\|\bm{v}(t)\|^2
	+ \frac{1}{2}\|\varphi(t)\|_{H^1}^2
	+\frac{1}{4}\|\sigma(t)\|^2 - L_1,
\label{low-bd1b}
\end{align}
where $L_1>0$ only depends on the coefficients of the system and $\Omega$. Similar to \eqref{BEL2} and \eqref{q3} (taking $\gamma=0$ therein),  we can deduce from  \eqref{aver-phib} that
\begin{align}
\frac{\mathrm{d}}{\mathrm{d}t} \left( \mathcal{E}(t)+ L_2 e^{-\alpha t}|\overline{\varphi_0}-c_0|\right)+\frac12 \mathcal{D}(t)\leq 0,\quad \forall\, t\in (0,T_1),
\label{BEL5}
\end{align}
where
\begin{align}
\mathcal{D}(t)& =\int_\Omega 2\eta(\varphi(t))|D\bm{v}(t)|^2\, \mathrm{d}x + \|\nabla \mu(t)\|^2+\|\nabla\big(\sigma(t)-\chi\varphi(t)\big)\|^2,
\label{D2}
\end{align}
and $L_2>0$  depends only on the coefficients of the system, $\overline{\varphi_0}$ (thus on $\delta_*$ in view of \eqref{tt}), $m_2$ as well as $\Omega$.

Using \eqref{hh}--\eqref{ii} and H\"{o}lder's inequality, we can derive a uniform upper bound for the initial free energy
\begin{align}
\mathcal{F}(\varphi_0,\sigma_0)
&=\int_\Omega \Big(\frac{1}{2}|\nabla \varphi_0|^2 + \varPsi(\varphi_0)\Big)\, \mathrm{d}x
+\frac{1}{2}\|\sigma_0\|^2 -\int_\Omega \chi\sigma_0\varphi_0\, \mathrm{d}x +\frac{\beta}{2}\|\nabla\mathcal{N}(\varphi_0-\overline{\varphi_0})\|^{2}\notag\\
&\le  \frac{1}{2}\|\nabla\varphi_0\|^2 +|\Omega|\underset{r\in[-1,1]}{\max}|\varPsi(r)| +\frac{1+\chi^2}{2}\|\sigma_0\|^2 +\frac{1}{2}\|\varphi_0\|^2
+\frac{|\beta|C}{2} \|\varphi_0-\overline{\varphi_0}\|^{2} \notag\\
&\le \frac{1}{2}(\|\varphi_{*}\|_{H^1}+1)^2+ |\Omega|K_3   +\frac{1+\chi^2}{2}(\|\sigma_{*}\|+1)^2+ \frac{1}{2}|\Omega|+2|\beta|C|\Omega|\notag\\
&=: L_3,\label{jj}
\end{align}
where $L_3>0$ depends on $\chi$, $\beta$, $\|\varphi_{*}\|_{H^1}$, $\|\sigma_{*}\|$, $K_3$ and $\Omega$, but not on the detailed values of $\eta_1$, $\eta_2$, $\eta_3$ and $(\varphi_0,\sigma_0)$. Then it follows from \eqref{hh}, \eqref{BEL5} and \eqref{jj} that
\begin{align}
\mathcal{E}(t) +\frac12\int_0^t \mathcal{D}(\tau)\, \mathrm{d}\tau
&\le \mathcal{E}(0) +L_2|\overline{\varphi_0}-c_0| -L_2 e^{-\alpha t}|\overline{\varphi_0}-c_0|\notag\\
&\le 2(1+L_2+L_3),\qquad \forall\,t\in[0,T_1].
\label{bbb}
\end{align}
The inequality \eqref{bbb} together with \eqref{low-bd1b} yields the lower-order estimate
\begin{align}
& \|\bm{v}(t)\|^2 + \|\varphi(t)\|_{H^1}^2 + \|\sigma(t)\|^2
+\int_0^t \left( \|\nabla\bm{v}(\tau)\|^2 + \|\nabla \mu(\tau)\|^2 +\|\nabla(\sigma(\tau)-\chi\varphi(\tau))\|^2\right) \mathrm{d}{\tau}\notag\\
&\quad  \le 8\max\big\{1, \nu_*^{-1}\big\}(1+L_1+L_2+L_3)\notag\\
&\quad =:L_4>1, \qquad \forall\,t\in [0,T_1]. \label{q}
\end{align}

\textbf{Higher-order estimates}.  Let us proceed to derive higher-order estimates. Recalling \eqref{h1}, we define the function
\begin{align}
\Lambda(t)& =\frac{1}{2}\left\|\nabla \bm{v}(t)\right\|^{2}
+\frac{a_1}{2}\left\|\nabla \mu(t)\right\|^{2} +\frac{1}{2} \|\nabla(\sigma(t)-\chi\varphi(t))\|^{2}\notag \\
&\quad +a_1\int_\Omega \big(\bm{v}(t)\cdot\nabla\varphi(t)\big)\mu(t)\,\mathrm{d}x + a_1 \alpha\big(\overline{\varphi}(t)-c_0\big)\int_\Omega \mu(t)\,\mathrm{d}x,
\label{h1a}
\end{align}
where the constant $a_1\in(0,1)$ is chosen as in \eqref{a1} and only depends on $\Omega$. Using \eqref{aver-phib} and \eqref{q}, slightly modifying the argument for \eqref{qqqa}, \eqref{m3b} and \eqref{m3} (taking $\gamma=0$ therein), we can obtain
\begin{align}
\Lambda(t) &\geq \frac{1}{4}\left\|\nabla \bm{v}(t)\right\|^{2}
+\frac{a_1}{4}\left\|\nabla \mu(t)\right\|^{2}
 +\frac{1}{4} \|\nabla(\sigma(t)-\chi\varphi(t))\|^{2}-L_5\alpha e^{-\alpha t}|\overline{\varphi_0}-c_0|,\label{m3b-b}\\
\Lambda(t) & \leq \frac34\left\|\nabla \bm{v}(t)\right\|^{2} +\frac{3a_1}{4}\left\|\nabla \mu(t)\right\|^{2}
 +\frac{3}{4} \|\nabla(\sigma(t)-\chi\varphi(t))\|^{2}+L_5\alpha e^{-\alpha t}|\overline{\varphi_0}-c_0|,\label{m3-b}
\end{align}
where $L_5>0$ depends on $L_4$, the coefficients of the system and $\Omega$. Setting
\begin{align}
\widetilde{\Lambda}(t)=\Lambda(t)+2L_5\alpha e^{-\alpha t}|\overline{\varphi_0}-c_0|,
\label{LAMB}
\end{align}
we infer from \eqref{l}--\eqref{tt} and \eqref{m3-b} that
\begin{align}
\widetilde{\Lambda}(0)
&=\Lambda(0)+2L_5\alpha|\overline{\varphi_0}-c_0|\notag\\
&\le\frac34\left\|\nabla \boldsymbol{v}_0\right\|^{2}+\frac{3a_1}{4}\left\|\nabla \mu(0)\right\|^{2}+\frac34\left\|\nabla (\sigma_0-\chi\varphi_0)\right\|^{2}+3L_5\alpha |\overline{\varphi_0}-c_0|\notag\\
&\le\frac34\|\nabla\bm{v}_0\|^2
+\frac{3}{4}\big(\|\varPsi''(\varphi_0)\|_{L^\infty}\|\nabla \varphi_0\|+ \|\nabla \Delta \varphi_0\| +\|\chi\nabla\sigma_0\|\notag\\
&\quad+\|\beta\nabla\mathcal{N}(\varphi_0-\overline{\varphi_0})\|\big)^2 +\frac{3}{2} \left\|\nabla \sigma_0\right\|^2+\frac{3\chi^2}{2}\|\nabla \varphi_0\|^2 +3L_5\alpha |\overline{\varphi_0}-c_0|\notag\\
&< 3 M_0^2+\big[ K_1(\|\varphi_{*}\|_{H^1}+1) +M_0+|\chi|M_0+|\beta| C|\Omega|^\frac12\big]^2\notag\\
&\quad + \frac{3}{2}M_0^2+  \frac{3\chi^2}{2}(\|\varphi_{*}\|_{H^1}+1)^2+6L_5 \alpha +1 \notag\\
&=:M_1.\label{p}
\end{align}
Similar to \eqref{m5}, we can derive the following  differential inequality
\begin{align}
&\frac{\mathrm{d}}{\mathrm{d}t}\widetilde{\Lambda}(t)
\le L_6	\widetilde{\Lambda}^5(t),\quad \forall\,t\in(0,T_1),
\label{aab}
\end{align}
where $L_6>0$ depends on $L_4$, the coefficients of the system and $\Omega$. By \eqref{p} and the continuity of $\widetilde{\Lambda}(t)$, we can find some time $T_2\in (0,T_1]$ such that
\begin{align}
\widetilde{\Lambda}(t)\le 2M_1,\quad \forall\,t\in[0,T_2].
\label{aac}
\end{align}
Indeed, we can choose (cf. \eqref{m5c})
\begin{align}
T_2= \min\left\{1,\ \frac{5}{6}T_1,\ \frac{15}{64L_6M_1^4}\right\}.
\label{T2}
\end{align}

From the equation \eqref{f4.d} and the elliptic estimate (cf. \cite[Lemma 6.4]{GGW}), we get
\begin{equation}
\|\varphi(t)\|_{H^3} \le L_7 \big(\|\nabla\mu(t)\|
+\|\varPsi''(\varphi(t))\|_{L^\infty}
\|\nabla \varphi(t)\| +\|\varphi(t)\|_{H^1}+\|\nabla\sigma(t)\|\big),
\label{o}
\end{equation}
where $L_7>0$ only depends on $\chi$, $\beta$ and $\Omega$. Without loss of generality, we take $L_7\geq 1$.
After the above preparations, we now choose the upper bound for higher-order norms of the initial data. Set
$$K_4= \max\{1,6L_5\alpha\}.$$
We take $M_0$ such that
\begin{align}
M_0& \ge    L_7\Big[\big(2a_1^{-\frac12}+2\big)\sqrt{K_4} +\big(K_1+2+|\chi|\big) \sqrt{L_4}\,\Big].
\label{ddM}
\end{align}
Once $M_0$ is fixed,  we can determine the constant $M_1$ (see \eqref{p}) and then the existence time $T_1$ as well as $T_2$ (see \eqref{T2}). Hence, by \eqref{r}, \eqref{q}, \eqref{m3b-b}, \eqref{aac} and \eqref{o}, we can derive the following estimates
\begin{align}
\|\bm{v}(t)\|_{\bm{H}^1}
&\le 2\sqrt{\widetilde{\Lambda}(t)} \le 2\sqrt{2M_1},\quad\forall\,t\in[0,T_2],
\label{421a}
\end{align}
\begin{align}
\|\sigma(t)\|_{H^1}
&\le 2\sqrt{\widetilde{\Lambda}(t)} +|\chi|\|\nabla \varphi(t)\|+\|\sigma(t)\|\notag\\
&\le 2\sqrt{2M_1}+(1+|\chi|)\sqrt{L_4},
\quad\forall\,t\in[0,T_2],
\label{4211a}
\end{align}
and
\begin{align}
\|\varphi(t)\|_{H^3}
&\le L_7\left(2a_1^{-\frac12}\sqrt{\widetilde{\Lambda}(t)}
+(K_1+1)\|\varphi(t)\|_{H^1} +\|\sigma(t)\|_{H^1}\right)\notag\\
&\le L_7\left(\big(2+2a_1^{-\frac12}\big) \sqrt{2M_1}
+(K_1+2+|\chi|)\sqrt{L_4} \right)\notag\\
&=:M_2,\quad\forall\,t\in[0,T_2].
\label{421}
\end{align}

\textbf{Change of the free energy}. For the sake of convenience, we derive an estimate on the change of the free energy along time evolution by using \eqref{ii}, \eqref{r} and \eqref{q}, that is,
\begin{align}
&|\mathcal{F}(\varphi(t),\sigma(t))-\mathcal{F}(\varphi_0,\sigma_0)|\notag\\
 &\quad\le
  \frac{1}{2}\|\nabla(\varphi(t)+\varphi_0)\|
\|\nabla(\varphi(t)-\varphi_0)\| +
 |\Omega|^\frac12\underset{r\in[-1+\frac{\delta_*}{3},\,1-\frac{\delta_*}{3}]}{\max}|\varPsi'(r)|\|\varphi(t)-\varphi_0\|
 \notag\\
&\qquad
+ \frac{1}{2}\|\sigma(t)+\sigma_0\|
\|\sigma(t)-\sigma_0\|  +|\chi|\|\sigma(t)-\sigma_0\|\|\varphi(t)\| +|\chi|\|\sigma_0\|\|\varphi(t)-\varphi_0\|\notag\\
&\qquad+\frac{|\beta|}{2}\|\nabla\mathcal{N}\big(\varphi(t)+\varphi_0-\overline{\varphi(t)+\varphi_0}\big)\| \|\nabla\mathcal{N}\big(\varphi(t)-\varphi_0-\overline{\varphi(t)-\varphi_0}\big)\|\notag\\
&\quad\le L_8\big(\|\varphi(t)-\varphi_0\|_{H^1} +\|\sigma(t)-\sigma_0\|\big),\quad \forall\, t\in [0,T_2],
\label{w}
\end{align}
where the constant $L_8>0$ depends on $L_4$, $K_2$, the coefficients of the system and $\Omega$. Without loss of generality, we take $L_8\geq 1.$

\begin{remark}
It is worth mentioning that all the estimates obtained in Section \ref{loc-re} are independent of the small parameters $\eta_1$, $\eta_2$, $\eta_3$ in \eqref{hh}.
\end{remark}

\subsection{Refined estimates}\label{opt}

For the given local minimizer $(\varphi_*, \sigma_*)$, we can apply Proposition \ref{prop-le2} and Theorem \ref{LSmain} to conclude that there exist constants
$\kappa_*\in (0,1/2)$, $\ell_* \in (0,1)$ and $C_*\geq 1$,
such that the following {\L}ojasiewicz--Simon inequality holds
\begin{align}
|\mathcal{F}(\varphi,\sigma) -\mathcal{F}(\varphi_*,\sigma_*)|^{1-\kappa_*}
& \leq C_*  \|-\Delta \varphi+ \varPsi'(\varphi)
- \chi(\sigma-\overline{\sigma})
+{\beta}\mathcal{N}(\varphi-\overline{\varphi})-\overline{ \varPsi'(\varphi)}\|\notag\\
& \quad +C_*  \|\sigma-\chi\varphi- \overline{\sigma -\chi\varphi}\| +C_* |\overline{\varphi}-c_0|^{1-\kappa_*},
\label{lojasb}
\end{align}
for any function $(\varphi,\sigma)\in H^2_N(\Omega)\times L^2(\Omega)$ satisfying $\overline{\sigma}=m_2$ and $\|\varphi-\varphi_{*}\|_{H^2} +\|\sigma-\sigma_*\|<\ell_*$.

Now for any given $\epsilon>0$, we take
\begin{equation}
		\omega=\min\left\lbrace 1,\,\epsilon,\,\frac{\lambda_*}{2},\,\ell_* ,\,\frac{\delta_*}{3C_S},\,\frac{L_9M_3}{5L_8}\right\rbrace,\label{t}
		\end{equation}
where
\begin{equation}
M_3=\frac{1}{3} \max\{1,6L_5\alpha\} T_2,\quad L_9=\frac{1}{2}\min\{1,\eta_*\}<1.
\label{e0}
\end{equation}
Next, for any given parameters
\begin{equation}
\eta_1,\eta_2,\eta_3 \in\left(0,\,\frac{\omega}{3}\right],
\label{eta123}
\end{equation}
if the corresponding solution satisfies $\|\varphi(t)-\varphi_{*}\|_{H^2(\Omega)} +\|\sigma(t)-\sigma_{*}\|\leq \omega $ for all $t\in [0,T_1]$, we set $\widehat{T}=T_1$. Otherwise,
we define
\begin{equation}
\widehat{T} =\inf\big\{t\in (0,T_1]\ \big|\ \|\varphi(t)-\varphi_{*}\|_{H^2(\Omega)} +\|\sigma(t)-\sigma_{*}\|>\omega\big\}.
\label{T-hat}
\end{equation}
By continuity of the solution, we have $\widehat{T}>0$.

Recalling \eqref{aac}, we see that the growth of higher-order norms of the local strong solution are controlled on the time interval $[0,T_2]$. Next, we proceed to show that the solution can stay in a suitable small neighborhood of $(\bm{0},\varphi_{*},\sigma_{*})$ (characterized by the constant $\omega$) on the time interval $[0,T_2]$ as well, if we choose $\bm{v}_0$ is sufficiently small and $(\varphi_0,\sigma_0)$ is a sufficiently small perturbation of the local minimizer $(\varphi_{*},\sigma_{*})$ (characterized by the parameters $\eta_1, \eta_2, \eta_3$).  This property will enable us to derive some refined estimates for the local strong solution on $[0,T_2]$, which are crucial in the construction of a global strong solution.
\begin{lemma} \label{claim}
There exists at least one triple of parameters $(\eta_1,\eta_2,\eta_3)$ satisfying \eqref{eta123} such that $\widehat{T} > T_2$.
\end{lemma}
\begin{proof}
We complete the proof by a contradiction argument. Assuming that for all $(\eta_1,\eta_2,\eta_3)$ satisfying  \eqref{eta123}, it always holds $\widehat{T} \in (0, T_2]$.

The energy inequality \eqref{BEL5} implies that problem \eqref{f3.c}--\eqref{ini0} has a strict Lyapunov functional. Let $\widetilde{\varphi}(t)= \varphi(t)-\overline{\varphi}(t)+c_0$. From \eqref{diffE} and \eqref{r}, we have
\begin{align}
\mathcal{E}(t) &= \frac12\|\bm{v}(t)\|^2+ \mathcal{F}(\widetilde{\varphi}(t),\sigma(t))   + \int_{\Omega} \big( \varPsi(\varphi(t))-\varPsi(\widetilde{\varphi}(t))
-\chi\sigma(t)(\varphi(t)-\widetilde{\varphi}(t)) \big) \mathrm{d}x\notag\\
&\geq \frac12\|\bm{v}(t)\|^2+ \mathcal{F}(\widetilde{\varphi}(t),\sigma(t)) -|\Omega|\max_{r\in[-1+\frac{\delta_*}{3}, 1-\frac{\delta_*}{3}]}|\varPsi'(r)| |\overline{\varphi}(t)-c_0| \notag\\ &\quad -|\chi||\Omega||m_2||\overline{\varphi}(t)-c_0|
\notag\\
&= \frac12\|\bm{v}(t)\|^2+ \mathcal{F}(\widetilde{\varphi}(t),\sigma(t)) -L_{10}e^{-\alpha t}|\overline{\varphi_0}-c_0|,\quad \forall\, t\in [0,T_1],
\label{diffE2}
\end{align}
where $L_{10}>0$ depends on $K_2$, $\chi$, $c_0$, $m_2$ and $\Omega$. On the other hand, applying Korn's inequality and (H1), we can rewrite \eqref{BEL5} as follows
\begin{align}
&-\frac{\mathrm{d}}{\mathrm{d}t} \left( \mathcal{E}(t)+(2  L_2+L_{10}) e^{-\alpha t}|\varphi_0-c_0|\right)\notag\\
&\quad\ge\frac{\nu_*}{2}\|\nabla \boldsymbol{v}(t)\|^{2}+\frac{1}{2}\|\nabla \mu(t)\|^{2}+\frac{1}{2}\|\nabla(\sigma(t)-\chi\varphi(t))\|^2 +(L_2+L_{10}) \alpha e^{-\alpha t}|\overline{\varphi_0}-c_0| \notag\\
&\quad\ge L_9\big(\|\nabla \boldsymbol{v}(t)\|^{2}+\|\nabla \mu(t)\|^{2}+\|\nabla(\sigma(t)-\chi\varphi(t))\|^2\big) +L_{11}\alpha e^{-\alpha t}|\overline{\varphi_0}-c_0|,
\label{v1}
\end{align}
for almost all $t\in (0,T_1)$ with $L_{11}= L_2+L_{10}$.

Recalling that $(\varphi_*, \sigma_*)$ is a local minimizer and observing that
\begin{align}
& \|\widetilde{\varphi}(t)-\varphi_*\|_{H^2}+ \|\sigma(t)-\sigma_*\| \notag\\
&\quad \leq \|\varphi(t)-\varphi_*\|_{H^2}+ |\Omega|^\frac12e^{-\alpha t}|\overline{\varphi_0}-c_0|+ \|\sigma(t)-\sigma_*\| \notag\\
&\quad \leq \frac{4}{3}\omega<\lambda_*,\quad \forall\, t\in [0,\widehat{T}],
\notag
\end{align}
if there exists some $\widetilde{t}_0\in[0,\widehat{T}]$ satisfying $$\mathcal{E}(\widetilde{t}_0)+(2  L_2+L_{10}) e^{-\alpha \widetilde{t}_0}|\overline{\varphi_0}-c_0|=\mathcal{F}(\varphi_*,\sigma_*),$$ then it follows from \eqref{diffE2}--\eqref{v1} that
$\mathcal{E}(t)+(2  L_2+L_{10}) e^{-\alpha t}|\overline{\varphi_0}-c_0|=\mathcal{F}(\varphi_*,\sigma_*)$ for all $t\geq \widetilde{t}_0$. As a consequence, $\|\nabla\bm{v}(t)\|
=\|\nabla\mu(t)\|=\|\nabla(\sigma(t)-\chi\varphi(t))\|
=|\overline{\varphi}(t)-c_0|=0$  for $t\ge \widetilde{t}_0$. This implies that  $(\bm{v}(t),\varphi(t),\sigma(t))$ becomes independent of time for all $t\ge \widetilde{t}_0$ (cf. \eqref{e3} below). Hence, we have obtained a (unique) global strong solution and Theorem \ref{3main} is proved.

Below we assume that
$$
\widehat{\mathcal{E}}(t):=\mathcal{E}(t)+(2  L_2+L_{10})  e^{-\alpha t}|\overline{\varphi_0}-c_0| >\mathcal{F}(\varphi_*,\sigma_*),\quad \forall  \, t\in [0,\widehat{T}].
$$
Using the Poincar\'{e}--Wirtinger inequality \eqref{poincare}, for the solution $(\varphi, \sigma)$ we have
\begin{align}
&\|-\Delta \varphi(t)+\varPsi'(\varphi(t))
- \chi(\sigma(t)-\overline{\sigma}(t))
+{\beta}\mathcal{N}(\varphi(t)-\overline{\varphi}(t))-\overline{ \varPsi'(\varphi(t))}\| \notag\\
&\qquad  +  \|\sigma(t)-\chi\varphi(t)- \overline{\sigma(t) -\chi\varphi(t)}\| \notag\\
&\quad\le C\big(\|\nabla \mu(t)\|
+ \|\nabla (\sigma(t)-\chi\varphi(t))\|\big),\quad \forall\, t\in [0,\widehat{T}].\label{e2}
\end{align}
In view of \eqref{t} and \eqref{T-hat}, we can apply the
{\L}ojasiewicz--Simon inequality \eqref{lojasb} to the solution $(\varphi(t), \sigma(t))$ on $[0,\widehat{T}]$ and deduce from \eqref{v1}, \eqref{e2} and Poincar\'{e}'s inequality that
\begin{align}
&-\frac{\mathrm{d}}{\mathrm{d} t}\left(\widehat{\mathcal{E}}(t)
 -\mathcal{F}(\varphi_*,\sigma_*) \right)^{\kappa_*}\notag\\
&\quad=-\kappa_* \left(\widehat{\mathcal{E}}(t) -\mathcal{F}(\varphi_*,\sigma_*) \right)^{\kappa_*-1} \frac{\mathrm{d}}{\mathrm{d}t}   \widehat{\mathcal{E}}(t) \notag\\
& \quad\ge \frac{\kappa_*\left[L_9\big(\|\nabla \boldsymbol{v}(t)\|^{2}+\|\nabla \mu(t)\|^{2}+\|\nabla(\sigma(t)-\chi\varphi(t))\|^2\big) +L_{11} \alpha e^{-\alpha t}|\overline{\varphi_0}-c_0|\right]} {\left[\displaystyle{\frac{1}{2}}\|\bm{v}(t)\|^2 +\mathcal{F}(\varphi(t),\sigma(t)) + (2  L_2+L_{10}) e^{-\alpha t}|\overline{\varphi_0}-c_0| - \mathcal{F}(\varphi_*,\sigma_*)\right]^{1-\kappa_*}} \notag\\
& \quad\ge  \frac{ C\left(\|\nabla \boldsymbol{v}(t)\|^{2}+\|\nabla \mu(t)\|^{2}+\|\nabla(\sigma(t)-\chi\varphi(t))\|^2 +\alpha e^{-\alpha t}|\overline{\varphi_0}-c_0|\right)}{\|\boldsymbol{v}(t)\|^{2(1-\kappa_*)} +\|\nabla \mu(t)\|+\|\nabla(\sigma(t)-\chi\varphi(t))\|
+\big( e^{-\alpha t}|\overline{\varphi_0}-c_0|\big)^{1-\kappa_*} } \notag\\
&\quad \ge L_{12}\min\{1,\alpha\} \left(\|\nabla \boldsymbol{v}(t)\|+\|\nabla \mu(t)\|+\|\nabla(\sigma(t)-\chi\varphi(t))\|+\sqrt{ e^{-\alpha t}|\overline{\varphi_0}-c_0|} \right)
\label{s}
\end{align}
for almost all $t\in (0,\widehat{T})$, where the positive constant $L_{12}$ depends on $\kappa_*$, $C_*$, $L_2$, $L_4$, $L_9$, $L_{10}$, $\Omega$ and coefficients of the system. By comparison and H\"{o}lder's inequality, we obtain the following estimate on time derivatives:
\begin{align}
& \|\partial_t \varphi- \overline{\partial_t \varphi}\|_{V_0'} + \|\partial_{t} \sigma\|_{V_0'} \notag\\
& \quad \leq \|\varphi\bm{v} \| +\|\nabla  \mu\| + \|(\sigma-\overline{\sigma})\bm{v}\| +\|\nabla(\sigma-\chi\varphi)\|\notag\\
&\quad \leq \|\varphi\|_{L^3} \|\bm{v} \|_{\bm{L}^6}+\|\nabla  \mu\| +\|\sigma-\overline{\sigma}\|_{L^3} \|\bm{v}\|_{\bm{L}^6} +\|\nabla(\sigma-\chi\varphi)\|\notag\\
&\quad
\leq C\|\varphi\|_{H^1}\|\nabla \bm{v}\| +\|\nabla \mu\|+ C \big(\|\nabla(\sigma-\chi\varphi)\|+\|\chi\nabla\varphi\|\big)\|\nabla \bm{v}\|  +\|\nabla(\sigma-\chi\varphi)\|\notag\\
&\quad \le L_{13}\big(\|\nabla\bm{v}\|^2 +\|\nabla(\sigma-\chi\varphi)\|^2\big) +L_{14}\big(\|\nabla\bm{v}\|+\|\nabla \mu\| +\|\nabla(\sigma-\chi\varphi)\|\big),\label{e3}
\end{align}
where $L_{13}>0$ depends only on  $\Omega$ and $L_{14}\geq 1$ depends on the coefficients of the system, the bound $L_4$ for lower-order norms and $\Omega$.

Hence, we infer from \eqref{conver1b}, \eqref{aver-phib}, \eqref{v1}, \eqref{s} and \eqref{e3} that
\begin{align}
&-\frac{\mathrm{d}}{\mathrm{d} t}\left(\widehat{\mathcal{E}}(t)
-\mathcal{F}(\varphi_*,\sigma_*) \right)^{\kappa_*} -  \frac{\mathrm{d}}{\mathrm{d}t} \left( \widehat{\mathcal{E}}(t)
 -\mathcal{F}(\varphi_*,\sigma_*)\right)\notag\\
&\quad \geq C\left(\|\partial_t \varphi- \overline{\partial_t \varphi}\|_{V_0'} + \|\partial_{t} \sigma\|_{V_0'}
+ |\overline{\partial_t \varphi}|\right)\notag\\
&\quad \geq L_{15} \big(\|\partial_t \varphi\|_{(H^1)'}+ \|\partial_{t} \sigma\|_{V_0'}\big),
\label{s1}
\end{align}
where $L_{15}>0$ depends on $L_2$, $L_4$, $L_{9}$, $L_{10}$, $L_{12}$, $L_{13}$, $L_{14}$, $\alpha$ and $\Omega$. In particular, $L_{15}$ does not depend on $M_0$.
By a similar argument for \eqref{w} and using \eqref{s1}, we obtain
\begin{align}
&\int_{0}^{\widehat{T}}\|\partial_t \varphi(t)\|_{(H^1)'}+\left\|\partial_{t} \sigma(t)\right\|_{V_0^{\prime}}\, \mathrm{d}t\notag\\
&\quad\le \frac{1}{L_{15}}\big(\mathcal{E}(0)
+ (2  L_2 +L_{10})  |\overline{\varphi_0}-c_0|-\mathcal{F}(\varphi_*,\sigma_*) \big)^{\kappa_*} \notag\\
&\qquad +
\frac{1}{L_{15}} \big( \mathcal{E}(0)
+ (2  L_2 + L_{10}) |\overline{\varphi_0}-c_0|-\mathcal{F}(\varphi_*,\sigma_*)\big)
\notag\\
&\quad\le C\big(\|\bm{v}_0\|^2
+   \|\varphi_0-\varphi_{*}\|_{H^1} +\|\sigma_0-\sigma_{*}\| \big)^{\kappa_*} \notag\\
&\quad\le L_{16}\big(\|\bm{v}_0\|^{2\kappa_*}
+\|\varphi_0-\varphi_{*}\|_{H^1}^{\kappa_*} +\|\sigma_0-\sigma_{*}\|^{\kappa_*}\big),\notag
\end{align}
where $L_{16}>0$ depends on $L_2$, $L_{10}$, $L_{15}$, $\|\varphi_*\|_{H^1}$, $\|\sigma_*\|$, $K_2$, $\delta_*$, the coefficients of the system and $\Omega$. As a consequence, using the interpolation inequality and \eqref{421}, we get
\begin{align}
&\|\varphi(\widehat{T})-\varphi_{*}\|_{H^2}\notag\\
&\quad \le\|\varphi_0-\varphi_{*}\|_{H^2}
+\|\varphi(\widehat{T})-\varphi_0\|_{H^2}\notag\\
&\quad\le\|\varphi_0-\varphi_{*}\|_{H^2}
+C\|\varphi(\widehat{T})-\varphi_0\|_{H^3}^\frac{3}{4}
\|\varphi(\widehat{T})-\varphi_0\|_{(H^1)'}^\frac{1}{4}\notag\\
&\quad\le\|\varphi_0-\varphi_{*}\|_{H^2}+C(M_0+M_2)^\frac{3}{4} \left(\int_{0}^{\widehat{T}}
\|\partial_t\varphi(t)\|_{(H^1)'}\,\mathrm{d}t\right)^\frac{1}{4}\notag\\
&\quad \le\|\varphi_0-\varphi_{*}\|_{H^2}
 +L_{17}(M_0+M_2)^\frac{3}{4}\left(\|\bm{v}_0\|^{\frac{\kappa_*}{2}}
+\|\varphi_0-\varphi_{*}\|_{H^1}^\frac{\kappa_*}{4} +\|\sigma_0-\sigma_{*}\|^\frac{\kappa_*}{4}\right),
\label{digg}
\end{align}
where the constant $L_{17}>0$ depend on $\Omega$ and $L_{16}$. In a similar manner, we can deduce that
\begin{align}
&\|\sigma(\widehat{T})-\sigma_{*}\|\notag\\
&\quad \le\|\sigma_0-\sigma_{*}\|
+L_{18}(M_0+M_2)^\frac{1}{2}\left(\|\bm{v}_0\|^{\kappa_*}
+\|\varphi_0-\varphi_{*}\|_{H^1}^\frac{\kappa_*}{2} +\|\sigma_0-\sigma_{*}\|^\frac{\kappa_*}{2}\right),
\label{digg1}
\end{align}
where the constant $L_{18}>0$ depends on $\Omega$ and $L_{16}$.

Taking the above estimates into account, we take
\begin{align}
&\eta_1=\min\left\lbrace\frac{\omega}{10},\ \left(\frac{\omega}{10L_{17}(M_0+M_2)^\frac{3}{4}}\right)^\frac{2}{\kappa_*},\ \left(\frac{\omega}{10L_{18}(M_0+M_2)^\frac{1}{2}}\right)^\frac{1}{\kappa_*}\right\rbrace, \label{eta1} \\
&\eta_3=\min\left\lbrace\frac{\omega}{10},\ \left(\frac{\omega}{10L_{17}(M_0+M_2)^\frac{3}{4}}\right)^\frac{4}{\kappa_*},\ \left(\frac{\omega}{10L_{18}(M_0+M_2)^\frac{1}{2}}\right)^\frac{2}{\kappa_*}\right\rbrace, \label{eta3} \\
&\eta_2= \min\left\lbrace \eta_3,\ \frac{\sqrt{|\Omega|}\omega}{4(L_2+L_5+L_{10})}\right\rbrace.
\label{eta2}
\end{align}
Then from \eqref{digg} and \eqref{digg1}, it follows that
$$\|\varphi(\widehat{T})-\varphi_{*}\|_{H^2}
\leq \frac{2\omega}{5},\qquad\|\sigma(\widehat{T})-\sigma_{*}\|
\leq\frac{2\omega}{5},
$$
which lead to a contradiction with the definition of $\widehat{T}$. The proof is complete.
\end{proof}

\textbf{Refined estimates on $[0,T_2]$}. Thanks to Lemma \ref{claim}, for any $\epsilon>0$, there exists $(\eta_1,\eta_2,\eta_3)$ sufficiently small (chosen as above) such that
\begin{equation}
\|\varphi(t)-\varphi_{*}\|_{H^2} +\|\sigma(t)-\sigma_{*}\|
\le\omega,\quad \forall\,t\in[0,T_2].
\label{u}
\end{equation}
This observation enables us to derive some refined estimates for the local strong solution on the time interval $[0,T_2]$. From \eqref{m1}, \eqref{hh}, \eqref{t} and \eqref{u}, for any $t\in[0,T_2]$,  we can deduce that
\begin{align}
&\|\varphi(t)\|_{C(\overline{\Omega})} \le \|\varphi_{*}\|_{C(\overline{\Omega})} +C_S\|\varphi(t)-\varphi_{*}\|_{H^2}\le1-\frac{2\delta_*}{3},
\label{xa}
\\
&\|\varphi(t)-\varphi_0\|_{H^2}
\le\|\varphi(t)-\varphi_{*}\|_{H^2}
+\|\varphi_{*}-\varphi_0\|_{H^2} <\frac{L_9M_3}{4L_8},
\label{x}\\
&\|\sigma(t)-\sigma_0\|
\le\|\sigma(t)-\sigma_{*}\|
+\|\sigma_{*}-\sigma_0\|< \frac{L_9M_3}{4L_8}.
\label{x1}
\end{align}
Next, it follows from \eqref{hh}, \eqref{m3-b}, \eqref{w}, \eqref{t}, \eqref{v1}, \eqref{eta1}--\eqref{eta2}  and $\eqref{x}$ that
\begin{align}
\int_{0}^{T_2}\widetilde{\Lambda}(t)\,\mathrm{d}t
&\le \int_{0}^{T_2}\left(\frac34\left\|\nabla \boldsymbol{v}(t)\right\|^{2}+\frac{3a_1}{4}\left\|\nabla \mu(t)\right\|^{2}+\frac34\left\| \nabla(\sigma(t)-\chi\varphi(t))\right\|^{2} \right) \mathrm{d}t\notag\\
&\quad + \int_{0}^{T_2} 3L_5\alpha e^{-\alpha t}|\overline{\varphi_0}-c_0|  \mathrm{d}t\notag\\
&\leq\int_{0}^{T_2}\left(\|\nabla \boldsymbol{v}(t)\|^{2} +\|\nabla \mu(t)\|^{2} +\|\nabla(\sigma(t)-\chi\varphi(t))\|^2+\frac{L_{11}}{L_9} \alpha e^{-\alpha t}|\overline{\varphi_0}-c_0|\right)\mathrm{d}t \notag\\
&\quad + 3L_5\big(1-e^{-\alpha T_2}\big)|\overline{\varphi_0}-c_0|\notag\\
&\le \frac{1}{L_9}\big(\mathcal{E}(0)+ (2  L_2 +L_{10}) |\overline{\varphi_0}-c_0| - \mathcal{E}(T_2) \big) + 3L_5|\overline{\varphi_0}-c_0|
\notag\\
&\le \frac{1}{L_9}\left(\frac{1}{2}\eta_1^2 +\frac{2L_2+L_{10}}{\sqrt{|\Omega|}}\eta_2+\big|\mathcal{F}(\varphi(0),\sigma(0)) -\mathcal{F}(\varphi(T_2),\sigma(T_2))\big|\right) +\frac{3L_5}{\sqrt{|\Omega|}}\eta_2\notag\\
&< \frac{1}{L_9}\left(\frac12\omega^2 + \frac12 \omega +L_8\big(\|\varphi_0-\varphi(T_2)\|_{H^1} +\|\sigma_0-\sigma(T_2)\|\big)\right)+\omega \notag\\
&< \frac{1}{L_9}\left(\frac15 M_3L_9+L_8\frac{M_3L_9}{2L_8}\right)+\frac15 M_3\notag\\
&< M_3.
\label{mm}
\end{align}

\subsection{Proof of Theorem \ref{3main}}\label{proof2.1}
 We are ready to prove Theorem \ref{3main} by an induction argument (cf. \cite{GGW}). For any given $\epsilon>0$, we take the parameters $M_0$, $\eta_1$, $\eta_2$, $\eta_3$ to be the same as in \eqref{ddM} and  \eqref{eta1}--\eqref{eta2}. \medskip

 \textbf{Step 1.}
  For the unique local strong solution $(\bm{v},\varphi,\mu,\sigma)$ on $[0,\,T_2]$, we deduce from \eqref{mm} that
\begin{equation*}
 \int_{\frac{1}{2}T_2}^{T_2}\widetilde{\Lambda}(t)\,\mathrm{d}t \le\int_{0}^{T_2}\widetilde{\Lambda}(t)\,\mathrm{d}t\le M_3.
 \end{equation*}
 Recalling the definition of $\widetilde{\Lambda}$ (see \eqref{LAMB}) and $M_3$ (see \eqref{e0}), we can find some  $t^*\in\big[T_2/2,\,T_2\big]$ such that
 \begin{equation}
 \widetilde{\Lambda}(t^*)< \max\{1,6L_5\alpha\}=K_4.
 \label{qq}
 \end{equation}
 Then it follows from
  \eqref{421a}, \eqref{4211a} and \eqref{421} that
  \begin{align}
 \|\bm{v}(t^*)\|_{\bm{H}^1}
   \leq  2\sqrt{K_4},\qquad
   \|\sigma(t^*)\|_{H^1}   \le 2\sqrt{K_4}+(1+|\chi|)\sqrt{L_4},
 \label{ddb}
 \end{align}
and
 \begin{align}
 \|\varphi(t^*)\|_{H^3} \leq  L_7\left[\big(2a_1^{-\frac12} +2\big) \sqrt{K_4}
+\big(K_1+2+|\chi|\big)\sqrt{L_4} \right],
 \label{dda}
 \end{align}
 where the constant $L_7\geq 1$ is given in \eqref{o}.
 According to \eqref{ddM}, it holds
 \begin{align}
 \|\bm{v}(t^*)\|_{\bm{H}^1}\le M_0,\qquad\|\varphi(t^*)\|_{H^3}\le M_0,\qquad\|\sigma(t^*)\|_{H^1}\le M_0.
 \label{ddc}
 \end{align}
 Thus $(\bm{v}(t^*),\varphi(t^*),\sigma(t^*))$ satisfies the same upper bounds for higher-order norms of the initial datum $(\bm{v}_0,\varphi_0,\sigma_0)$ (cf. \eqref{l}). Besides, thanks to \eqref{BEL5}  and \eqref{xa}, we recover the estimate for the ``initial energy'' at $t^*$ and the separation property (cf. \eqref{tt})
 \begin{align}
 &\mathcal{E}(t^*)+L_2 e^{-\alpha t^*}|\overline{\varphi_0}-c_0|\le \mathcal{E}(0)+L_2|\overline{\varphi_0}-c_0|,\label{ddd} \\
 & \|\varphi(t^*)\|_{C(\overline{\Omega})} \le1-\frac{2\delta_*}{3}.
 \label{dde}
 \end{align}

\textbf{Step 2.} Taking $(\bm{v}(t^*),\,\varphi(t^*),\sigma(t^*))$ as the new initial datum, with the bounds \eqref{ddc}--\eqref{dde}, we are able to obtain a local strong solution to problem \eqref{f3.c}--\eqref{ini0} on $[t^*,t^*+T_2]$ as in Section \ref{opt}. By the uniqueness (cf. Theorem \ref{ls}), we have thus proved the existence of a unique local strong solution to problem \eqref{f3.c}--\eqref{ini0} on the extended time interval $[0,t^*+T_2]$, which satisfies the estimates \eqref{q}, \eqref{aac} and  \eqref{421a}--\eqref{421} on $[0,t^*+T_2]$. Moreover, on the subinterval $[0,3T_2/2]\subset [0,t^*+T_2]$, we can derive \eqref{u} and then obtain the refined estimates \eqref{xa}--\eqref{x1} as well as
\be
\int_{0}^{\frac{3}{2}T_2}\widetilde{\Lambda}(t)\,\mathrm{d}t\le M_3.
\label{intbd}
\ee
 By \eqref{intbd} and applying the  argument like  in Step 1, we find some $t^{**}\in \big[T_2,\, 3T_2/2\big]$ such that $(\bm{v}(t^{**}),\,\varphi(t^{**}),\sigma(t^{**}))$
satisfies the same estimates as \eqref{ddc}--\eqref{dde}.

Hence, we can repeat the above procedure and obtain the local strong solution on $[t^{**},t^{**}+T_2]$ subject to the initial data   $(\bm{v}(t^{**}),\,\varphi(t^{**}),\sigma(t^{**}))$. This yields a unique local strong solution on $[0,2T_2]$ with the same bounds as before. By induction, we can finally construct a unique global strong solution  that satisfies the uniform-in-time estimates \eqref{q}, \eqref{aac},  \eqref{421a}--\eqref{421} and  \eqref{u}--\eqref{x1} on $[0,+\infty)$. Moreover, it holds
\be
\int_{0}^{+\infty}\widetilde{\Lambda}(t)\,\mathrm{d}t\le M_3.\label{inf}
\ee
The proof of Theorem \ref{3main} is complete.
$\hfill\square$

%%%%%%%%%%%%%%%%%%%%%%%%%%%%%%%%%%%%%%%%%%%%%
 \section{Long-time Behavior}\label{lbg}
 \setcounter{equation}{0}

In this section, we study long-time behavior of the global strong solution to problem \eqref{f3.c}--\eqref{ini0}.
Again we shall focus on the proof for the case $\alpha>0$, while the case $\alpha=0$ is an easy consequence with minor modification.

\subsection{Convergence to equilibrium}
Let $(\bm{v}, \varphi, \mu, \sigma)$ be the unique global strong solution obtained in Theorem \ref{3main}. Then we define its $\omega$-limit set by
\begin{align}
\omega(\bm{v},\varphi, \sigma)
&=\big\{(\bm{z}_1,z_2,z_3)  \in \bm{L}^2_{0,\mathrm{div}}(\Omega) \times H_N^2(\Omega) \times L^2(\Omega):\quad \exists\,\left\{t_{n}\right\} \nearrow +\infty\notag\\
& \quad \text { s.t. } \left(\bm{v}(t_{n}),\varphi\left(t_{n}\right),\sigma \left(t_{n}\right)\right) \rightarrow (\bm{z}_1,z_2,z_3)\text{ strongly in }  \bm{L}^2(\Omega)\times H^{2}(\Omega)\times L^{2}(\Omega) \big\}. \notag
\end{align}
Since \eqref{aver-phib} holds on $[0,+\infty)$, we find that
\begin{align}
\lim_{t\to +\infty}\overline{\varphi}(t)=c_0,\label{con-avep}
\end{align}
exponentially fast. Next, it has been shown in Section \ref{proof2.1} that the quantity $\widetilde{\Lambda}(t)$ is uniformly bounded and integrable on the infinite interval $[0,+\infty)$ (cf. \eqref{aac}, \eqref{inf}). Hence, we can infer from the inequality \eqref{aab} that
$$
\lim_{t\to +\infty}\widetilde{\Lambda}(t)=0.
$$
This together with \eqref{m3b-b} and \eqref{LAMB} further implies
\begin{align}
\lim_{t\to +\infty}\big(\|\nabla\bm{v}(t)\|+\|\nabla\mu(t)\| +\|\nabla(\sigma(t)-\chi\varphi(t))\|\big)=0.
\label{424}
\end{align}
By Poincar\'{e}'s inequality, we easily get
\begin{equation}
\lim_{t\to +\infty}\|\bm{v}(t)\|_{\bm{H}^1}=0. \label{con-v}
\end{equation}
Similarly, from the Poincar\'{e}--Wirtinger inequality, \eqref{con-avep}, \eqref{424} and \eqref{aver-sigb}, we  obtain
\begin{align}
\lim_{t\to +\infty} \|(\sigma(t)-\chi\varphi(t))-(\overline{\sigma_0}-\chi c_0)\|=0. \label{con-aves}
\end{align}

Taking advantage of the dissipative energy inequality \eqref{BEL5}, the continuous dependence estimate \eqref{uniA1} (in lower-order norms), the uniform-in-time higher-order estimates \eqref{421a}--\eqref{421} on $[0,+\infty)$ and \eqref{424}, we can follow a standard argument (see for instance, \cite{A2007,GGW,JWZ,ZWH}) to show that the $\omega$-limit set $\omega(\bm{v},\varphi,\sigma)$ is a non-empty, compact and connected subset of $\bm{L}^2_{0,\mathrm{div}}(\Omega) \times H_N^2(\Omega) \times L^2(\Omega)$, which is given by
$$
\omega(\bm{v},\varphi,\sigma)\subset
\big\{(\bm{0}, \varphi_\mathrm{s}, \sigma_\mathrm{s})\,:\,
(\varphi_\mathrm{s}, \sigma_\mathrm{s})\in \big(H^3(\Omega)\cap H_N^2(\Omega)\big) \times H_N^2(\Omega)\ \text{satisfying}\ \eqref{s5bchv}-\eqref{s5dchv}
\big\}
$$
with $m_1=c_0$ and $m_2=\overline{\sigma_0}$.
For the Lyapunov functional $\widehat{\mathcal{E}}(t)=\mathcal{E}(t)+(2  L_2+L_{10})  e^{-\alpha t}|\overline{\varphi_0}-c_0|$ (recall \eqref{low-bd1b}, \eqref{v1}), we have
$$
\lim_{t\to+\infty} \widehat{\mathcal{E}}(t)= E_\infty,
$$
for some $E_\infty \in \mathbb{R}$. Thus, for any
$(\bm{0},\varphi_{\mathrm{s}}, \sigma_\mathrm{s})\in \omega(\bm{v},\varphi,\sigma)$, it holds $\mathcal{F}( \varphi_{\mathrm{s}}, \sigma_\mathrm{s} )=E_\infty$. Moreover, it follows from \eqref{glo-sep} that
\begin{align}
|\varphi_{\mathrm{s}}(x)|\leq 1-\delta,\quad \forall\, x\in \overline{\Omega}.
\label{sta-sep1}
\end{align}
The above observation yields the existence of some  $(\bm{0},\varphi_{\infty}, \sigma_\infty)\in \omega(\bm{v},\varphi,\sigma)$ and an unbounded increasing sequence $\{t_{n}\}$ such that
\begin{equation}
\lim_{t_n\to +\infty}\|\varphi(t_n)-\varphi_{\infty}\|_{H^2}=0,\quad \lim_{t_n\to +\infty}\|\sigma(t_n)-\sigma_\infty\|=0.
\label{con-seq}
\end{equation}
Combining the above result with \eqref{con-aves}, we obtain $\sigma_\infty-\chi\varphi_\infty=\overline{\sigma_0}-\chi c_0$, i.e., a constant.

Analysis similar to that in the proof of Lemma \ref{claim} implies, if there exists some $\widetilde{t}_1>0$ such that $\widehat{\mathcal{E}}(\widetilde{t}_1)= E_\infty$, then the global strong solution simply becomes a stationary one after $\widetilde{t}_1$. This leads to a trivial case and the conclusion of Theorem \ref{3main1} follows (including the convergence rate).

Therefore, in the remaining part of the proof, we assume that
$$\widehat{\mathcal{E}}(t)> E_\infty,\quad \forall\, t\geq 0.
$$
 Thanks to the assumption (H3) and the strict separation properties \eqref{glo-sep}, \eqref{sta-sep1}, for the global strong solution $\varphi(t)$ the potential $\varPsi$ is an analytic function confined on $(-1+\delta/2,1-\delta/2)$ along time evolution. This enables us to improve the sequential convergence \eqref{con-seq} by applying the {\L}ojasiewicz--Simon approach and show that $\omega(\bm{v},\varphi,\sigma)=\{(\bm{0},\varphi_\infty, \sigma_\infty)\}$. For results on some related systems, we refer the reader to \cite{A2009,A2007,GG2010,JWZ,RH99,ZWH}.

For every $(\bm{0},\varphi_{\mathrm{s}}, \sigma_\mathrm{s})\in \omega(\bm{v},\varphi,\sigma)$, we can apply Lemma \ref{LS} and conclude that there exist   $\kappa_\mathrm{s}\in(0,1/2)$, $\ell_\mathrm{s}\in (0,1)$ and $C_\mathrm{s}\geq 1$ such that \eqref{lojas} holds
for any
$$
(\varphi,\sigma)\in \mathcal{B}_{\ell_\mathrm{s}}(\varphi_\mathrm{s},\sigma_\mathrm{s})  =\big\{ (\varphi,\sigma)\in H^2_N(\Omega)\times L^2(\Omega)\ \big|\ \overline{\sigma}=\overline{\sigma_0},\ \|\varphi-\varphi_\infty\|_{H^2} +\|\sigma-\sigma_\infty\|<\ell_\mathrm{s}\big\}.
$$
Following the argument in \cite{RH99} and thanks to the compactness of $\omega(\bm{v},\varphi,\sigma)$, we can find a finite number of elements
$\big(\bm{0},\varphi_{\mathrm{s}}^{(i)}, \sigma_\mathrm{s}^{(i)}\big)\in \omega(\bm{v},\varphi,\sigma)$, $i=1,...,n$, such that
$$
\omega(\bm{v},\varphi,\sigma)\subset \{\bm{0}\}\times \bigcup_{i=1}^n \mathcal{B}_{\ell^{(i)}_\mathrm{s}}\big(\varphi^{(i)}_\mathrm{s}, \sigma^{(i)}_\mathrm{s}\big).
$$
Here, we denote by $\ell^{(i)}_\mathrm{s}$, $\kappa^{(i)}_\mathrm{s}$ and $C^{(i)}_\mathrm{s}$ the  constants in Lemma \ref{LS} corresponding to $(\varphi^{(i)}_\mathrm{s}, \sigma^{(i)}_\mathrm{s}\big)$.
From the definition of $\omega(\bm{v},\varphi,\sigma)$, there exists $\widetilde{t}_2>0$ sufficiently large such that
$$
(\varphi(t), \sigma(t))\in \bigcup_{i=1}^n \mathcal{B}_{\ell^{(i)}_\mathrm{s}}\big(\varphi^{(i)}_\mathrm{s}, \sigma^{(i)}_\mathrm{s}\big),\quad \forall\, t\geq \widetilde{t}_2.
$$
Thus, taking $\ell_\infty=\min_{i}\ell^{(i)}_\mathrm{s}$, $\kappa_\infty= \min_{i}\kappa^{(i)}_\mathrm{s}$ and $C_\infty=\max_i C^{(i)}_\mathrm{s}$, we have
\begin{align}
&|\mathcal{F}(\varphi(t),\sigma(t)) -E_\infty|^{1-\kappa_\infty}\notag\\
&\quad  \leq C_\infty  \|-\Delta \varphi(t)+ \varPsi'(\varphi(t))
- \chi(\sigma(t)-\overline{\sigma_0})
+{\beta}\mathcal{N}(\varphi(t)-\overline{\varphi}(t))-\overline{ \varPsi'(\varphi(t))}\|\notag\\
&\qquad +C_\infty  \|\sigma(t)-\chi\varphi(t)- \overline{\sigma(t) -\chi\varphi(t)}\| +C_\infty |\overline{\varphi(t)}-c_0|^{1-\kappa_\infty},\quad \forall\, t\geq  \widetilde{t}_2.
\label{lojasinf}
\end{align}
Proceeding analogously to the proof of \eqref{e2} and \eqref{s}, we can apply \eqref{lojasinf} to derive the following inequality for all $t> \widetilde{t}_2$:
\begin{align}
&-\frac{\mathrm{d}}{\mathrm{d} t}\left(\widehat{\mathcal{E}}(t)
 -E_\infty \right)^{\kappa_\infty}\notag\\
& \quad \ge C\min\{1,\alpha\} \left(\|\nabla \boldsymbol{v}(t)\|+\|\nabla \mu(t)\|+\|\nabla(\sigma(t)-\chi\varphi(t))\|+\sqrt{ e^{-\alpha t}|\overline{\varphi_0}-c_0|} \right).
\label{s1a}
\end{align}
On the other hand, using the higher-order estimates \eqref{421a}--\eqref{421} on $[0,+\infty)$, we can slightly modify \eqref{e3} to conclude that
\begin{align}
& \|\partial_t \varphi- \overline{\partial_t \varphi}\|_{V_0'} + \|\partial_{t} \sigma\|_{V_0'}   \leq  C\big(\|\nabla\bm{v}\|+\|\nabla \mu\| +\|\nabla(\sigma-\chi\varphi)\|\big).\label{e3a}
\end{align}
Hence, integrating \eqref{s1a} on $[\widetilde{t}_2,+\infty)$ and using \eqref{conver1b}, \eqref{aver-phib}, we obtain
\begin{equation*}
\int_{\widetilde{t}_2}^{+\infty} \big(\|\partial_t \varphi(t)\|_{(H^1)'}+\|\partial_t \sigma(t)\|_{V_0'}\big) \mathrm{d}t<+\infty,
\end{equation*}
which together with \eqref{con-seq} implies the convergence
\begin{equation}
\lim_{t\to +\infty}\|\varphi(t)-\varphi_\infty\|_{H^2}=0, \quad\lim_{t\to +\infty}\|\sigma(t)-\sigma_\infty\|=0.\label{s3}
\end{equation}
Recalling \eqref{con-v}, we can assert that $\omega(\bm{v}_0,\varphi_0,\sigma_0) =\big\{(\bm{0},\varphi_\infty,\sigma_\infty)\big\}$.

To draw the conclusion \eqref{convv}, we first observe from \eqref{424}, \eqref{con-aves} and \eqref{s3} that
\begin{align}
\|\sigma(t)-\sigma_\infty\|_{H^1}
& \leq C\|\nabla (\sigma(t)-\sigma_\infty)\| \notag\\
& \leq C\|\nabla (\sigma(t)-\chi\varphi(t))\| + C|\chi|\|\nabla(\varphi(t)-\varphi_\infty)\| +
C\|\nabla (\sigma_\infty-\chi\varphi_\infty)\|\notag\\
&\to 0, \quad \text{as}\ \ t\to+\infty.
\label{s4}
\end{align}
Besides, from equations \eqref{f4.d}, \eqref{5bchv}, the separation properties for $\varphi$, $\varphi_\infty$, and the convergence results \eqref{424}, \eqref{s3}, \eqref{s4}, we can conclude
\begin{align}
&\|\nabla\Delta(\varphi(t)-\varphi_{\infty})\| \notag \\
&\quad\le
\|\nabla \mu(t)\|+\|\nabla(\varPsi'(\varphi(t))
-\varPsi'(\varphi_{\infty}))\| +|\chi|\|\nabla(\sigma(t)-\sigma_{\infty})\| \notag\\
&\qquad +|\beta|\|\nabla \mathcal{N}\big[\varphi(t)-\overline{\varphi}(t) -(\varphi_{\infty}-\overline{\varphi_{\infty}})\big]\|\notag \\
&\quad \le\|\nabla \mu(t)\|+C\|\varphi(t)-\varphi_{\infty}\|_{H^1} +|\chi|\|\sigma(t)-\sigma_\infty\|_{H^1}\to 0,\quad \text{as}\ \ t\to +\infty.\label{s5}
\end{align}
Combining \eqref{424}, \eqref{s5} with the elliptic estimate, we finally get
\begin{equation}
\lim_{t\to +\infty}\|\varphi(t)-\varphi_\infty\|_{H^3}=0. \label{s6}
\end{equation}
In view of \eqref{con-v}, \eqref{s4} and \eqref{s6}, we have proved the first part of Theorem \ref{3main1}.

\subsection{Convergence rate}
We proceed to estimate the convergence rate.

The lower-order estimate can be obtained by the method in \cite{HJ2001}. Recalling \eqref{s3}, we infer that there exists $\widetilde{t}_3>0$ sufficiently large such that $(\varphi(t), \sigma(t))$ satisfies \eqref{lojasinf} with
$E_\infty=\mathcal{F}(\varphi_\infty,\sigma_\infty)$
 for all $t\geq \widetilde{t}_3$. This fact combined with \eqref{con-avep}, \eqref{con-v} and \eqref{s1a} yields that
\begin{align}
&-\frac{\mathrm{d}}{\mathrm{d} t}\left(\widehat{\mathcal{E}}(t)
 -\mathcal{F}(\varphi_\infty,\sigma_\infty) \right)^{\kappa_\infty}\notag\\
& \quad \ge C\min\{1,\alpha\} \left(\|\nabla \boldsymbol{v}(t)\|+\|\nabla \mu(t)\|+\|\nabla(\sigma(t)-\chi\varphi(t))\|+\big( e^{-\alpha t}|\overline{\varphi_0}-c_0|\big)^{1-\kappa_\infty} \right)\notag\\
&\quad \geq C\big(\|\bm{v}(t)\|^{2(1-\kappa_\infty)} +\big( e^{-\alpha t}|\overline{\varphi_0}-c_0|\big)^{1-\kappa_\infty} +|\mathcal{F}(\varphi(t),\sigma(t)) -\mathcal{F}(\varphi_\infty,\sigma_\infty)|^{1-\kappa_\infty}\big)\notag\\
&\quad \geq  C\left(\widehat{\mathcal{E}}(t)
 -\mathcal{F}(\varphi_\infty,\sigma_\infty) \right)^{1-\kappa_\infty}.
\label{s1b}
\end{align}
As a consequence, we obtain
\be
\frac{\mathrm{d}}{\mathrm{d} t}\left(\widehat{\mathcal{E}}(t)
 -\mathcal{F}(\varphi_\infty,\sigma_\infty) \right) +C \left(\widehat{\mathcal{E}}(t)
 -\mathcal{F}(\varphi_\infty,\sigma_\infty) \right)^{2-2\kappa_\infty}\le 0,\quad \forall\, t>\widetilde{t}_3.
 \label{e-decay}
\ee
This gives an estimate on the decay of the energy
\be
 0<\widehat{\mathcal{E}}(t)
 -\mathcal{F}(\varphi_\infty,\sigma_\infty)   \leq C(1+t)^{-\frac{1}{1-2 \kappa_\infty}},\quad \forall\, t\geq \widetilde{t}_3.
 \label{m7}
\ee
Integrating \eqref{s1b} from $t\geq \widetilde{t}_3$ to $+\infty$ and using \eqref{con-avep}, \eqref{e3a}, \eqref{m7}, we obtain
\be
\int_{t}^{+\infty}
\big(\|\nabla \boldsymbol{v}(t)\|+
\left\|\partial_{t} \varphi(t)\right\|_{\left(H^{1}\right)^{\prime}}+\left\|\partial_{t} \sigma(t)\right\|_{V_0^{\prime}}\big) \mathrm{d} t \leq C(1+t)^{-\frac{\kappa_\infty}{1-2 \kappa_\infty}}, \label{pv}
\ee
which yields that
\be
\left\|\varphi(t)-\varphi_\infty\right\|_{(H^{1})'}
+ \left\|\sigma(t)-\sigma_\infty\right\|_{V_0'} \leq C(1+t)^{-\frac{\kappa_\infty}{1-2 \kappa_\infty}},\quad \forall\, t\geq 0.
\label{ra}
\ee

Next, we derive higher-order estimate by an energy method (see e.g., \cite{ZWH}). Multiplying \eqref{f3.c} with $\bm{v}$ and integrating over $\Omega$, we get
\begin{align}
\frac{1}{2}\frac{\mathrm{d}}{\mathrm{d}t} \|\bm{v}\|^2
+ 2\int_\Omega \nu(\varphi)|D\bm{v}|^2\,\mathrm{d}x =\mathcal{W}_1(t),
\label{w1}
\end{align}
where
\begin{align}
\mathcal{W}_1(t)&=\big(\mu(t)\nabla \varphi(t),\bm{v}(t)\big) +\chi\big( \sigma(t)\nabla \varphi(t),\bm{v}(t)\big).
\label{w12}
\end{align}
Next, we observe that the difference  $(\varphi-\varphi_\infty,\sigma-\sigma_\infty)$ satisfies the following system
\begin{subequations}
	\begin{alignat}{3}
	&\partial_t (\varphi-\varphi_\infty)+\bm{v} \cdot \nabla \varphi=\Delta \mu-\alpha(\overline{\varphi}-c_0),\label{3f1.a} \\
	&\mu-\mu_\infty=-\Delta (\varphi-\varphi_\infty)
+\varPsi'(\varphi) -\varPsi'(\varphi_\infty)  -\chi (\sigma-\sigma_\infty)\notag\\
	& \qquad\qquad  \ \ +\beta\mathcal{N}(\varphi-\overline{\varphi}-\varphi_\infty+c_0),\label{3f4.d} \\
	&\partial_t (\sigma-\sigma_\infty)+\bm{v} \cdot \nabla \sigma= \Delta \big((\sigma-\sigma_\infty)-\chi(\varphi-\varphi_\infty)\big), \label{3f2.b}
	\end{alignat}
\end{subequations}
where
$\mu_\infty= \overline{\varPsi'(\varphi_\infty)}-\chi\overline{\sigma_0}$ is a constant. Then multiplying \eqref{3f1.a} with $\mu-\mu_\infty$, integrating over $\Omega$ and using \eqref{conver1}, we obtain
\begin{align}
& \frac{\mathrm{d}}{\mathrm{d} t}\left [\frac{1}{2}\left\|\nabla\left(\varphi-\varphi_{\infty}\right)\right\|^{2} +\frac{\beta}{2}\left\|\nabla\mathcal{N}\left(\varphi -\overline{\varphi}-\varphi_\infty+c_0\right)\right\|^{2} +\int_{\Omega}\big( \varPsi(\varphi) - \varPsi'\left(\varphi_{\infty}\right)\varphi\big)  \mathrm{d} x\right]\notag\\
&\quad+\|\nabla \mu\|^{2}=\mathcal{W}_2(t),\label{w2}
\end{align}
where
\begin{align}
\mathcal{W}_2(t) &=\big(\partial_t\varphi(t),\chi (\sigma(t)-\sigma_\infty)\big)
-\big(\bm{v}(t) \cdot \nabla \varphi(t),\mu(t)\big)\notag\\
&\quad  -\int_\Omega \alpha(\overline{\varphi}(t)-c_0)(\mu(t)-\mu_\infty)\, \mathrm{d}x.\label{w22}
\end{align}
On the other hand, multiplying \eqref{3f1.a} by $\mathcal{N} (\varphi-\overline{\varphi}-\varphi_\infty+c_0)$ and integrating over $\Omega$, we get
\begin{align}
 \frac{1}{2}\frac{\mathrm{d}}{\mathrm{d} t}\left\|\varphi-\overline{\varphi}-\varphi_\infty+c_0\right\|^{2}_{V_0'}
 +\left\|\nabla\left(\varphi-\varphi_{\infty}\right)\right\|^{2} =\mathcal{W}_3(t),
 \label{w3}
\end{align}
where
\begin{align}
\mathcal{W}_3(t)&
=-\big(\bm{v}(t) \cdot \nabla \varphi(t),\mathcal{N} (\varphi(t)-\overline{\varphi}(t)-\varphi_\infty+c_0)\big)\notag\\
&\quad -\big(\varPsi'(\varphi(t))-\varPsi'(\varphi_\infty), \varphi(t)-\overline{\varphi}(t)-\varphi_\infty+c_0\big)\notag\\
&\quad+\chi\big( \sigma(t)-\sigma_\infty,\varphi(t)-\varphi_\infty\big)
-\beta\|\nabla \mathcal{N}(\varphi(t)-\overline{\varphi}(t)-\varphi_\infty+c_0)\|^2.
\label{w32}
\end{align}
Multiplying \eqref{3f2.b} by $(\sigma-\sigma_\infty)-\chi(\varphi-\varphi_\infty)+ \mathcal{N}(\sigma-\sigma_\infty)$ and integrating over $\Omega$, we find that
\begin{align}
&\frac{1}{2}\frac{\mathrm{d}}{\mathrm{d} t}\big(\|\sigma-\sigma_{\infty}\|^{2} +
\|\sigma-\sigma_{\infty}\|^{2}_{V_0'}\big)
+\|\nabla(\sigma-\sigma_\infty)-\chi\nabla(\varphi-\varphi_\infty)\|^{2}
+ \|\sigma-\sigma_{\infty}\|^{2}\notag\\
&\quad =\mathcal{W}_4(t),\label{w4}
\end{align}
where
\begin{align}
\mathcal{W}_4(t)
&=\big(\partial_t\sigma(t),\chi(\varphi(t)-\varphi_\infty)\big) +\big(\bm{v}(t) \cdot \nabla \sigma(t),\chi (\varphi(t)-\varphi_\infty)  \big)\notag\\
&\quad -\big(\bm{v}(t) \cdot \nabla \sigma(t), (\sigma(t)-\sigma_\infty)\big) - \big(\bm{v}(t) \cdot \nabla \sigma(t),\mathcal{N}(\sigma(t)-\sigma_\infty)\big) \notag\\
&\quad +\chi\big(\varphi(t)-\varphi_\infty,\sigma(t)-\sigma_\infty\big)\notag\\ &=\big(\partial_t\sigma(t),\chi(\varphi(t)-\varphi_\infty)\big) +\chi\big(\bm{v}(t) \cdot \nabla \sigma(t), \varphi(t)  \big)\notag\\
&\quad   -\big(\bm{v}(t) \cdot \nabla \sigma(t),\mathcal{N}(\sigma(t)-\sigma_\infty)\big)  +\chi\big(\varphi(t)-\varphi_\infty,\sigma(t)-\sigma_\infty\big).
\label{w42}
\end{align}
 In the derivation of \eqref{w42} we have used the facts that $\nabla \cdot \bm{v}=0$ and $\sigma_\infty-\chi\varphi_\infty$ is a constant.

Combining \eqref{w1}, \eqref{w2}, \eqref{w3} and \eqref{w4} yields that
\begin{align}
\frac{\mathrm{d}}{\mathrm{d}t}\mathcal{Y}(t)
+\mathcal{H}(t)=\mathcal{W}(t),
\label{BELdd}
\end{align}
where
\begin{align}
\mathcal{Y}(t)&=\|\bm{v}\|^2+
\|\nabla(\varphi-\varphi_{\infty})\|^{2}
+(1+\beta)\|\nabla\mathcal{N}(\varphi-\overline{\varphi} -\varphi_\infty+c_0)\|^{2}\notag\\
&\quad
+2\int_{\Omega} \big( \varPsi(\varphi) - \varPsi(\varphi_{\infty})- \varPsi'(\varphi_{\infty})(\varphi-\varphi_{\infty}) \big) \mathrm{d} x\notag\\
&\quad
 + \|\sigma-\sigma_{\infty}\|^{2}
 + \|\sigma-\sigma_{\infty}\|^{2}_{V_0'}
  -2\chi(\sigma-\sigma_\infty,\varphi-\varphi_\infty),
\label{T1}
\end{align}
\begin{align}
\mathcal{H}(t)& =\int_\Omega 4\nu(\varphi(t))|D\bm{v}(t)|^2\, \mathrm{d}x + 2\|\nabla \mu(t)\|^2 +2 \|\nabla(\varphi-\varphi_{\infty})\|^{2}
\notag\\
&\quad
+ 2\|\nabla(\sigma-\sigma_\infty) -\chi\nabla(\varphi-\varphi_\infty)\|^{2}
+ 2\|\sigma-\sigma_{\infty}\|^{2},
\label{S1}
\end{align}
and
\begin{align}
\mathcal{W}(t)&= -2\alpha\int_\Omega (\overline{\varphi}(t)-c_0)(\mu(t)-\mu_\infty)\, \mathrm{d}x
 \notag\\
&\quad -2\big(\bm{v}(t) \cdot \nabla \varphi(t),\mathcal{N} (\varphi(t)-\overline{\varphi}(t)-\varphi_\infty+c_0)\big)\notag\\
&\quad
-2\big(\varPsi'(\varphi(t))-\varPsi'(\varphi_\infty), \varphi(t)-\overline{\varphi}(t)-\varphi_\infty+c_0\big) \notag\\
&\quad +4\chi\big( \sigma(t)-\sigma_\infty,\varphi(t)-\varphi_\infty\big)
-2\beta\|\nabla \mathcal{N}(\varphi(t)-\overline{\varphi}(t)-\varphi_\infty+c_0)\|^2
\notag\\
&\quad -2\big(\bm{v}(t) \cdot \nabla \sigma(t),\mathcal{N}(\sigma(t)-\sigma_\infty)\big).\notag
\end{align}
Thanks to the strict separation properties \eqref{glo-sep}, \eqref{sta-sep1} and the Newton--Leibniz formula, we easily obtain
\begin{align}
&\left|\int_{\Omega} \varPsi(\varphi(t))
- \varPsi(\varphi_{\infty})
- \varPsi'(\varphi_{\infty})(\varphi(t)-\varphi_{\infty})
\, \mathrm{d} x\right|\notag\\
&\quad =\left| \int_\Omega \int_0^1\int_0^1 \varPsi''(sz\varphi(t)+(1-sz)\varphi_\infty) z(\varphi(t)-\varphi_\infty)^2\mathrm{d}s \mathrm{d}z\mathrm{d}x\right|\notag\\
&\quad \le C\max_{r\in[-1+\frac{\delta}{2},1-\frac{\delta}{2}]}|\varPsi''(r)|\|\varphi(t)-\varphi_{\infty}\|^2,\label{var3}
\end{align}
and in a similar manner,
\begin{align}
   \left|\int_\Omega \mu(t)-\mu_{\infty}\,\mathrm{d}x\right|
& \leq  \left|\int_\Omega \varPsi'(\varphi(t))-\varPsi'(\varphi_\infty)\,\mathrm{d}x\right| +C\|\varphi(t)-\varphi_{\infty}\|+C\|\sigma(t)-\sigma_\infty\|\notag\\
&\leq C\|\varphi(t)-\varphi_{\infty}\|+C\|\sigma(t)-\sigma_\infty\|.
\label{mu}
\end{align}
Hence, a direct computation yields that
\begin{align}
\mathcal{W}(t)
&  \le \frac{\eta_*}{2}\|\nabla\bm{v}\|^2+ C|\overline{\varphi}(t)-c_0|^2+  C\|\varphi(t)-\varphi_{\infty}\|^2 +C\|\sigma(t)-\sigma_\infty\|^2.
\label{es-W}
\end{align}
Besides, we have
\begin{align}
\|\varphi(t)-\varphi_{\infty} \|^{2} & \leq  2\|\varphi(t)-\overline{\varphi}(t)-\varphi_{\infty} +c_0\|^2+2 \|\overline{\varphi}(t)-c_0\|^2 \notag\\
	&\leq C \|\varphi(t)-\overline{\varphi}(t)-\varphi_{\infty} +c_0\|_{(H^{1})'}\|\nabla(\varphi(t)-\varphi_{\infty})\| +C|\overline{\varphi}(t)-c_0|^2
\notag\\
	& \leq a_3\|\nabla(\varphi(t)-\varphi_{\infty})\|^{2} +C \|\varphi(t)-\overline{\varphi}(t)-\varphi_{\infty} +c_0\|_{V_0'}^2 +C|\overline{\varphi}(t)-c_0|^2\notag\\
& \leq   a_3\|\nabla(\varphi(t)-\varphi_{\infty})\|^{2} +C \|\varphi(t)-\varphi_{\infty} \|_{(H^1)'}^2,
\label{var0}
\end{align}
and
\begin{align}
\|\sigma(t)-\sigma_{\infty} \|^{2} &\leq \|\nabla(\sigma(t)-\sigma_{\infty})\| \|\sigma(t)-\sigma_{\infty}'\|_{V_0'}\notag\\
&\le a_3\|\nabla(\sigma(t)-\sigma_\infty )-\chi\nabla(\varphi(t)-\varphi_\infty )\|^{2} +a_3\chi^2\|\nabla(\varphi(t)-\varphi_{\infty})\|^2\notag\\
&\quad+C\|\sigma(t)-\sigma_{\infty}\|_{V_0'}^2,
\label{var1}
\end{align}
where $a_3>0$ is an arbitrary constant.
Then from \eqref{T1}--\eqref{var1}, we see that
\begin{align}
\mathcal{W}(t)
&  \le \frac{\eta_*}{2}\|\nabla\bm{v}\|^2+ \|\nabla(\varphi(t)-\varphi_{\infty})\|^{2}+ \|\nabla(\sigma(t)-\sigma_\infty )-\chi\nabla(\varphi(t)-\varphi_\infty )\|^{2} \notag\\
&\quad
 + C \|\varphi(t)-\varphi_{\infty} \|_{(H^1)'}^2 + C\|\sigma(t)-\sigma_\infty\|_{V_0'}^2,
\label{es-W2}
\end{align}
\begin{align}
\mathcal{Y}(t)
&\le 2\big(\|\bm{v}(t)\|^2+\|\nabla(\varphi(t)-\varphi_{\infty})\|^{2}+ \|\nabla(\sigma(t)-\sigma_\infty )-\chi\nabla(\varphi(t)-\varphi_\infty )\|^{2}\big)\notag\\
&\quad  + C \|\varphi(t)-\varphi_{\infty} \|_{(H^1)'}^2 + C\|\sigma(t)-\sigma_\infty\|_{V_0'}^2,
\label{t1}
\end{align}
and
\begin{align}
\mathcal{Y}(t)&\ge
 \frac12\big(\|\bm{v}(t)\|^2+  \|\nabla(\varphi(t)-\varphi_{\infty})\|^{2} + \|\sigma(t)-\sigma_{\infty}\|^{2}\big)
-C \|\varphi(t)-\varphi_{\infty}\|^{2}_{(H^1)'}.
\label{t2}
\end{align}

Combining the estimates \eqref{es-W2}, \eqref{t2}, we can find some small constant $a_4>0$ so that the inequality \eqref{BELdd} can be reduced to
\begin{align}
&\frac{\mathrm{d}}{\mathrm{d}t}\mathcal{Y}(t) + a_4 \mathcal{Y}(t)\le C \|\varphi(t)-\varphi_{\infty} \|_{(H^1)'}^2 + C\|\sigma(t)-\sigma_\infty\|_{V_0'}^2,\quad\forall\, t\in [0,+\infty).
\label{BEL7}
\end{align}
The above inequality together with  \eqref{ra} easily yields
\begin{align}
\mathcal{Y}(t)\leq C(1+t)^{-\frac{2\kappa_\infty}{1-2 \kappa_\infty}},\quad\forall\, t\in [0,+\infty).
\label{Ydecay}
\end{align}
Hence, from \eqref{t2}, \eqref{Ydecay} and \eqref{ra} we can conclude
\be
\|\bm{v}(t)\|+ \|\varphi(t)-\varphi_\infty \|_{H^1} + \|\sigma(t)-\sigma_\infty \|\leq C(1+t)^{-\frac{\kappa_\infty}{1-2 \kappa_\infty}},\quad\forall\, t\in [0,+\infty). \label{ra1}
\ee
The proof of Theorem \ref{3main1} is complete.
$\hfill\square$

\begin{remark}
Usually we only have $\kappa_\infty\in (0,1/2)$ (recall Remark \ref{rem:kappa}). If $\kappa_\infty=1/2$, then we can improve \eqref{ra1} to be an exponential decay (cf. \eqref{e-decay}).
\end{remark}

%%%%%%%%%%%%%%%%%%%%%%%%%%%%%%%%
\section*{Declarations}
\noindent \textbf{Acknowledgments}
The authors are grateful to the anonymous referee for insightful comments and valuable suggestions that improved the manuscript.
H. Wu is a member of the Key Laboratory of Mathematics for Nonlinear Sciences (Fudan University), Ministry of Education of China. \medskip
\\
\noindent \textbf{Funding}
This work was partially supported by NNSFC 12071084 and the Shanghai Center for Mathematical Sciences.  \medskip
\\
\noindent \textbf{Conflict of interest}
The authors declare that they have no conflict of interests.  \medskip
\\
\noindent \textbf{Data availability statement}
Data sharing not applicable to this article as no datasets were generated or analysed during the current study.

%%%%%%%%%%%%%%%%%%%%%%%%%%%%%%%%


\begin{thebibliography}{99}
\itemsep=-3pt
%%%%%%%%%%%%%%%%%%%%%%%%%%%%%%%%%
\bibitem{A2009}
H. Abels,
On a diffuse interface model for two-phase flows of viscous, incompressible fluids with matched densities,
Arch. Ration. Mech. Anal., \textbf{194} (2009), 463--506.

%\bibitem{A2013}
%H. Abels, D. Depner and H. Garcke,
%Existence of weak solutions for a diffuse interface model for two-phase flows of incompressible fluids with different densities,
%J. Math. Fluid Mech., \textbf{15} (2013), 453--480.

\bibitem{A2012}
H. Abels, H. Garcke and G. Gr\"{u}n,
Thermodynamically consistent, frame indifferent diffuse interface models for incompressible two-phase flows with different densities, Math. Models Methods Appl. Sci., \textbf{22} (2012), 1150013.

\bibitem{A2007}
H. Abels and M. Wilke,
Convergence to equilibrium for the Cahn--Hilliard equation with a logarithmic free energy,
Nonlinear Anal., \textbf{67} (2007), 3176--3193.

\bibitem{AACGV}
A. Agosti, P. F. Antonietti, P. Ciarletta, M. Grasselli and A. Verani,
A Cahn--Hilliard-type equation with application to tumor growth dynamics,
Math. Methods Appl. Sci., \textbf{40} (2017), 7598--7626.

\bibitem{AMW}
D. M. Anderson, G. B. McFadden and A.A. Wheeler, Diffuse-interface methods in fluid mechanics,
Annu. Rev. Fluid Mech., \textbf{30} (1998), 139--165.

%\bibitem{B}
%I. Ben Hassen, Decay estimates to equilibrium for some asymptotically autonomous semilinear evolution equations,
%Asymptot. Anal., \textbf{69} (2010), 31--44.

\bibitem{BGM}
S. Bosia, M. Grasselli and A. Miranville,
On the longtime behavior of a 2D hydrodynamic model for chemically reacting binary fluid mixtures,
Math. Methods Appl. Sci., \textbf{37} (2014), 726--743.

\bibitem{Boyer}
F. Boyer,
Mathematical study of multi-phase flow under shear through order parameter formulation,
Asymptot. Anal., \textbf{20} (1999), 175--212.

\bibitem{Bre}
H. Brezis,
\emph{Functional Analysis, Sobolev Spaces and Partial Differential Equations}, Springer-Verlag, New York, 2010.

\bibitem{CH}
J. Cahn and J. Hilliard,
Free energy of a nonuniform system I. Interfacial free energy,
J. Chem. Phys., \textbf{28} (1958), 258--267.

%\bibitem{CL}
%W. Chen and C. Li,
%\emph{Methods on Nonlinear Elliptic Equations}.
%American Institute of Mathematical Sciences, 2010.

\bibitem{CMZ}
L. Cherfils, A. Miranville and S. Zelik,
The Cahn--Hilliard equation with logarithmic potentials,
Milan J. Math., \textbf{79} (2011), 561--596.

\bibitem{C2003}
R. Chill, On the \L ojasiewicz--Simon gradient inequality, J. Funct. Anal., \textbf{201} (2003), 572--601.


\bibitem{Choski}
R. Choksi, M. Maras  and J. F. Williams,
2D phase diagram for minimizers of a Cahn--Hilliard functional with long-range interactions,
SIAM J. Appl. Dyn. Syst., \textbf{10} (2011), 1344--1362.

%\bibitem{C}
%M. Chipot,
%\emph{Elliptic Equations: An Introductory Course},
%Basel Switzerland: Birkh$\rm\ddot{a}$user, 2009.

\bibitem{CGRS22}
P. Colli, G. Gilardi, E. Rocca and J. Sprekels,  Well-posedness and optimal control for a Cahn--Hilliard--Oono system with control in the mass term, Discrete Contin. Dyn. Syst. Ser. S, \textbf{15} (2022), 2135--2172.

\bibitem{CG}
M. Conti and A. Giorgini,
Well-posedness for the Brinkman--Cahn--Hilliard system with unmatched viscosities,
J. Differential Equations, \textbf{268} (2020), 6350--6384.

\bibitem{DFW}
A. Diegel, X. Feng and S. M. Wise,
Analysis of a mixed finite element method for a Cahn--Hilliard--Darcy--Stokes system,
SIAM J. Numer. Anal. \textbf{53} (2015), 127--152.

\bibitem{EG19jde}
M. Ebenbeck and H. Garcke,
Analysis of a Cahn--Hilliard--Brinkman model for tumor growth with chemotaxis,
J. Differential Equations, \textbf{266} (2019), 5998--6036.

\bibitem{Fa15}
H. Fakih,
A Cahn--Hilliard equation with a proliferation term for biological and chemical applications,
Asymptot. Anal., \textbf{94} (2015), 71--104.

\bibitem{GG2006}
H. Gajewski and A. Griepentrog,
A descent method for the free energy of multi-component systems,
Discrete Contin. Dynam. Syst., \textbf{15} (2006), 505--528.

\bibitem{GG2010}
C. Gal and  M. Grasselli,
Asymptotic behavior of a Cahn--Hilliard--Navier--Stokes system in 2D,
Ann. Inst. H. Poincar\'{e} Anal. Non Lin\'{e}aire, \textbf{27} (2010), 401--436.

\bibitem{G}
G. Galdi,
\emph{An Introduction to the Mathematical Theory of the Navier--Stokes Equations: Steady State Problems},
Second edition, Springer-Verlag, New York, 2011.

\bibitem{GL17}
H. Garcke and K.-F. Lam,
Analysis of a Cahn--Hilliard system with non-zero Dirichlet conditions modeling tumor growth with chemotaxis,
Discrete Contin. Dyn. Syst., \textbf{37} (2017), 4277--4308.

\bibitem{GL17e}
H. Garcke and K.-F. Lam,
Well-posedness of a Cahn--Hilliard system modelling tumour growth with chemotaxis and active transport,
European J. Appl. Math., \textbf{28} (2017), 284--316.

\bibitem{GLSS}
H. Garcke, K.-F. Lam, E. Sitka and V. Styles,
A Cahn--Hilliard--Darcy model for tumour growth with chemotaxis and active transport,
Math. Models Methods Appl. Sci., \textbf{26} (2016), 1095--1148.
	
\bibitem{Gio2021}
A. Giorgini,
Well-posedness of the two-dimensional Abels--Garcke--Gr\"{u}n model for two-phase flows with unmatched densities,
Calc. Var. Partial Differential Equations, \textbf{60} (2021), Paper No. 100, 40 pp.

\bibitem{Gio2022}
A. Giorgini,
Existence and stability of strong solutions to the Abels--Garcke--Gr\"{u}n model in three dimensions,
Interfaces Free Bound., \textbf{24} (2022), 565--608.


\bibitem{GGM2017}
A. Giorgini, M. Grasselli and A. Miranville,
The Cahn--Hilliard--Oono equation with singular potential, Math. Models Meth. Appl. Sci., \textbf{27} (2017), 2485--2510.

\bibitem{GGW}
A. Giorgini, M. Grasselli and H. Wu,
The Cahn--Hilliard--Hele--Shaw system with singular potential,
Ann. Inst. Henri Poincar\'{e} Anal. Non Lin\'{e}aire, \textbf{35} (2018), 1079--1118.

\bibitem{GLRS}
A. Giorgini, K.-F. Lam, R. Rocca and G. Schimperna,
On the existence of strong solutions to the Cahn--Hilliard--Darcy system with mass source,
SIAM J. Math. Anal., \textbf{54} (2022), 737--767.

\bibitem{GMT}
A. Giorgini, A. Miranville and R. Temam,
Uniqueness and regularity for the Navier--Stokes--Cahn--Hilliard system,
SIAM J. Math. Anal., \textbf{51} (2019), 2535--2574.


\bibitem{Glo95}
S. C. Glozter, E. A. De Marzio  and M. Muthukumar,
Reaction-controlled morphology of phase-separating mixtures,
Phys. Rev. Lett., \textbf{74} (1995), 2034--2037.

\bibitem{Gur}
M. E. Gurtin, D. Polignone and J. Vi\~{n}als,
Two-phase binary fluids and immiscible fluids described by an order parameter,
Math. Models Methods Appl. Sci., \textbf{6} (1996), 815--831.

\bibitem{HJ2001}
A. Haraux and M. A. Jendoubi,
Decay estimates to equilibrium for some evolution equations with an analytic nonlinearity,
Asymptot. Anal., \textbf{26} (2001), 21--36.

\bibitem{H}
J.-N. He,
Global weak solutions to a Navier--Stokes--Cahn--Hilliard system with chemotaxis and singular potential,
Nonlinearity, \textbf{34} (2021), 2155--2190.

\bibitem{H2}
J.-N. He,
On the viscous Cahn--Hilliard--Oono system with chemotaxis and singular potential,
Math. Methods Appl. Sci., \textbf{45} (2022), 3732--3763.


\bibitem{H1}
J.-N. He and H. Wu,
Global well-posedness of a Navier--Stokes--Cahn--Hilliard system with chemotaxis and singular potential in 2D,
J. Differential Equations,  \textbf{297} (2021), 47--80.


\bibitem{HH}
P. Hohenberg and B. Halperin,
Theory of dynamic critical phenomena,
Rev. Mod. Phys., \textbf{49} (1977), 435--479.

\bibitem{H03}
Y. Huo, X. Jiang, H. Zhang and Y. Yang,
Hydrodynamic effects on phase separation of binary mixtures with reversible chemical reaction,
J. Chem. Phys., \textbf{118} (2003), 9830--9837.


\bibitem{JWZ}
J. Jiang, H. Wu and S.-M. Zheng,
Well-posedness and long-time behavior of a non-autonomous Cahn--Hilliard--Darcy system with mass source modeling tumor growth,
J. Differetial Equations, \textbf{259} (2015), 3032--3077.

\bibitem{KS22}
P. Knopf and A. Signori,
Existence of weak solutions to multiphase Cahn--Hilliard--Darcy and Cahn--Hilliard--Brinkman models for stratified tumor growth with chemotaxis and general source terms,
Comm. Partial Differential Equations, \textbf{47} (2022), 233--278.

\bibitem{Lam22}
K.-F. Lam,
Global and exponential attractors for a Cahn--Hilliard equation with logarithmic potentials and mass source,
J. Differential Equations, \textbf{312} (2022), 237--275.

\bibitem{LW}
K.-F. Lam and H. Wu,
Thermodynamically consistent Navier--Stokes--Cahn--Hilliard models with mass transfer and chemotaxis,
European J. Appl. Math., \textbf{29} (2018), 595--644.

\bibitem{LL95}
F.-H. Lin and C. Liu,
Nonparabolic dissipative systems modeling the flow of liquid crystals,
Comm. Pure Appl. Math., \textbf{48} (1995), 501--537.

\bibitem{LS}
C. Liu and J. Shen,
A phase field model for the mixture of two incompressible fluids and its approximation by a Fourier-spectral method,
Phys. D, \textbf{179} (2003), 211--228.

\bibitem{LT98}
J. Lowengrub and L. Truskinovsky,
Quasi-incompressible Cahn--Hilliard fluids and topological transitions,
R. Soc. Lond. Proc. Ser. A: Math. Phys. Eng. Sci., \textbf{454} (1998), 2617--2654.

\bibitem{Mi11}
A. Miranville,
Asymptotic behavior of the Cahn--Hilliard--Oono equation,
J. Appl. Anal. Comp., \textbf{1} (2011), 523--536.

\bibitem{Mi19}
A. Miranville,
\emph{The Cahn--Hilliard Equation: Recent Advances and Applications},
CBMS-NSF Regional Conference Series in Applied Mathematics, \textbf{95}, SIAM, Philadelphia, 2019.


\bibitem{MRS}
A. Miranville, E. Rocca and G. Schimperna,
On the long time behavior of a tumor growth model,
J. Differential Equations, \textbf{267} (2019), 2616--2642.

\bibitem{MT}
A. Miranville and R. Temam,
On the Cahn--Hilliard--Oono--Navier--Stokes equations with singular potentials,
Appl. Anal., \textbf{95} (2016), 2609--2624.


\bibitem{MZ04}
A. Miranville and S. Zelik,
Robust exponential attractors for Cahn--Hilliard type equations with singular potentials,
Math. Methods Appl. Sci., \textbf{27} (2004), 545--582.

%\bibitem{OHP}
%J. Oden, A. Hawkins and S. Prudhomme,
%General diffuse-interface theories and an approach to predictive tumor growth modeling,
%Math. Models Methods Appl. Sci., \textbf{20} (2010), 477--517.

\bibitem{OK}
T. Ohta and K. Kawasaki,
Equilibrium morphology of block copolymer melts, Macromolecules,
\textbf{19} (1986), 2621--2632.


\bibitem{OP87}
Y. Oono and S. Puri,
Computationally efficient modeling of ordering of quenched
phases,
Phys. Rev. Lett., \textbf{58} (1987), 836--839.

\bibitem{PP21}
B. Perthame and A. Poulain,
Relaxation of the Cahn--Hilliard equation with singular single-well potential and degenerate mobility,
European J. Appl. Math., \textbf{32} (2021), 89--112.

\bibitem{Rupp}
F. Rupp,
On the {\L}ojasiewicz--Simon gradient inequality on submanifolds,
J. Funct. Anal., \textbf{279} (2020), 108708, 33 pp.

\bibitem{RH99}
P. Rybka and K.-H. Hoffmann,
Convergence of solutions to Cahn--Hilliard equation,
Commun. Partial Differential Equations, \textbf{24} (1999), 1055--1077.


\bibitem{simon}
J. Simon,
Compact sets in the space $L^p(0, T; B)$,
Ann. Mat. Pura Appl. (4), \textbf{146} (1987), 65--96.

\bibitem{LS83}
L. Simon,
Asymptotics for a class of nonlinear evolution equation with applications to geometric problems,
Ann. Math., \textbf{118} (1983), 525--571.

\bibitem{Sitka}
E. Sitka,
\emph{Modeling Tumor Growth: A Mixture Model with Mass Exchange}, Master's thesis, Universit\"{a}t Regensburg, 2013.

\bibitem{S}
H. Sohr,
\emph{The Navier--Stokes Equations: An Elementary Functional Analytic Approach},
Birkh\"{a}user Advanced Texts, Springer, Basel, 2012.

\bibitem{T}
R. Temam,
\emph{Infinite Dimensional Dynamical Systems in Mechanics and Physics},
Second edition, Applied Mathematical Sciences, Vol. 68, Springer-Verlag, New York, 1997.

\bibitem{W07}
H. Wu,
Convergence to equilibrium for a Cahn--Hilliard model with the Wentzell boundary condition,
Asymptot. Anal., \textbf{54} (2007), 71--92.

\bibitem{ZWH}
L.-Y. Zhao, H. Wu and H.-Y. Huang,
Convergence to equilibrium for a phase-field model for the mixture of two incompressible fluids,
Commun. Math. Sci., \textbf{7} (2009), 939--962.

\bibitem{Zhou20}
D.-X. Zhou,
\emph{Global Well-posedness and Long-time Behavior of the Three Dimensional Cahn--Hilliard--Navier--Stokes System
with Singular Potential},
Master's thesis (in Chinese), Fudan University, 2020.

%%%%%%%%%%%%%%%%%%%%%%%%%%%%%%
\end{thebibliography}
\end{document}